\documentclass{amsart}
\usepackage{amsfonts}
\usepackage{amsmath}
\usepackage{amssymb}
\usepackage{enumerate}
\newtheorem{theorem}{Theorem}[section]

\newtheorem{lemma}[theorem]{Lemma}

\newtheorem{proposition}[theorem]{Proposition}

\newtheorem{corollary}[theorem]{Corollary}

\theoremstyle{definition}

\newtheorem{definition}[theorem]{Definition}

\newtheorem{claim}[theorem]{Claim}

\newtheorem{example}[theorem]{Example}

\newtheorem{notation}[theorem]{Notation}

\newtheorem *{Theorem A}{Theorem A}
\newtheorem *{Corollary B}{Corollary B}
\newtheorem *{Corollary D}{Corollary D}
\newtheorem *{Theorem B}{Theorem B}
\newtheorem *{Theorem B1}{Theorem}
\newtheorem *{Projective Schur's Lemma}{Projective Schur's Lemma}
\newtheorem *{Graded Artin Wedderburn Theorem}{Graded Artin Wedderburn Theorem}
\newtheorem *{Higgs' Conjecture}{Higgs' Conjecture}
\newtheorem *{Theorem C}{Theorem C}
\newtheorem *{Theorem D}{Theorem D}
\newtheorem *{BSZ}{Generalized BSZ Theorem}
\newtheorem *{Theorem E}{Theorem E}
\newtheorem *{Theorem F}{Theorem F}
\newtheorem *{Theorem G}{Theorem G}
\newtheorem *{Theorem H}{Theorem H}
\newtheorem *{Corollary I}{Corollary I}
\newtheorem *{Theorem J}{Theorem J}
\newtheorem *{Corollary K}{Corollary K}
\newtheorem *{Theorem L}{Theorem L}
\newtheorem *{Question1}{Question 1}
\newtheorem *{Question2}{Question 2}
\newtheorem *{Problem 1}{Problem 1}
\newtheorem *{Problem 2}{Problem 2}
\newtheorem *{Problem 3}{Problem 3}

\newcommand{\bigslant}[2]{{\raisebox{.1em}{$#1$}\left/\raisebox{-.1em}{$#2$}\right.}}
\newcommand{\Id}{\text {Id}}
\newcommand{\IM}{\text {Im}}
\newcommand{\md}{\text {mod }}
\newcommand{\Hom}{\text {Hom}}
\newcommand{\Aut}{\text {Aut}}
\newcommand{\End}{\text {End}}
\newcommand{\C}{{\mathbb C}}
\newcommand{\F}{{\mathbb F}}

\newcommand{\res}{\text{res}}
\newcommand{\e}{\text{eq}}

\newcommand{\Gr}{\text{Gr}}
\newcommand{\Supp}{\text{Supp}}
\newcommand{\dimF}{\dim_{\F}}
\newcommand{\SpF}{\text{Span}_{\F}}
\numberwithin{equation}{section}

\begin{document}
\title{Mackey's obstruction map for discrete graded algebras}
\author{Yuval Ginosar}

\address{Department of Mathematics, University of Haifa, Haifa 3498838, Israel}
\email{ginosar@math.haifa.ac.il}
\begin{abstract}
G.W. Mackey's celebrated obstruction theory for projective representations of locally compact groups was remarkably generalized by J. M. G. Fell and R. S. Doran to
the wide area of saturated Banach *-algebraic bundles.

Analogous obstruction is suggested here for discrete group graded algebras which are not necessarily saturated, i.e. strongly graded in the discrete context.
The discrete obstruction is a map assigning a certain second cohomology class to every equivariance class of absolutely simple graded modules.

The set of equivariance classes of such modules is equipped with an appropriate multiplication, namely a graded product, such that the obstruction map is a homomorphism of abelian monoids.
Graded products, essentially arising as pull-backs of bundles, admit many nice properties, including a way to twist graded algebras and their graded modules.

The obstruction class turns out to determine the fine part that appears in the Bahturin-Zaicev-Sehgal decomposition, i.e. the
graded Artin-Wedderburn theorem for graded simple algebras which are graded Artinian,
in case where the base algebras (i.e. the unit fiber algebras) are finite-dimensional over algebraically closed fields.
\end{abstract}

\date{\today}

\maketitle

\tableofcontents
\bibliographystyle{abbrv}
\section{Introduction}
A graded version of the Artin-Wedderburn Theorem \cite[Theorem I.5.8]{NVO82},\cite[Corollary 4.6.7]{NVO04},\cite[Theorem 2.6]{EK13} says that
a group-graded algebra, which satisfies the descending chain condition on graded (left) ideals, is graded-simple if and only if it is
graded-isomorphic to a graded endomorphism algebra of a graded (left) module which is finite-dimensional over some graded division algebra.
This result has a more explicit description, the Bahturin-Zaicev-Sehgal decomposition \cite[Theorem 3]{MR2488221},
when the group-graded algebra is assumed to be finite-dimensional over an algebraically closed field $\F$, and as can easily be verified
(see Lemma \ref{grdivtwisted} herein) also when the finite-dimension assumption is imposed only on the base algebra while keeping the algebraic closure demand on $\F$.
In this case, the corresponding graded division algebra is a twisted group $\F$-algebra of a subgroup of the grading group, which is determined up to conjugation.
How is this conjugacy class determined, and where does the twist pop-up from? How is the grading on the module over this twisted group algebra determined?

These questions are apparently related to a problem of deciding whether there exist modules of a $G$-graded algebra
\begin{equation}\label{eq:algebragrading}
\mathcal{G}:A_G=\bigoplus _{g\in G} A_g
\end{equation}
that restrict to a given module $M$ of the base algebra $A_e$, or more generally, of finding the $A_G$-modules which {\it lie above} (see the beginning of \S\ref{sb}) this $A_e$-module $M$.
A solution to the restriction problem was given by E.C. Dade \cite[Theorem 2.8]{Dade81} when \eqref{eq:algebragrading} is strongly graded.
A more concrete criterion is well-known in the following special setting. Let $c\in Z^2(\Gamma,\F^*)$ be a 2-cocycle of a group $\Gamma$ acting trivially on a field $\F$.
Then the corresponding twisted group algebra admits a natural $\Gamma$-graded structure $\F^c\Gamma=\oplus_{\gamma\in\Gamma}(\F^c\Gamma)_{\gamma}.$
Furthermore, for every normal subgroup $N\lhd \Gamma$ this twisted group algebra admits a {\it quotient grading} (see \S\ref{grmorph}) by $G:=\Gamma/N$ as follows
\begin{equation}\label{quotgrdef}
\F^c\Gamma=\bigoplus_{\gamma N\in G}(\F^c\Gamma)_{\gamma N}, \text{ where } (\F^c\Gamma)_{\gamma N}=\bigoplus_{\gamma'\in\gamma N}(\F^c\Gamma)_{\gamma'}.
\end{equation}
The base algebra of this $G$-grading is the twisted subgroup algebra $\F^{c_N}N$, where the 2-cocycle $c_N\in Z^2(N,\F^*)$ is the restriction of $c$ from $\Gamma$ to $N$.

This setup leads us to the seminal work of G.W. Mackey \cite{M58}.
As part of the so-called {\it Mackey's machine} for locally compact groups,
his theory in the discrete case (translated from the language of projective representations to that of modules over twisted group algebras)
classifies the simple $\F^{c}\Gamma$-modules lying above a given $\F^{c_N}N$-module, which is $\F$-absolutely simple and $\Gamma$-invariant.
In particular, it determines when this module is restricted from an $\F^c\Gamma$-module by means of the {\it Mackey obstruction}, that lies in the second cohomology of $G:=\Gamma/N$.
Mackey's normal subgroup analysis was pushed further thirty years later by J. M. G. Fell and R. S. Doran to much broader structures,
namely to \textit{Banach *-algebraic bundles} (also known as \textit{Fell bundles}),
over locally compact groups $G$ \cite{FellDoran1,FellDoran2}. The Mackey obstruction in this generality is given
in terms of a certain central extension of $G$ and is shown to exist for saturated bundles in many cases, in particular whenever $G$ is discrete \cite[Proposition XII.4.18]{FellDoran2}.
For non-saturated bundles there is no satisfactory definition of the obstruction, or even of the stability subgroup. Sometimes, the {\it saturated part} of a $G$-bundle may remedy
the situation \cite[\S XII.5.12]{FellDoran2}, however even for discrete groups this part may provide poor data about the entire bundle.

We would like to suggest a natural definition of Mackey's obstruction for any discrete group graded algebra.
This obstruction for discrete group graded complex algebras essentially coincides with the Mackey obstruction for Fell bundles which are saturated,
i.e. whose $L^1$-spaces boil down in the discrete case to ``strongly graded algebras".

In order to formulate an obstruction theorem for extension of absolutely simple modules to general discrete group graded algebras we need two notions,
the precise definitions can be found in the sequel.
Let \eqref{eq:algebragrading} be a $G$-graded $\F$-algebra.
Such a $G$-graded $\F$-algebra can be twisted by any 2-cocycle $\alpha\in Z^2(G,\F^*)$ giving rise to another $G$-graded $\F$-algebra
$$\alpha(\mathcal{G}):\alpha(A_G)=\F^{\alpha}G\otimes^G_{\F}A_G$$
that admits the same base algebra $A_e$. This is a special case
of a more general concept of {\it graded products} of $G$-graded $\F$-algebras (see \eqref{Gprod}).
Let $\text{Irr}_{\F}(A_e)$ denote the set of isomorphism types of the ${\F}$-absolutely simple left $A_e$-modules (or the $\F$-absolutely simple representations, see \S\ref{sb}),
and let $\text{Irr}_{\F}(A_e)^{\mathcal{G}}\subseteq\text{Irr}_{\F}(A_e)$ be its subset of $\mathcal{G}$-invariant classes (see \eqref{inirr}).
As customary, let $H^2(G,\F^*)$ be the second cohomology group of $G$, acting trivially on the group $\F^*$ of $\F$-units.
We have
\begin{Theorem A}\label{ThA}
Let \eqref{eq:algebragrading} be a $G$-graded algebra over any field $\F$. Then there is a well-defined map
\begin{equation}\label{omegaintro}
\omega_{\mathcal{G}}:\text{Irr}_{\F}(A_e)^{\mathcal{G}}\to H^2(G,\F^*),
\end{equation}
such that a $\mathcal{G}$-invariant  ${\F}$-absolutely simple $A_e$-module $M$ is restricted from the twisted graded algebra $\alpha(A_G)$
if (and only if for strongly graded algebras) the 2-cocycle $\alpha^{-1}\in Z^2(G,\F^*)$ belongs to the cohomology class
$\omega_{\mathcal{G}}([M])\in H^2(G,\F^*)$.
In particular, $M$ is restricted from the original graded algebra \eqref{eq:algebragrading} if
(and only if for strongly graded algebras) $\omega_{\mathcal{G}}([M])=1$.
\end{Theorem A}
Theorem A, proven in \S\ref{pfA1},
follows from a criterion for extending any $A_e$-module, not necessarily simple (see Theorem \ref{gendade}). This criterion generalizes Dade's Theorem,
originally formulated for strongly graded algebras.

A nice property, given in Theorem \ref{obsinvas}, says that if $[M]\in\text{Irr}_{\F}(A_e)^{\mathcal{G}}$, then $M$
is restricted from at least one of the twisted graded algebras $\alpha(A_G)$, where $\alpha\in Z^2(G,\F^*)$.
Corollary \ref{cyclic}, another consequence of Theorem A, is an extension criterion in case the second cohomology of the grading group is trivial,
it was known in the strongly graded case using Dade's Theorem.

Note the particular setup, where the $G$-graded algebra in Theorem A is a twisted group algebra $\F^c\Gamma$, graded by the quotient group $G=\Gamma/N$ as in \eqref{quotgrdef}.
In this case the twisting $\alpha(\F^c\Gamma)$ by $\alpha\in Z^2(G,\F^*)$ is isomorphic to the twisted group algebra $\F^{c\cdot\inf^G_{\Gamma}(\alpha)}\Gamma$ as $G$-graded algebras
(Corollary \ref{formackey}), where $\inf^G_{\Gamma}(\alpha)$ is the inflation of $\alpha$ from $G$ to $\Gamma$.
Consequently, Theorem A is Mackey's Obstruction Theorem for projective representations of discrete groups \cite[Theorem 8.2]{M58}.
The successive result in Mackey's paper \cite[Theorem 8.3]{M58}, that is ``Mackey's bijection", was generalized by Fell and Doran to saturated bundles.
The Mackey-Fell-Doran correspondence is formulated
in the end of \S\ref{pfA1} for discrete groups.

Does the map \eqref{omegaintro} enjoy a multiplicative property?
Well, this question does not make too much sense, for $\text{Irr}_{\F}(A_e)^{\mathcal{G}}$ has no
multiplicative structure.
However,
$\text{Irr}_{\F}(A_e)^{\mathcal{G}}$ is mapped into a certain abelian monoid, such that
the obstruction map $\omega_{\mathcal{G}}$ is essentially a restriction from this monoid to $\text{Irr}_{\F}(A_e)^{\mathcal{G}}$ of a homomorphism $\omega_{G,\F}$ whose values are in $H^2(G,\F^*)$.
More precisely, to any $A_e$-module $M$ one associates a well-defined graded $A_G$-module ${A_G}\overline{\otimes}_{A_e}M$
(see \eqref{assogrd}), which is a certain quotient of the graded $A_G$-module $A_G\otimes_{A_e}M$. We have
\begin{Theorem B}
With the above notation there exist
\begin{enumerate}
\item an abelian monoid $\text{Irr}^{\text{inv}}_{G,\F},$ consisting of graded equivariance classes of
invariant and absolutely graded-simple left graded modules over $G$-graded $\F$-algebras of bounded infinite cardinality
(Notation \ref{MIrr}),
\item a map $\iota_{\mathcal{G}}:\text{Irr}_{\F}(A_e)^{\mathcal{G}}\rightarrow\text{Irr}^{\text{inv}}_{G,\F}$,
sending $[M]\in\text{Irr}_{\F}(A_e)^{\mathcal{G}}$ to the graded equivariance class of
${A_G}\overline{\otimes}_{A_e}M$,
as well as
\item a homomorphism of abelian monoids (see \eqref{omega})
\begin{equation}\label{WGF}\omega_{G,\F}:\text{Irr}^{\text{inv}}_{G,\F}\to H^2(G,\F^*),\end{equation}
\end{enumerate}
such that $$\omega_{\mathcal{G}}=\omega_{G,\F}\circ\iota_{\mathcal{G}}.$$
\end{Theorem B}
Theorem B is proven in \S\ref{sb11}. The multiplicative property it manifests makes the obstruction more tractable to work with.
For example, it implies that the map $\omega_{\mathcal{G}}$ respects twistings in the following sense.
Noting that a class $[M]\in \text{Irr}_{\F}(A_e)$ is invariant with respect to a graded algebra \eqref{eq:algebragrading}
iff it is invariant with respect to every twisting $\alpha(A_G)$ of this algebra, we have (see \eqref{aseq1})
\begin{equation}\label{tract}
\omega_{\alpha(\mathcal{G})}([M])=[\alpha]\cdot\omega_{\mathcal{G}}([M]), \ \ \forall \alpha\in Z^2(G,\F^*), [M]\in\text{Irr}_{\F}(A_e)^{\mathcal{G}}.
\end{equation}

The map \eqref{omegaintro} is key in order to answer the questions posed above about the graded Artin-Wedderburn Theorem over algebraically closed fields.
An $A_e$-module $M$ is in general not necessarily $\mathcal{G}$-invariant.
Still, there exists a unique maximal subgroup of the grading group, namely the {\it inertia} $\mathcal{I}=\mathcal{I}_{\mathcal{G}}(M)<G$
(see \eqref{inertiasuspension}), such that $M$ is invariant under the $\mathcal{I}$-graded subalgebra $A_{\mathcal{I}}\subseteq A_G$.
If, additionally, $[M]\in\text{Irr}_{\F}(A_e)$ then
the obstruction cohomology class $\omega_{\mathcal{I}}([M])\in H^2(\mathcal{I},\F^*)$ is well-defined.
Next, the endomorphism algebra End$_{\F}({A_G}\overline{\otimes}_{A_e}M)$ of the associated graded module ${A_G}\overline{\otimes}_{A_e}M$ admits a natural $G$-grading (see \S \ref{Ghgm}).
These graded endomorphisms, on one hand, and the obstruction class, on the other hand, yield the following formulation of the graded Artin-Wedderburn Theorem,
giving an answer to the guiding questions above.
\begin{Theorem C}
Let \eqref{eq:algebragrading} be a $G$-graded algebra over an algebraically closed field $\F$.
Suppose that $A_G$ satisfies the descending chain condition on graded left ideals, and that  $A_e$ is finite-dimensional over $\F$.
Then for any simple $A_e$-module $M$, the simply $G$-graded endomorphism algebra $\text{End}^{\text{l}(G)}_{D_G}({A_G}\overline{\otimes}_{A_e}M)$ is a graded image of $A_G$,
where \begin{enumerate}
\item $D_G\stackrel{G}{\cong}\F^{\omega}\mathcal{I}$ (i.e. these $G$-graded division algebras are graded isomorphic) for any 2-cocycle $\omega$ in the obstruction cohomology class
$\omega_{\mathcal{I}}([M])\in H^2(\mathcal{I},\F^*)$ of the inertia $\mathcal{I}=\mathcal{I}_{\mathcal{G}}(M)<G$, and
\item ${A_G}\overline{\otimes}_{A_e}M$ is finite-dimensional over the graded division algebra $D_G$.
\end{enumerate}
\end{Theorem C}
The base algebra of a left graded Artinian algebra is itself left Artinian (see Lemma \ref{redext}).
Then a simple left $A_e$-module $M$ in Theorem C can be chosen to be any minimal left $A_e$-ideal.
Theorem C is proven in \S\ref{pfB}.

The kernel of a graded homomorphism is a graded ideal. Hence, if in addition to the assumptions in Theorem C, the algebra \eqref{eq:algebragrading} is also simply graded,
then a graded image of $A_G$ is graded isomorphic to $A_G$.
Theorem C immediately yields
\begin{Corollary D}
With the notation of Theorem C, if $A_G$ is further assumed to be simply $G$-graded, then
\begin{equation}\label{stackrel}
A_G\stackrel{G}{\cong}\text{End}^{\text{l}(G)}_{\F^{\omega}\mathcal{I}}({A_G}\overline{\otimes}_{A_e}M)\stackrel{G}{\cong}
\text{M}_{n}(\F)\otimes_{\F}\F^{\omega}\mathcal{I},
\end{equation}
where $M$ is any minimal left $A_e$-ideal, and $n=\dim_{\F^{\omega}\mathcal{I}}({A_G}\overline{\otimes}_{A_e}M)$.
\end{Corollary D}
The matrix algebra $\text{M}_{n}(\F)$ in \eqref{stackrel} is the \textit{elementary} part \cite{das2008} of the grading \eqref{eq:algebragrading}, whereas $F^{\omega}\mathcal{I}$
is its \textit{fine} part \cite{MR2488221}.

Different choices of simple $A_e$-modules $M$ in Corollary D yield graded-isomorphic (up to suspension) associated graded $A_G$-modules ${A_G}\overline{\otimes}_{A_e}M$,
conjugate inertia subgroups $\mathcal{I}$,
as well as conjugate obstruction cohomology classes $\omega_{{\mathcal{I}}}([M])$ (see Theorem \ref{uptoconj}).
Nevertheless, all such choices yield \eqref{stackrel} up to graded isomorphism.

As already mentioned, a substantial theme throughout is the graded product $A_G\otimes_{\F}^GA'_G$ of $G$-graded $\F$-algebras $A_G$ and $A_G'$.
This graded product, which can be regarded as the $L^1$-space of a certain pull-back of the corresponding $G$-bundles, gives rise to a monoidal structure Gr$(G,\F)$
on the graded isomorphism classes of $G$-graded $\F$-algebras (Theorem \ref{thmgrprod}).
The abelian monoid Gr$(G,\F)$ is endowed with an $H^2(G,\F^*)$-action via graded products with classes of twisted group algebras (see Equation \eqref{Z2actgr}),
as well as with an action of Aut$(G)$, and in particular of $G$ itself as inner automorphisms (\S\ref{autsec}).

An interesting family of examples is presented in \S\ref{secpb}.
Theorem \ref{pullback} therein says that the graded product of twisted group $\F$-algebras, whose grading is given via $G$-quotients as in \eqref{quotgrdef}, is graded
isomorphic to a twisted group $\F$-algebra
$$\F^{\alpha}\Gamma\otimes_{\F}^G\F^{\alpha'}\Gamma'\stackrel{G}\cong\F^{\alpha\times_G\alpha'}(\Gamma\times_G\Gamma')$$
of the pull-back $\Gamma\times_G\Gamma'$ of the given groups, twisted by a certain ``pull-back" $\alpha\times_G\alpha'$ of their twistings.

The $G$-graded modules over $G$-graded $\F$-algebras are endowed with a compatible graded product $W_G\otimes_{\F}^GW'_G$ (Definition \ref{gmodprod})
and, particularly, with a twisting option $W_G\mapsto \alpha(W_G)$ (Definition \ref{tgm}).
We show that graded products admit the following important features, which are used in the proofs of the main theorems.
\begin{enumerate}
\item The algebra of graded endomorphisms of $\alpha(W_G)$ over $\alpha(A_G)$ is graded isomorphic to the algebra of
graded endomorphisms of $W_G$ over $A_G$, twisted by $\alpha$ (Corollary \ref{endtwist}).
\item If $W_G,W'_G$ are graded modules over strongly $G$-graded algebras $A_G$ and $A'_G$ respectively,
then the full tensor product module $W_G\otimes_{\F}W'_G$ over the full tensor product algebra $A_G\otimes_{\F}A'_G$ is induced from
the graded product module $W_G\otimes^G_{\F}W'_G$ over the graded product algebra $A_G\otimes^G_{\F}A'_G$ (Corollary \ref{indstgrprod}).
\item Graded End-Tensor relations: Suppose that $W_G,W'_G$ are graded left modules over $G$-graded $\F$-algebras $A_G,A'_G$ respectively.
Then the graded endomorphism algebra $\End^{\text{r}(G)}_{A_G\otimes^G_{\F}A'_G}(W_G\otimes^G_{\F}W'_G)$ of the graded product $W_G\otimes^G_{\F}W'_G$
is graded isomorphic to the graded product of the graded endomorphism algebras $\End^{\text{r}(G)}_{A_G}(W_G)\otimes^G_{\F}\End^{\text{r}(G)}_{A'_G}(W'_G)$
either if both $W_G,W'_G$ are finitely-presented and $A_G,A'_G$ are both strongly graded ({Corollary} \ref{grHTrcor}),
or if these graded modules are both invariant and absolutely $\F$-simply graded
(Theorem \ref{preobst}).
\item Suspensions commute with graded products (Lemma \ref{susresp}), 
and in particular with twistings (Corollary \ref{galphcomm}).
\item The inertia of a module over the base algebra $A_e$ with respect to the graded algebras $\alpha(A_G)$ is the same for all $\alpha\in Z^2(G,\F^*)$
(Proposition \ref{invwrtgract}).
\item An $A_e$-module is restricted from $A_G$ if and only if it is restricted from $\alpha(A_G)$ for every 2-coboundary $\alpha\in B^2(G,\F^*)$ (Proposition \ref{resalpha}).
\end{enumerate}
The homomorphism \eqref{WGF} is a restriction of the more general {\it endomorphism map} (see \S\ref{MGFsec})
\begin{equation}\label{preend}\End:\mathcal{M}_{G,\F} \to\text{Gr}_{}(G,\F),\end{equation} where $\mathcal{M}_{G,\F}$
denotes the set of graded equivariance classes of left $G$-graded modules (not necessarily simply $G$-graded) over $G$-graded $\F$-algebras, both of bounded cardinality.
The graded product endow $\mathcal{M}_{G,\F}$ with a monoidal structure and, in particular, with an $H^2(G,\F^*)$-set structure.
Similarly to Gr$(G,\F)$, the monoid $\mathcal{M}_{G,\F}$ is acted also by $G$ via the suspension.
Loosely speaking, the map \eqref{preend} takes the equivariance class of a left $G$-graded module $W_G$ over a graded algebra \eqref{eq:algebragrading}
to the grading class $[\End^{\text{r}(G)}_{A_G}(W_{G})]$ of its graded $A_G$-endomorphisms. It possesses the following properties:
\begin{enumerate}
\item End is a map of pointed $H^2(G,\F^*)$-sets (Theorem \ref{mapH2sets}), as well as of pointed $G$-sets (Corollary \ref{mapGsets}).
\item When restricted to certain submonoids of $\mathcal{M}_{G,\F}$, it is a homomorphism of monoids (Lemma \ref{endhom}, Theorem \ref{preobst}).
\item In particular, the restriction of the endomorphism map to $\text{Irr}^{\text{inv}}_{G,\F}<\mathcal{M}_{G,\F}$ is essentially the homomorphism \eqref{WGF},
thinking of the cohomology group $H^2(G,\F^*)$ as a submonoid of Gr$(G,\F)$ (see Notation \ref{notsbmon}). This homomorphism $\omega_{G,\F}$ is surjective (Theorem \ref{surjomega}).
\end{enumerate}

Known results, mostly from \cite{NVO04}, are often rewritten here in a way that serves the goals of this article. The reader should however be aware that
the term ``simply graded" in that source tacitly assumes the graded Artinian property.

\section{Preliminaries}\label{sb}
Throughout this paper, all the algebras are assumed to be associative and unital. Unless stated explicitly, when we deal with modules and ideals of an algebra we refer to them as being acted
by the algebra from the left.
When a module $W$ of an $\F$-algebra $A$ is given by a homomorphism
$$\eta:A\to\text{End}_{\F}(W),$$
we write for short $x(w)$ instead of $\eta(x)(w)$ for $x\in A$ and $w\in W$, and either refer to $[W]$ or to $[\eta]$ as the same member in Mod$(A)$.

As customary, the linear endomorphisms of $W$ which respect its $A$-module structure are denoted by $\text{End}_{A}(W).$ Then $W$ is an $A\otimes_{\F}\text{End}_{A}(W)$-module
(or, alternatively, $_AW_{\text{End}_{A}(W)}$ is a bimodule, thinking of $\text{End}_{A}(W)$ as acting on $W$ from the right).

The restriction of an $A$-module $W$ to a subalgebra $B\subseteq A$ is denoted by $W|_B$ or just by $W$.
A (left) $B$-module $M$ is {\it restricted} from $A$ if there exists an $A$-module $W$ such that
$M=W|_B$. In this case we also say that $W$ {\it extends} $M$.
More generally, an $A$-module $W$ {\it lies above} $M$ if $M$ is a $B$-submodule of $W|_B$.
The {\it induction} of the $B$-module $M$ to $A$ is the $A$-module $M|^A:=A\otimes_BM$.
We can regard $M$ as a $B$-submodule of $M|^A$ by identifying it with $B\otimes_BM$ and so $M|^A$ lies above $M$.

Lemma \ref{fptens} below makes use of the following standard result in linear algebra.
\begin{lemma}\label{linalg}
Let $\varphi:W_1\to W_2$ and $\varphi':W_1'\to W'_2$ be linear maps of vector spaces over the same field $\F$. Then the linear map
$$ \varphi'':\begin{array}{ccc}
W_1\otimes_{\F}W'_1&\to &W_2\otimes_{\F}W_2'\\
x\otimes_{\F}x'&\mapsto &\varphi(x)\otimes_{\F}\varphi'(x')
\end{array}$$
satisfies $\text{Im}(\varphi'')=\text{Im}(\varphi)\otimes_{\F}\text{Im}(\varphi')\subseteq W_2\otimes_{\F}W_2'$, and
\begin{equation}\label{kertens}
\ker(\varphi'')\stackrel{}{=}\ker(\varphi)\otimes_{\F}W'_1+W_1\otimes_{\F}\ker(\varphi').
\end{equation}
\end{lemma}
\begin{proof}
The description of the image, as well as the inclusion $``\supseteq"$ in \eqref{kertens} are clear. Let us prove the converse inclusion $``\subseteq"$ in \eqref{kertens}.
Assume first that $W_1$ and $W_1'$ are finite-dimensional over $\F$. In this case, the finite dimensions of the domains of the linear maps $\varphi,\varphi'$ and $\varphi''$
are the sums of the dimensions of the corresponding kernels and images. Thus,
$$\dim_{\F}(\ker(\varphi''))=\dimF(W_1\otimes_{\F}W_1')-\dimF(\IM(\varphi''))=\cdots$$
By the above description of the image of $\varphi''$ we put
$$\begin{array}{cl}
\cdots &=\dimF(W_1\otimes_{\F}W_1')-\dimF(\IM(\varphi)\otimes_{\F}\IM(\varphi'))\\
 &=\dimF(W_1)\cdot\dimF(W_1')-\dimF(\IM(\varphi))\cdot\dimF(\IM(\varphi'))\\
  &=\dimF(\ker(\varphi))\cdot\dimF(W_1')+\dimF(W_1)\cdot\dimF(\ker(\varphi'))-\dimF(\ker(\varphi))\cdot\dimF(\ker(\varphi'))\\
  &=\dimF(\ker(\varphi)\otimes_{\F}W_1')+\dimF(W_1\otimes_{\F}\ker(\varphi'))-\dimF(\ker(\varphi)\otimes_{\F}\ker(\varphi'))=\cdots
 \end{array}$$
Since $$(\ker(\varphi)\otimes_{\F}W_1')\cap (W_1\otimes_{\F}\ker(\varphi'))=\ker(\varphi)\otimes_{\F}\ker(\varphi')$$ we obtain
  $$\cdots=\dimF(\ker(\varphi)\otimes_{\F}W_1'+W_1\otimes_{\F}\ker(\varphi')).$$
An equality of the finite dimensions of the two sides in \eqref{kertens}, together with the inclusion $``\supseteq"$ proves \eqref{kertens} under the finite-dimensionality assumption on
the vector spaces $W_1$ and $W_1'$.
We relax this assumption and move to the general, perhaps infinite-dimensional, case. For any $\sum_{j=1}^mx_j\otimes_{\F}y_j\in\ker(\varphi''),$
let $\hat{\varphi}$ and $\hat{\varphi}'$ be the restrictions of the linear maps $\varphi$ and $\varphi'$ to the finite-dimensional subspaces
$$V:=\SpF\{x_1\cdots,x_m\}\subseteq W_1\text{ and }V':=\SpF\{y_1\cdots,y_m\}\subseteq W_1'$$ respectively.
Clearly, $\sum_{j=1}^mx_j\otimes_{\F}y_j\in\ker(\hat{\varphi}''),$ where $\hat{\varphi}''$ is the restriction of $\varphi''$ to the finite-dimensional subspace
$V\otimes_{\F}V'\subseteq W_1\otimes_{\F}W'_!$.
This reduction to the finite-dimensional setup yields
$$\begin{array}{cl}\sum_{j=1}^mx_j\otimes_{\F}y_j\in\ker(\hat{\varphi}'')&=\ker(\hat{\varphi})\otimes_{\F}V'+V\otimes_{\F}\ker(\hat{\varphi}')\\
 & \subseteq \ker(\varphi)\otimes_{\F}W'_1+W_1\otimes_{\F}\ker(\varphi')\end{array}$$
and we are done.
\end{proof}
Let $M$ and $M'$ be modules over $\F$-algebras $A$ and $A'$ respectively. Then $M\otimes_{\F}M'$ has an $A\otimes_{\F}A'$-module structure by
$$(a\otimes_{\F}a')(m\otimes_{\F}m'):=am\otimes_{\F}a'm',\ \ a\in A,a'\in A', m\in M, m'\in M'.$$
Clearly, if $M$ and $M'$ are free, then so is $M\otimes_{\F}M'$.
This fact, as well as the next property of the tensor product will be useful in the sequel.
Recall that an $A$-module $M$ is {\it finitely-presented} if it is an image $\varphi:A^n\twoheadrightarrow M$ of a finitely-generated free $A$-module $A^n$ such that
the $A$-submodule of relations $\ker(\varphi)\subseteq A^n$
is finitely-generated. Alternatively, $M$ admits a projective resolution $$\cdots \to P_1\to P_0\to M\to 0,$$ where $P_0,P_1$ are finitely-generated $A$-modules. We have
\begin{lemma}\label{fptens}
Let $M$ be an $A$-module and let $M'$ be an $A'$-module, for some $\F$-algebras $A$ and $A'$.
Suppose that both $M$ and $M'$ are finitely-generated, respectively finitely-presented.
Then their tensor product $M\otimes_{\F}M'$ is finitely-generated, respectively finitely-presented, over $A\otimes_{\F}A'$.
\end{lemma}
\begin{proof}
A set of generators for $M\otimes_{\F}M'$ can be taken as the tensor products of generators of $M$ and generators of $M'$.
Hence, if both $M$ and $M'$ are finitely-generated, then so is $M\otimes_{\F}M'$.
Suppose now that $M$ and $M'$ are finitely-presented. Let $\varphi:A^n\twoheadrightarrow M$ and $\varphi':(A')^m\twoheadrightarrow M'$ such that $\ker(\varphi)$ and $\ker(\varphi')$
are finitely-generated. Then $M\otimes_{\F}M'$ is an image of the finitely-generated free $A\otimes_{\F}A'$-module $A^n\otimes_{\F}(A')^m$ under
$$ \varphi'':\begin{array}{ccc}
A^n\otimes_{\F}(A')^m&\to &M\otimes_{\F}M'\\
x\otimes_{\F}x'&\mapsto &\varphi(x)\otimes_{\F}\varphi'(x')
\end{array}.$$
Concentrating on $ \varphi''$ as a linear map, we infer by \eqref{kertens}
$$\ker(\varphi'')\stackrel{}{=}\ker(\varphi)\otimes_{\F}(A')^m+A^n\otimes_{\F}\ker(\varphi').$$
Since $\ker(\varphi)$ and $A^n$ are finitely-generated over $A$, while $\ker(\varphi')$ and $(A')^m$ are finitely-generated over $A'$, then by appropriate tensor products
of their generators we can construct finitely many generators of
$\ker(\varphi'')$ over $A\otimes_{\F}A'$, and deduce that $M\otimes_{\F}M'$ is finitely-presented over $A\otimes_{\F}A'$.
\end{proof}
More properties of the tensor products are given in the following subsections.

\subsection{Hom-Tensor relations} \label{HTrsec}
Let $A_1$ and $A_2$ be two $\F$-algebras. For $i=1,2$ let $M_i$ and $M'_i$ be left $A_i$-modules.
Then as can easily be checked, the Hom-Tensor mapping
$$\Psi:\Hom_{A_1}(M_1,M'_1)\otimes_{\F}\Hom_{A_2}(M_2,M'_2)\to\Hom_{A_1\otimes_{\F}A_2}(M_1\otimes_{\F}M_2,M'_1\otimes_{\F}M'_2)$$
determined by
\begin{equation}\label{end-ten}
\Psi(f_1\otimes_{\F}f_2):m_1\otimes_{\F}m_2\mapsto f_1(m_1)\otimes_{\F}f_2(m_2),\ \ m_1\in M_1, m_2\in M_2
\end{equation}
is a well-defined $\F$-linear transformation, which preserves multiplication when $M_i=M'_i$ for $i=1,2$.
By \cite[2.4 (P. 14)]{DMI71} $\Psi$ is bijective when $M_1$ and $M_2$ are finitely-generated and projective.
The projectivity assumption on $M_1$ and $M_2$ can be relaxed when the algebras $A_1$ and $A_2$ are Noetherian. This observation follows from
\begin{theorem}\label{ungrHTr}
With the above notation, suppose that the left modules $M_1$ and $M_2$ are finitely-presented.
Then $\Psi$ is an $\F$-linear space isomorphism.
\end{theorem}
\begin{proof}
The following construction is needed for the relaxation of the projectivity assumption. Let
\begin{equation}\label{resolve}\cdots \to P_1\to P_0\to M_1\to 0\end{equation}
be a projective resolution of $A_1$-modules. Then the sequence
$$0\to\Hom_{A_1}(M_1,M'_1)\to\Hom_{A_1}(P_0,M'_1)\to\Hom_{A_1}(P_1,M'_1)$$
is exact. Tensoring this exact sequence over $\F$ with $M_2'':=\Hom_{A_2}(M_2,M'_2)$, i.e.,
$$0\to\Hom_{A_1}(M_1,M'_1)\otimes_{\F}M_2''\to\Hom_{A_1}(P_0,M'_1)\otimes_{\F}M_2''\to\Hom_{A_1}(P_1,M'_1)\otimes_{\F}M_2''$$
preserves exactness.
On the other hand, if we first tensor \eqref{resolve} with $M_2$, and then apply $\Hom_{A_3}(-,M'_3)$,
where $A_3:=A_1\otimes_{\F}A_2$ and $M'_3:=M'_1\otimes_{\F}M'_2$, we obtain
the exact sequence
$$0\to\Hom_{A_3}(M_1\otimes_{\F}M_2,M'_3)\to\Hom_{A_3}(P_0\otimes_{\F}M_2,M'_3)\to\Hom_{A_3}(P_1\otimes_{\F}M_2,M'_3) .$$
Moreover, the diagram
$$\begin{array}{llll}
0\to &\Hom_{A_1}(M_1,M'_1)\otimes_{\F}M_2''&\to \Hom_{A_1}(P_0,M'_1)\otimes_{\F}M_2''&\to \Hom_{A_1}(P_1,M'_1)\otimes_{\F}M_2'' \\
 (*) &\ \ \ \ \ \ \ \ \ \Psi_1\downarrow  & \ \ \ \ \ \ \ \ \ \Psi_2\downarrow   & \ \ \ \ \ \ \ \ \ \Psi_3\downarrow\\
0\to &\Hom_{A_3}(M_1\otimes_{\F}M_2,M'_3)&\to\Hom_{A_3}(P_0\otimes_{\F}M_2,M'_3)&\to \Hom_{A_3}(P_1\otimes_{\F}M_2,M'_3)
  \end{array}$$
is commutative. Here, $\Psi_1,\Psi_2,$ and $\Psi_3$ are the corresponding Hom-Tensor mappings.

With the diagram $(\ast)$ in our disposal we first show that Theorem \ref{ungrHTr} holds supposing that at least one of the modules
$M_1$ or $M_2$ is projective. W.l.o.g. let $M_2$ be projective over $A_2$.
The finite presentation assumption on $M_1$ enables us to choose
the projective $A_1$-modules $P_0$ and $P_1$ in \eqref{resolve} to be finitely-generated. 
Since are $A_2$-module $M_2$ is also projective and finitely-generated, then again by \cite[2.4 (P. 14)]{DMI71}, $\Psi_2$ and $\Psi_3$ are bijective.
The Five Lemma (see \cite[\S VIII.4, Lemma 4]{ML71}) ascertains that so is $\Psi_1$, proving the theorem under the additional projectivity assumption on
either $M_1$ or $M_2$.
Consequently, $\Psi_2$ and $\Psi_3$ in the diagram ($\ast$) are bijective even without assuming that $M_2$ is projective (that is only by the
finite-generation of $P_0,P_1$ and $M_2$, as well as by the projectivity of $P_0$ and $P_1$).
Again, the Five Lemma takes care of the bijectivity of $\Psi_1$ under the weaker conditions of Theorem \ref{ungrHTr}.
\end{proof}
We deduce the following End-Tensor relation.
\begin{corollary}\label{ungrHTrcor}
Let $M_1$ and $M_2$ be finitely-presented left modules over $\F$-algebras $A_1$ and $A_2$ respectively.
Then the $\F$-algebras $\End_{A_1\otimes_{\F}A_2}(M_1\otimes_{\F}M_2)$ and
$\End_{A_1}(M_1)\otimes_{\F}\End_{A_2}(M_2)$
are isomorphic.
\end{corollary}
In the next section we establish an End-Tensor isomorphism under another assumption.
\subsection{Absolutely irreducible modules}\label{absec}
By Schur's Lemma, if $M$ is an irreducible module over an $\F$-algebra $A$, then End$_{A}(M)$ is a division algebra containing the multiplications by scalars in $\F$.
The irreducible module $M$ is called \textit{absolutely simple} (or \textit{totally simple}) over $\F$ if these scalar multiplications are its only $A$-endomorphisms,
that is End$_{A}(M)=\F\cdot$Id$_M$.
We write
\begin{equation}\label{Irr}
\text{Irr}_{\F}(A):=\{[M]\in\text{Mod}(A)| M \text{ is absolutely irreducible over }\F\}.
\end{equation}
In case $\F$ is algebraically closed, then by the second part of Schur's Lemma if $M$ is simple and dim$_{\F}M<\infty$ then $[M]\in$Irr$_{\F}(A)$. This classical result can be generalized as follows.
\begin{theorem}\label{dixmier}(``Amitsur's trick", see \cite[Theorem 2.1.1, page 61]{CG97}, \cite[Theorem 2, page 20]{Jacobson64})
Let $D$ be a division algebra over a field $\F$.
Then for every $d\in D$, which is transcendental over $\F$, the set $$\Lambda_d:=\{(d-\lambda)^{-1}\}_{\lambda\in\F}\subset D$$ is linearly independent over $\F$ (in particular, if dim$_{\F}(D)<|\F|$
then $D$ is algebraic over $\F$). Consequently, if
$A$ is an algebra over an algebraically closed field $\F$ and $M$ is a simple $A$-module with dim$_{\F}(M)<|\F|$. Then $[M]\in\text{Irr}_{\F}(A)$.

\end{theorem}
\begin{proof}
Suppose by negation that $\Lambda_d$ is not linearly independent, and let
\begin{equation}\label{minimal}
\sum_{i=1}^n\alpha_i(d-\lambda_i)^{-1}=0, \ \ \ n\geq 1,\ \alpha_i\in \F\setminus\{0\}
\end{equation}
be a minimal non-trivial vanishing linear combination over $\F$. Then in particular
$$\begin{aligned}
\sum_{i=1}^{n-1}\alpha_i(\lambda_i-\lambda_n)(d-\lambda_i)^{-1}+\sum_{i=1}^n\alpha_i&=\sum_{i=1}^n\alpha_i(\lambda_i-\lambda_n+d-\lambda_i)(d-\lambda_i)^{-1}\\
 &=(d-\lambda_n)\sum_{i=1}^n\alpha_i(d-\lambda_i)^{-1}=0.
\end{aligned}$$
By the minimality of $n$ we deduce that $\sum_{i=1}^n\alpha_i\neq 0.$
Next, multiplying the vanishing combination \eqref{minimal} by the product $\prod_{j=1}^n(d-\lambda_j)$ yields
$$0=\prod_{j=1}^n(d-\lambda_j)\sum_{i=1}^n\alpha_i(d-\lambda_i)^{-1}=\sum_{i=1}^n\alpha_i\prod_{j\neq i}(d-\lambda_j)=\sum_{j=0}^{n-1}b_jd^{j}$$
for some $b_0,\cdots,b_{n-1}\in \F$, such that $b_{n-1}=\sum_{i=1}^n\alpha_i$ is non-zero.
Thus, $d$ annihilates a polynomial of degree $n-1$ over $\F$ (with leading coefficient $\sum_{i=1}^n\alpha_i\neq 0)$ contradicting the transcendency of $d$ over $\F$.
This proves that $\Lambda_d$ is linearly independent over $\F$.

To prove the second claim put $D:=$End$_A(M)$, which is indeed a division algebra by the simplicity of $M$. Denote its identity by $\Id_M$.
We show that any $\varphi\in D$ is algebraic over $\F$. Indeed, let $m\in M\setminus \{0\}$.
The demand dim$_{\F}(M)<|\F|$ yields that the subset $$\{(\varphi-\lambda\cdot\Id_M)^{-1}(m)\}_{\lambda\in \F}\subset M$$ cannot be linearly independent over $\F$.
Hence, there is a non-trivial vanishing $\F$-linear combination
\begin{equation}\label{ntvs}
0=\sum_{i=1}^{n}\alpha_i(\varphi-\lambda_i\cdot\Id_M)^{-1}(m),\ \ \ n\geq 1,\ \alpha_i\in \F\setminus\{0\},\ \lambda_i\in \F.
\end{equation}
By \eqref{ntvs} we deduce that
\begin{equation}\label{elae}
\sum_{i=1}^{n}\alpha_i(\varphi-\lambda_i\cdot\Id_M)^{-1}=0,
\end{equation} or else this element is invertible in $D$, a contradiction to $m$ being non-zero.
Equation \eqref{elae} is a non-trivial linear combination in $D$.
By the first part of the theorem $\varphi\in D=$End$_A(M)$ is algebraic over $\F$. Since $\F$ is algebraically closed then any  $\varphi\in D$
is in fact in $\F=\F\cdot$Id$_M,$ proving that $M$ is absolutely simple
over $\F$.
\end{proof}
Here is a nice End-Tensor property (see \S\ref{HTrsec}) of absolutely simple modules.
\begin{theorem}\label{tensor}
Let $[M]\in\text{Irr}_{\F}(A)$ and $[M']\in\text{Mod}(A')$ then \begin{enumerate}
\item End$_{A\otimes_{\F}A'}(M\otimes_{\F}M')\cong\End_{A}(M)\otimes_{\F}\End_{A'}(M')\cong$End$_{A'}(M')$.
\item For every $A\otimes_{\F}A'$-submodule $N''\subseteq M\otimes_{\F}M'$ there exists an $A'$-submodule $N'\subseteq M'$ such that $N''=M\otimes_{\F}N'$.
In particular, if $M'$ is simple then so is $M\otimes_{\F}M'$.
\end{enumerate}
Consequently, if both $M$ and $M'$ are $\F$-absolutely simple then so is the ${A\otimes_{\F}A'}$-module $M\otimes_{\F}M'$.\end{theorem}
\begin{proof}
To prove (1) it is enough to show that the End-Tensor mapping \eqref{end-ten}
is surjective when $M$ is $\F$-absolutely simple (injectivity under this condition is clear).
Let $\phi''\in\text{End}_{A\otimes_{\F}A'}(M\otimes_{\F}M')$.
We claim that for every $0\neq m\in M$ there exists an $A'$-endomorphism $\phi_m:M'\to M'$ such that
\begin{equation}\label{suchthat}
\phi''(m\otimes_{\F}m')=m\otimes_{\F}\phi_m(m'),\ \ \forall m'\in M'.
\end{equation}
Indeed, choose an $\F$-basis $\mathcal{B}$ of $M$ which contains $m$ and write
$$\phi''(m\otimes_{\F}m')=\sum_{i=1}^nm_i\otimes_{\F}m'_i,\ \ \ m_1=m,m_2,\cdots,m_n\in\mathcal{B}.$$
By Jacobson's Density Theorem \cite{Jac}, since $m_1,\cdots,
m_n\in M$ are linearly independent over $\F=$End$_A(M)$, there
exists $x\in A$ such that $xm_1=m_1$ and $xm_i=0$ for every
$i=2,\cdots, n$. We obtain
\begin{eqnarray}\label{rule}\begin{array}{ll}
\phi''(m\otimes_{\F}m')&=\phi''((x\otimes_{\F}1)(m\otimes_{\F}m'))=(x\otimes_{\F}1)\phi''(m\otimes_{\F}m')\\
&=(x\otimes_{\F}1)(\sum_{i=1}^nm_i\otimes_{\F}m'_i)=m\otimes_{\F}m'_1,
\end{array}\end{eqnarray}
and the rule $m'\mapsto m'_1$ in \eqref{rule}
determines a well-defined map $\phi_m:M'\to M'$ satisfying \eqref{suchthat}. Furthermore, it is not hard to verify that $\phi_m\in$End$_{A'}(M')$.
We now claim that $\Psi(\text{Id}_{M}\otimes_{\F}\phi_m)=\phi''$,
providing a pre-image under the End-Tensor mapping $\Psi$ \eqref{end-ten} for every
$\phi''\in\text{End}_{A\otimes_{\F}A'}(M\otimes_{\F}M')$ (in particular, $\phi_m$ is independent of the choice of $0\neq m\in M$).
Again by the density theorem, for any $\hat{m}\in M$ there
exists $\hat{x}\in A$ such that $\hat{x}m=\hat{m}$. Thus, for every $\hat{m}\otimes_{\F}m'\in M\otimes_{\F}M'$
$$\begin{array}{ll}
\Psi(\text{Id}_{M}\otimes_{\F}\phi_m)(\hat{m}\otimes_{\F}m')&=\hat{m}\otimes_{\F}\phi_{m}(m')=\hat{x}m\otimes_{\F}\phi_{m}(m')\\
&=(\hat{x}\otimes_{\F}1)(m\otimes_{\F}\phi_m(m'))=\cdots
\end{array}$$
By \eqref{suchthat}
$$\cdots=(\hat{x}\otimes_{\F}1)\phi''(m\otimes_{\F}m')=
\phi''(\hat{x}m\otimes_{\F}m')=\phi''(\hat{m}\otimes_{\F}m'),$$
establishing (1).
The proof of (2) is postponed at the moment and will be given in a
more general setup of graded products, see Theorem \ref{grtensor}.
\end{proof}
The following straightforward result is recorded for a later use.
\begin{lemma}\label{phiMstable}
Let $B\subseteq A$ be two $\F$-algebras, let $W$ be an irreducible $A$-module lying above a $B$-module $M$ such that $[M]\in$Irr$_{\F}(B)$, and let $\phi\in$End$_A(W)$
such that $\phi(M)\subseteq M$.
Then $\phi=\lambda\cdot$Id$_{W}$ for some $\lambda\in\F$.
In particular, $A$-modules which extend an $\F$-absolutely simple $B$-module are themselves absolutely simple over $\F$.
\end{lemma}
\begin{proof}
Since $M$ is $\F$-absolutely simple, then the restriction $\phi|_{M}$, which is a well-defined endomorphism of $M$ as a $B$-module, is a multiplication by some $\lambda\in\F$.
The $A$-module $W$ is irreducible, hence it is generated by the elements of $M$ over $A$.
Therefore, $\phi$ acts on $W$ as a multiplication by $\lambda$.
\end{proof}
\section{Graded algebras, graded modules and graded morphisms}
Before beginning the chapter, a few words about the connection, mentioned in the introduction, between group graded algebras and group bundles.
It should be pointed out that since we work over arbitrary fields $\F$, there are no convergence demands on group graded $\F$-algebras, even if $\F$ happens to be normed.
The readers are referred to \cite[\S 2.7]{GS16} for the different terminologies used by various authors dealing with discrete group graded algebras, and
are invited to compare these with the corresponding terminologies for Fell Bundles in \cite{FellDoran1,FellDoran2}. 
\subsection{Families of group graded algebras}\label{sgcpsect}

The notion of group-graded rings goes back to the work of E. Dade \cite{Dade70,D80}, and earlier, for abelian groups, to that of M. Knus \cite{Knus69}.
A {\it grading} of an $\F$-linear space $A_G$ by a group $G$ is a homogeneous vector space decomposition \eqref{eq:algebragrading}.
A grading of an $\F$-algebra $A_G$ by a group $G$ is an $\F$-space grading \eqref{eq:algebragrading} which satisfies
\begin{equation}\label{gradsubset}
A_gA_h\subseteq A_{gh}
\end{equation}
for every $g,h\in G$.
The summand $A_e$, which corresponds to the trivial element $e\in G$, contains the identity element 1 of $A_G$, and is itself an $\F$-subalgebra of $A_G$, called the {\it base algebra}
(or the {\it unit fiber algebra}).
Equation \eqref{eq:algebragrading} describes an $A_e$-bimodule decomposition of $A_G$.
The {\it support} of a graded algebra \eqref{eq:algebragrading} is
\begin{equation}\label{supp}
\text{Supp}_G(A_G)=\{g\in G|\text{dim}_{\F}(A_g)\geq 1\}.\end{equation}

For $g\in G$, $A_g$ is the {\it $g$-homogeneous component} (or the {\it $g$-fiber}) of the $G$-graded algebra \eqref{eq:algebragrading}.
It determines a two-sided ideal
\begin{equation}\label{Ig}
\mathcal{I}_g:=A_{g^{-1}}A_{g}\subseteq A_e
\end{equation} of the base algebra $A_e$.
We regard this component as \textit{strong} if \eqref{Ig} is the entire base algebra $A_e$.
In this case, there exist homogeneous elements
$\{a_j^{(g)}\}_j\subset A_{g^{-1}},\{b_j^{(g)}\}_j\subset A_{g}$ (non-uniquely determined) admitting a \textit{homogeneous unit decomposition}
\begin{equation}\label{unitdecomp}
1=\sum_j a_j^{(g)}b_j^{(g)}.
\end{equation}
It is not hard to verify that
\begin{equation}\label{strcomp}
\mathcal{G}_{\text{str}}:=\{g\in G |\ \  A_g\text{ is strong}\}
\end{equation} is a submonoid of $G$ (but not necessarily a subgroup).
A graded algebra \eqref{eq:algebragrading} is said to be {\it
strongly graded} (or {\it saturated}) if for every $g,h\in G$
there is an equality in \eqref{gradsubset}. The algebra
\eqref{eq:algebragrading} is strongly graded if and only if
$\mathcal{G}_{\text{str}}=G$ \cite[Proposition 1.1.1]{NVO04}.

If the ideal \eqref{Ig} is unital, and for all $h\in G$ the equalities
$$A_hA_g=A_{hg}A_{g^{-1}}A_g=A_hA_{h^{-1}}A_{hg}$$ hold, the homogeneous component $A_g$ can be regarded as \textit{epsilon-strong} in the spirit of \cite{NOP18}.

When a homogeneous component $A_g$ admits an invertible element, say $u_g$,
then clearly $$A_g=u_gA_e=A_eu_g.$$
We call such a special instance of a strong component an {\it invertible} component. It is easily verified that
$$\mathcal{G}_{\text{inv}}:=\{g\in G |\ \  A_g\text{ is invertible}\}$$ is a subgroup of $G$.
If all the homogeneous components of \eqref{eq:algebragrading} are invertible, then
$A_G$ is called a (discrete) {\it crossed product}. It is, in particular, strongly graded.
The broad notion in the theory of Fell bundles is {\it homogeneity}, which might cause ambiguity and will not be used here.
The $\mathcal{G}_{\text{inv}}$-graded subalgebra of any $G$-graded algebra is then a crossed product.
A crossed product supported by $G$ is usually written as
\begin{equation}\label{crosp}
A_e*G=\oplus_{g\in G}A_eu_g=\oplus_{g\in G}u_gA_e,
\end{equation}
where $u_g$ is an invertible element in $A_g$.

Let \eqref{crosp} be a crossed product. Then any $g\in G$ gives rise to the conjugation automorphism $r\mapsto\iota_{u_g}(r):=u_g ru_g^{-1}$ of the base algebra $A_e$ by $u_g$.
The map
\begin{eqnarray}\label{preouter}
\begin{array}{rcl}
G&\to & \text{Aut}(A_e)\\
g& \mapsto &\iota_{u_g}
\end{array},\ \ g\in G
\end{eqnarray}
is not necessarily a group homomorphism. However, modulo the inner automorphisms of the base algebra $A_e$, we do get a group homomorphism
\begin{eqnarray}\label{outer}
\begin{array}{ccc}
G&\to & \text{Out}(A_e)=\bigslant{\text{Aut}(A_e)}{\text{Inn}(A_e)}\\
g& \mapsto &\iota_{u_g}\cdot\text{Inn}(A_e)
\end{array},\ \ g\in G,
\end{eqnarray}
which is regarded as an \textit{outer action} of $G$ on $A_e$ (see \cite[\S 2]{AGdR}).
A choice of invertible elements $\{u_g\}_{g\in G}$ give rise to a \textit{twisting}
\begin{eqnarray}\label{fdef}
\begin{array}{rcl}
f:G\times G & \to & A_e^*\\
(g,h) & \mapsto & u_gu_hu_{gh}^{-1}.
\end{array}
\end{eqnarray}
Developing the associativity condition $(u_gu_h)u_k=u_g(u_hu_k)$ we obtain
\begin{equation}\label{assoc}
f(g,h)f(gh,k)=u_gf(h,k)u_g^{-1}f(g,hk)=\iota_{u_g}(f(h,k))f(g,hk),
\end{equation}
for every $g,h,k\in G$.
The invertible element $u_e$ in the base algebra is usually chosen to be $1\in A_e*G$, thus one may assume the \textit{normalization} condition $$f(e,-)=f(-,e)=1.$$
Note that choosing any other invertible homogeneous element $u'_g\in A_{g}$ is the same as multiplying $u_g$ by an element in $A_e^*$. This implies that the corresponding $A_e$-automorphisms
$\iota_{u'_g}$ and $\iota_{u_g}$ differ by an inner $A_e$-automorphism.
Consequently, a choice $\{u'_g\}_{g\in G}$ of invertible homogeneous element gives rise to the same outer action \eqref{outer} of $G$ on $A_e$.
\begin{definition}\label{skewdef}(compare with {\it semidirect product bundles} \cite[VIII.4.2]{FellDoran2})
A crossed product \eqref{crosp} is said to be a {\it skew group algebra} if it admits a set $\{u_g\}_{g\in G}$
of homogeneous invertible elements such that the corresponding twisting \eqref{fdef} is identically 1, that is if $u_gu_h=u_{gh}$ for every $g,h\in G$.
\end{definition}

Another special family of crossed products $\F$-algebras \eqref{crosp}, which is often used here, arises when the base algebra $A_e$ is exactly $\F$
(and so the corresponding outer action \eqref{outer} is trivial).
This is the same as to say that all the homogeneous components of this crossed product $\F$-algebra are 1-dimensional over $\F$.
\begin{definition}\label{deftga}(compare with the definition of {\it cocycle bundles} \cite[VIII.4.8]{FellDoran2})
A crossed product $\F$-algebra is called a {\it twisted group $\F$-algebra} if its base algebra is 1-dimensional over $\F$.
\end{definition}
In this special case, equation \eqref{assoc} is the
\textit{2-cocycle condition} of the group $G$ with coefficients in the trivial $G$-module $\F^*$.
More explicitly, for every 2-cocycle $\alpha\in Z^2(G,\F^*)$, where $G$ acts trivially on $\F^*$, let
\begin{equation}\label{tga}
\F^{\alpha} G:=\oplus_{g\in G}\text{Span}_{\F}\{v_g\}
\end{equation}
be the corresponding twisted group algebra, with multiplication given by
$$v_g\cdot v_h=\alpha(g,h)v_{gh}$$ and endowed with the natural $G$-grading, which we denote by {\it twisted grading}.
As above, any non-zero homogeneous element in $\text{Span}_{\F}\{v_g\}$ is of the form $\lambda_gv_g$ with $\lambda_g\in \F^*$.
Thus, another choice $\{v'_g\}_{g\in G}$ of homogeneous $\F$-basis of $\F^{\alpha} G$ yields a corresponding 2-cocycle
$$\alpha'(g,h)=v'_gv'_h{v'}_{gh}^{-1}=\lambda_gv_g\lambda_hv_hv_{gh}^{-1}\lambda^{-1}_{gh}=\lambda_g\lambda_h\lambda^{-1}_{gh}\alpha{(g,h)}.$$
Then the cocycles $\alpha'{(g,h)}$ and $\alpha(g,h)$ differ by a {\it 2-coboundary} (or a {\it cohomologically trivial} cocycle)
\begin{equation}\label{coboundary}
(g,h)\mapsto \lambda_g\lambda_h\lambda^{-1}_{gh}
\end{equation}
and are called {\it cohomologous}.
Recall that the 2-coboundaries, which are regarded as the trivial twistings, form a subgroup $B^2(G,\F^*)\lhd Z^2(G,\F^*)$. The quotient
$$H^2(G,\F^*)=\bigslant{Z^2(G,\F^*)}{B^2(G,\F^*)}$$ is the second
cohomology group of $G$ over its trivial module $\F^*$.
So, without the choice of homogeneous $\F$-basis, a twisted group algebra determines a cohomology class $[\alpha]\in H^2(G,\F^*)$.
As in the general case of crossed products, the homogeneous element $v_e\in (\F^{\alpha} G)_e$ can be chosen to be 1. The corresponding normalized
2-cocycle $\alpha\in Z^2(G,\F^*)$ satisfies $\alpha(e,-)=\alpha(-,e)=1$.

A graded algebra, all of whose non-zero homogeneous elements are invertible is called a {\it graded division algebra} (or a {\it graded skew field}).
The following claim is straightforward.
\begin{lemma}\label{grdiv}
A graded algebra \eqref{eq:algebragrading} is a graded division algebra if and only if it is a crossed product (over its support), whose base algebra $A_e$ is an ungraded division algebra.
\end{lemma}
Twisted group algebras are thus graded division algebras. Here is a setting where twisted group algebras are the only graded division algebras.
The following lemma was proven in \cite{MR2488221,EK13} for graded division algebras which are finite-dimensional over an algebraically-closed field $\F$.
The finiteness of dimension may however be entailed only on the corresponding base algebra.
\begin{lemma}\label{grdivtwisted}(see \cite[Theorem 2]{MR2488221}, \cite[Theorem 2.13]{EK13})
A graded division algebra over an algebraically-closed field $\F$, whose base algebra is finite-dimensional over $\F$ is a twisted group algebra over $\F$.
\end{lemma}
\begin{proof}
By Lemma \ref{grdiv}, a graded division algebra is a crossed product, whose base algebra $A_e$ is a division algebra.
Since $A_e$ is finite-dimensional over the algebraically-closed field $\F$,
then $A_e=\F$. By Definition \ref{deftga} the graded division algebra is a twisted group $\F$-algebra.
\end{proof}

\subsection{Graded homomorphisms of graded algebras}\label{grmorph}
A {\it graded-homomorphism} between the graded algebra
\eqref{eq:algebragrading} and another graded $\F$-algebra
\begin{equation}\label{eq:equivgr}
\mathcal{H}:B_H=\bigoplus _{h\in H} B_h
\end{equation}
is a pair $(\psi,\phi)$, where $\psi:A_G\rightarrow B_H$ is an
algebra homomorphism and $\phi:G\rightarrow H$ is a group
homomorphism such that $\psi(A_g)\subseteq B_{\phi(g)}$ for any
$g\in G$.
In case $G=H$ we write $A_G\stackrel{G}{\to} B_G$ for short. A
graded-homomorphism $(\psi,\phi):\mathcal{G}\to\mathcal{H}$
between these two graded algebras is a {\it graded-equivalence} if
$\psi$ and $\phi$ are isomorphisms of algebras and groups
respectively. A graded-equivalence $(\psi,\phi)$ is a {\it
graded-isomorphism} if $H=G$ and $\phi$ is the identity map. We
then write $A_{G}\stackrel{G}{\cong}B_G$.
In the terminology of \S\ref{Ghgm}, a graded isomorphism is an isomorphism of algebras, which is homogeneous of trivial degree $e\in G$ as a linear map of $G$-graded $\F$-spaces.
The grading isomorphism class of $\mathcal{G}$ is denoted by $[\mathcal{G}]$. The following lemma is recorded for a later use.
\begin{lemma}\label{cphomim}
Graded homomorphisms take strong and invertible components to strong and invertible components respectively.
In particular, graded homomorphic images of strongly graded algebras and, alternatively, of crossed products are also strongly graded, respectively crossed products. Moreover,
suppose that $\psi:A_e*G\stackrel{G}{\to} B_G$ is a graded homomorphism from a crossed product \eqref{crosp} to a $G$-graded $\F$-algebra $B_G$,
such that $\psi$ takes $A_e$ isomorphically onto the base algebra $B_e$ of $B_G$. Then $\psi$ is a graded isomorphism.
\end{lemma}
\begin{proof}
The first claim follows from the fact that graded homomorphisms take homogeneous unit decompositions \eqref{unitdecomp} to homogeneous unit decompositions, and in particular,
invertible homogeneous elements to invertible homogeneous elements.
The second claim follows from the decomposition $B_G=\oplus_{g\in G}B_e\psi(u_g)$, where all the homogeneous images $\{\psi(u_g)\}_{g\in G}$ are invertible.
\end{proof}
\begin{corollary}\label{preserved}
The properties of being either strongly graded or a crossed product are preserved under graded isomorphism and graded equivalence.
\end{corollary}
A graded-homomorphism
$(\psi,\phi):\mathcal{G}\to\mathcal{H}$ is a {\it quotient
homomorphism} if $\psi$ is an isomorphisms of algebras, and the
group homomorphism $\phi$ is onto. In this case we say that
\eqref{eq:equivgr} is a {\it quotient-grading} of
\eqref{eq:algebragrading}. Quotients of graded algebras are often
given just by a quotient of the grading group. Let $N\lhd G$ be a
normal subgroup of $G$. Then by a $G/N$-grading of
\eqref{eq:algebragrading} we mean the quotient-grading of the
algebra $A_G$ given by the pair $({1}_A,\md N)$, where
\begin{equation}\label{quotgrad}
({1}_A,\md N)(\mathcal{G}):A_{G/N}=\bigoplus _{gN\in G/N} A_{gN}.
\end{equation}
Here $1_A:A\to A$ is the identity map, and $
A_{gN}:=\bigoplus_{g\in gN}A_g.$ When a quotient homomorphism is given
in terms of the normal subgroup $N\lhd G$, we denote the
corresponding quotient grading by $\mathcal{G}/ N$ for short. The
following assertion is straightforward.
\begin{claim}\label{cpquotcp}
Any $G/N$-quotient grading of a strongly graded algebra
\eqref{eq:algebragrading} is itself strongly graded over the base
algebra $ A_{N}=\bigoplus_{n\in N}A_n.$ A $G/N$-quotient grading
of a crossed product $A_e*G$ is a crossed product whose base
algebra is $A_e*N$.
\end{claim}
Evidently, any algebra $A$ can be viewed as graded by the trivial
group. This trivial grading is a quotient of all possible gradings
of $A$. The quotient homomorphism, which is responsible for that, just
``forgets" the graded data of a grading \eqref{eq:algebragrading}
and leaves only its algebra structure (see \cite[\S 1.2]{NVO04}).
With the above notation, this forgetful homomorphism is given by the
pair $(1_A,$ mod $G)$.

\subsection{Graded modules and graded-simple algebras}\label{gmgsa}

A left (right) module $W$ of a graded algebra \eqref{eq:algebragrading} is {\it $G$-graded} if it also affords a linear decomposition
\begin{equation}\label{modec}
W_G=\bigoplus_{g\in G} W_g
\end{equation}
which respects the grading $\mathcal{G}$, that is
for every $g,h\in G$
\begin{equation}\label{stgrmod}
A_gW_h\subseteq W_{gh}\ \ \ (\text{respectively }W_hA_g\subseteq W_{hg}).
\end{equation}
In particular, every homogeneous component $W_g\subset W_G$ is an $A_e$-module. Note that
\begin{enumerate}
\item the quotient $\bigslant{W_G}{W'_G}$ of \eqref{modec} by a graded submodule $W'_G\subseteq W_G$ is a graded $A_G$-module with the natural $G$-grading, and that
\item a graded algebra \eqref{eq:algebragrading} is a graded bimodule over itself, moreover, any free $A_G$-module is naturally $G$-graded.
\end{enumerate}

The {\it support} of a graded module \eqref{modec} is
$$\text{Supp}_G(W_G)=\{g\in G|\ \ \text{dim}_{\F}(W_g)\geq 1\}.$$
In particular, the support \eqref{supp} of a graded algebra \eqref{eq:algebragrading} is its support as a bimodule over itself.
For graded modules over strongly graded algebras we have
\begin{lemma}\label{genst}
Let \eqref{modec} be a graded left module over a $G$-graded algebra \eqref{eq:algebragrading} and let $g^{-1}\in\mathcal{G}_{\text{str}}$ (see \eqref{strcomp}).
Then for every $h\in G$, \eqref{stgrmod} is an equality. Hence if $W_h$ vanishes, then so is $W_{gh}$.
In particular, if \eqref{eq:algebragrading} is strongly graded and $W_G\neq 0$ then $\text{Supp}_G(W_G)=G$, and $W_G$ is generated by $W_e$.
\end{lemma}
\begin{proof}
The strong component condition yields $$W_{gh}=A_eW_{gh}=A_gA_{g^{-1}}W_{gh}\subseteq A_gW_h,$$
proving the converse inclusion of \eqref{stgrmod}.
\end{proof}
\begin{lemma}\label{fgenst}
Let \eqref{modec} be a finitely-generated graded left module over a $G$-graded algebra \eqref{eq:algebragrading}.
Then there exists a finite set of homogeneous elements that generate $W_G$.
Moreover, if \eqref{eq:algebragrading} is further assumed to be strongly graded,
then there exists a finite generating set of $W_G$ consisting solely of homogeneous elements of $W_e$.
\end{lemma}
\begin{proof}
The first part of the claim is obtained by replacing any element in a finite generating set of $W_G$ by the homogeneous elements in its graded decomposition.
Suppose now that \eqref{eq:algebragrading} is strongly graded and let $X_G$ be a finite generating set of $W_G$ consisting of homogeneous elements.
The extra strongly graded assumption enables us to substitute every $w_g\in X_G$ by finitely many elements in $W_e$ as follows.
Let $$1=\sum_{j=1}^m c_j^{(g)}d_j^{(g)},\ \ \{c_j^{(g)}\}_{j=1}^m\subset A_{g},\ \ \{d_j^{(g)}\}_{j=1}^m\subset A_{g^{-1}}$$
be a homogeneous unit decomposition. Then replace $w_g$ by the elements $$d_1^{(g)}w_g,\cdots, d_m^{(g)}w_g\in A_{e}.$$
Clearly, $w_g$ is generated by these $e$-homogeneous elements. Repeating this procedure for all elements in $X_G$ yields a finite generating set of $e$-homogeneous elements for $W_G$.
\end{proof}
\begin{definition}
Let \eqref{modec} be a graded left module over a $G$-graded algebra \eqref{eq:algebragrading}.
\end{definition}

A non-zero graded $A_G$-module \eqref{modec} is {\it graded simple} if it does not admit proper non-zero graded submodules. The following claim is straightforward.
\begin{lemma}\label{grsimpsimp}
Let \eqref{modec} be a simply graded left module over a $G$-graded algebra \eqref{eq:algebragrading}. Then $W_G=A_GW_g$ for every $g\in\text{Supp}_G(W_g)$.
In particular, $W_g$ is simple as an ungraded $A_e$-module.
\end{lemma}

\begin{definition}(see \cite[\S I.7]{NVO82})
A left ideal $I_G$ of a group-graded algebra \eqref{eq:algebragrading} is a {\it graded ideal} if it is graded as a left $A_G$-module, that is if $I_G=\oplus _{g\in G} (I_G\cap A_g)$.
A graded algebra is {\it graded-simple} if it admits no non-trivial two-sided graded ideals.
\end{definition}
Let \eqref{eq:algebragrading} be a group-graded algebra and let $I$ be a left ideal of $A_e$. Then the extended left ideal $A_G\cdot I$ of $A_G$ is $G$-graded by
$A_G\cdot I=\bigoplus _{g\in G} A_g\cdot I$.
Since
$(A_G\cdot I)\cap A_e=I$ we obtain
\begin{lemma}\label{inj}
The extension $I\mapsto A_G\cdot I$ determines an injective correspondence from the set of left $A_e$-ideals to the set of left graded $A_G$-ideals.
\end{lemma}
\begin{lemma}\label{AeAss}
Let \eqref{eq:algebragrading} be graded-simple. Then $A_e$ is semiprimitive, that is $A_e$ admits a trivial Jacobson radical.
\end{lemma}
\begin{proof}
The \textit{graded Jacobson radical} of $A_G$ is trivial, as a property
of simply graded algebras (see \cite[I.7.3]{NVO82}).
By \cite[Corollary 2.9.3]{NVO04}, the ordinary Jacobson radical of the corresponding base algebra $A_e$ also vanishes.
\end{proof}
\begin{definition}
A graded left module \eqref{modec} is called \textit{ graded Artinian} if it satisfies the descending chain condition on graded submodules.
In particular, a group-graded algebra \eqref{eq:algebragrading} is \textit{left graded Artinian} if it satisfies the descending chain condition as a left graded module over itself.
\end{definition}
\begin{lemma}\label{redext}
If the group-graded algebra \eqref{eq:algebragrading} is
left graded Artinian
then $A_e$ is left Artinian.
\end{lemma}
\begin{proof}
This is a consequence of the fact that the extension $I\mapsto A_g\cdot I$ is an injective map from the left $A_e$-modules to the left graded $A_G$-modules (see Lemma \ref{inj}),
which respects inclusions.
\end{proof}
\begin{corollary}\label{AeAsscor}
A graded Artinian graded-simple algebra admits a semisimple Artinian base algebra.
\end{corollary}
\begin{proof}
Apply Lemma \ref{AeAss} and Lemma \ref{redext}.
\end{proof}
Evidently, graded division algebras do not afford non-trivial one-sided, let alone two-sided, graded ideals.
They hence serve as seminal examples for graded Artinian and graded simple algebras as follows.
\begin{proposition}\label{examplegrartsimp}
Any graded division algebra is simply graded Artinian.
\end{proposition}

Let \eqref{modec} be a graded module over a graded division $\F$-algebra.
In general, homogeneous elements of $W_G$, which are linearly independent over $\F$ may be dependent over the graded division algebra. When the graded division algebra is a twisted group algebra
we have the following straightforward
\begin{lemma}\label{linind}
Let \eqref{modec} be a graded module over a twisted group algebra \eqref{tga}.
Suppose that the homogeneous elements $w_1^g,\cdots,w^g_n\in W_g$, all of degree $g\in G$, are linearly independent over $\F$, then they are linearly independent
also over $\F^{\alpha}G$.
\end{lemma}

Let \eqref{modec} be a graded left module over a $G$-graded algebra \eqref{eq:algebragrading}. Then for every $g\in G$ there exists a unique graded submodule
$t_{\mathcal{C}_g}(W_G)\subseteq W_G$ (see \cite[p. 38]{NVO04}), which is maximal among those graded submodules of $W_G$ having trivial $g$-th homogeneous component.
This submodule is the \textit{localizing radical} of $W_G$.
Following \cite{NVO04}, our interest herein is in the quotient $\bigslant{W_G}{t_{\mathcal{C}_e}(W_G)}$.
The next result is a consequence of Lemma \ref{genst}.
\begin{lemma}\label{tcg}
Let \eqref{modec} be a graded left module over a $G$-graded algebra \eqref{eq:algebragrading} and let $g^{-1}\in\mathcal{G}_{\text{str}}$. Then for every $h\in G$,
$t_{\mathcal{C}_h}(W_G)\subseteq t_{\mathcal{C}_{gh}}(W_G)$ as well as $(t_{\mathcal{C}_h}(W_G))_{gh}=0$.
In particular, if \eqref{eq:algebragrading} is strongly graded then $t_{\mathcal{C}_{g}}(W_G)=0$ for every $g\in G$.
\end{lemma}
A left $A_e$-module $M$ determines an extended graded left $A_G$-module
$$A_G\otimes_{A_e}M=\bigoplus_{g\in G} (A\otimes_{A_e}M)_g,$$ where
$(A_G\otimes_{A_e}M)_g:=A_g\otimes_{A_e}M$ for every $g\in G$.
Factoring out the extension ${A_G}\otimes_{A_e}M$ by the corresponding localizing radical $t_{\mathcal{C}_e}({A_G}\otimes_{A_e}M)$ as above yields the $G$-graded ${A_G}$-module
\begin{equation}\label{assogrd}
{A_G}\overline{\otimes}_{A_e}M:=\bigslant{{A_G}{\otimes}_{A_e}M}{t_{\mathcal{C}_e}({A_G}{\otimes}_{A_e}M)}
\end{equation}
{\it associated} to $M$.
Lemma \ref{tcg} implies the first part of the following theorem.
\begin{theorem}\label{indfunct}\cite[Theorem 2.7.2]{NVO04}
Let \eqref{eq:algebragrading} be a $G$-graded algebra and let $M$ be a left $A_e$-module. Then with the above notation \begin{enumerate}
\item If $g^{-1}\in\mathcal{G}_{\text{str}}$, then there is an $A_e$-module isomorphism $({A_G}\overline{\otimes}_{A_e}M)_g\cong {A_g}\otimes_{A_e}M$.
In particular, $({A_G}\overline{\otimes}_{A_e}M)_e\cong M$ as $A_e$-modules, and
if ${A_G}$ is further assumed to be strongly graded then there is a graded isomorphism of graded left $A_G$-modules (see \S\ref{Ghgm})
$${A_G}\overline{\otimes}_{A_e}M\stackrel{G}\cong {A_G}\otimes_{A_e}M.$$
\item If $M$ is simple over $A_e$ then ${A_G}\overline{\otimes}_{A_e}M$ is graded simple over ${A_G}$ with $({A_G}\overline{\otimes}_{A_e}M)_e\cong M$.
\item $M\cong M'$ if and only if ${A_G}\overline{\otimes}_{A_e}M\stackrel{G}\cong {A_G}\overline{\otimes}_{A_e}M'$.
\end{enumerate}
\end{theorem}

\subsection{Graded homomorphisms of graded modules}\label{Ghgm}
A linear map $\varphi$ between $G$-graded $\F$-spaces $W_G$ and $W'_G$ is \textit{left homogeneous of degree $h\in G$} if
\begin{equation}\label{lefthom}
\varphi:W_g\to W'_{hg},\ \ \forall g\in G,
\end{equation}
and, analogously, \textit{right homogeneous of degree $h\in G$} if
\begin{equation}\label{righthom}
\varphi:W_g\to W'_{gh},\ \ \forall g\in G.
\end{equation}
A homogeneous linear $\F$-map of trivial degree $e\in G$ is termed a \textit{$G$-graded} map and may be denoted by $W_G\stackrel{G}{\to}W'_G$.
It is easily verified that the kernel and image of left and right homogeneous homomorphisms are graded subspaces.
The space of left (right) homogeneous homomorphisms of degree $h\in G$ between $W_G$ and $W'_G$ are denoted by Hom$_{\F}^{\textmd{l}(h)}(W_G,W_G')$ (respectively, Hom$^{\textmd{r}(h)}_{\F}(W_G,W'_G)$).
The direct sums of these homogeneous subspaces of Hom$_{\F}(W_G,W_G')$
are $G$-graded spaces denoted by
$$\begin{array}{l}\text{Hom}^{\text{l}(G)}_{\F}(W_G,W'_G):=\bigoplus_{h\in G}\text{Hom}_{\F}^{\textmd{l}(h)}(W_G,W_G'),\text{ and }\\ \text{Hom}^{\text{r}(G)}_{\F}(W_G,W'_G):
=\bigoplus_{h\in G}\text{Hom}_{\F}^{\textmd{r}(h)}(W_G,W'_G).\end{array}$$
When $W_G=W_G'$ then the corresponding graded spaces $\text{End}^{\text{l}(G)}_{\F}(W_G)$ and $\text{End}^{\text{r}(G)}_{\F}(W_G)$ are graded algebras acting on $W_G$ from the left,
or, respectively, from the right.
We adopt these orientations when acting with either of these endomorphism algebras.
When we do not care about orientation we write $\text{Hom}^{\text{gr}}_{\F}(W_G,W'_G)$ and, particularly, $\text{End}^{\text{gr}}_{\F}(W_G)$.

If the graded spaces $W_G$ and $W'_G$ are furnished with additional graded left module structures over a $G$-graded algebra \eqref{eq:algebragrading} then
\begin{equation}\label{homgr}\text{Hom}^{\text{r}(G)}_{A_G}(W_G,W_G'):=\text{Hom}^{\text{r}(G)}_{\F}(W_G,W_G')\cap\text{Hom}_{A_G}(W_G,W'_G)\end{equation}
is a graded subspace of $\text{Hom}^{\text{r}(G)}_{\F}(W_G,W_G')$.
If there exists an invertible graded homomorphism $\varphi\in$Hom$_{A_G}^e(W_G,W_G')$
(i.e., homogeneous of trivial degree $e\in G$), we say that $W_G$ and $W_G'$ are {\it graded isomorphic} and write $W_G\stackrel{G}\cong W_G'$.
Nonetheless, our main distinguishing between graded modules will be coarser, namely up to \textit{graded equivariance} (see Definition \ref{defeq}).

As a particular case of \eqref{homgr}, the graded algebra of right graded endomorphisms of the graded left module $W_G$ is denoted by
\begin{equation}\label{endgr}
A_{W}:=\text{End}^{\text{r}(G)}_{A_G}(W_G)=\bigoplus_{h\in G}\text{End}_{A_G}^{\textmd{r}(h)}(W_G).
\end{equation}
\begin{example}\label{frendex}
Let $W_G:=(A_G)^x$ be the free left $A_G$-module, generated by one element $x$. Since any $\varphi\in \text{End}_{A_G}(W_G)$ is determined by $\varphi(x)$, there is a graded isomorphism of algebras
$$\text{End}^{\text{r}(G)}_{A_G}(W_G)=\text{End}^{}_{A_G}(W_G)\stackrel{G}{\cong} A_G.$$
\end{example}
Under the $G$-grading \eqref{modec}, $W_G$ is a graded bimodule with respect to the actions of $A_G$ from the left and of $A_W$ from the right
(and since it is an $A_G\otimes_{\F} A_W$-module, it is a graded module over the graded product $A_G\otimes_{\F}^G A_W$, see \S\ref{actongral}).
As above, let
$$\text{End}^{\text{l}(G)}_{A_W}(W_G):=\text{End}_{\F}^{\text{l}(G)}(W_G)\cap\text{End}_{A_W}(W_G)$$ be the corresponding graded algebra of intertwining graded endomorphisms of $W_G$.
For any homogeneous element $x_g\in A_g$, let $$\varphi_{x_g}:\begin{array}{ccc}W_G&\to& W_G\\ w&\mapsto&x_g(w)\end{array}$$ denote
the action of $x_g$ on the left module $W_G$.
Then the above setup says that $\varphi_{x_g}$ is a left homogeneous $A_W$-endomorphism of $W_G$ of degree $g$, that is $\varphi_{x_g}\in$End$^{\textmd{l}(g)}_{A_W}(W_G)$.
The following result is easily verified.
\begin{lemma}\label{varphixglemma}
Let \eqref{modec} be a graded left module of a graded algebra \eqref{eq:algebragrading}. Then with the notation \eqref{endgr}
\begin{eqnarray}\label{varphixg}\varphi_W:\begin{array}{ccc}
A_G&\to&\text{End}^{\text{l}(G)}_{A_W}(W_G)\\
\sum_gx_g&\mapsto&\sum_g\varphi_{x_g}
\end{array}\end{eqnarray} is a homomorphism of graded algebras.
\end{lemma}
Note that any homogeneous component $W_g$ is stable under the endomorphisms in $\text{End}^e_{A_G}(W_G)$. Hence, for every $g\in G$ the restriction
\begin{equation}\label{gres}
\text{res}_g:\text{End}^e_{A_G}(W_G)\to\text{End}_{A_e}(W_g)
\end{equation}
is a well-defined algebra homomorphism. Moreover, if $W_G$ is simply graded, then it is generated as an ${A_G}$-module by any non-zero homogeneous element.
Consequently, if $g\in$Supp$(W_G)$ then endomorphisms in $\text{End}_{A_e}(W_g)$ determine a unique endomorphism in $\text{End}^e_{A_G}(W_G)$. We have
\begin{lemma}\label{aAWlemma}
Let \eqref{modec} be a simply graded module over a $G$-graded algebra \eqref{eq:algebragrading}. Then for every $g\in$Supp$(W_G)$ the restriction \eqref{gres} yields an algebra isomorphism
\begin{equation}\label{eAW}
\text{End}^e_{A_G}(W_G)\cong\text{End}_{A_e}(W_g).
\end{equation}
\end{lemma}
The following is a graded version of Schur's Lemma.
\begin{lemma}\label{schurgr}(Graded Schur's Lemma \cite[Proposition 2.7.1(3)]{NVO04})
Let \eqref{modec} be a graded simple $A_G$-module. Then $\text{End}^{\text{r}(G)}_{A_G}(W_G)$ is a graded division algebra.
\end{lemma}
\begin{proof}
This is a consequence of the fact that the kernel and the image of a graded homomorphism are graded themselves.
\end{proof}
Graded modules over strongly graded algebras admit graded covers by free modules as follows.
\begin{lemma}\label{freest}
Let \eqref{modec} be a graded left module over a strongly $G$-graded algebra \eqref{eq:algebragrading}.
Then $W_G$ is a graded image $$\varphi_G:(A_G)^X\stackrel{G}{\twoheadrightarrow}W_G$$ of a free $A_G$-module $(A_G)^X$, admitting a graded module of relations ker$(\varphi_G)$.
\end{lemma}
\begin{proof}
By Lemma \ref{genst}, $W_G$ is generated by $W_e$. Let $X\subset W_e$ be a set of $e$-homogeneous generators for $W_G$ and let $(A_G)^X$ be a free $A_G$-module on the set $\{U_x\}_{x\in X}$.
Then the rule $U_x\mapsto x, \ \ x\in X$ determines a well-defined homomorphism of $A_G$-modules $\varphi_G:(A_G)^X\stackrel{}{\to}W_G$, which is graded and surjective.
Since the kernel of a graded homomorphism is graded, it follows that ker$(\varphi_G)$ is a graded submodule of $(A_G)^X$.
\end{proof}

\subsection{Absolutely graded simple modules}\label{absgsm}
Let \eqref{modec} be a graded module of a $G$-graded $\F$-algebra \eqref{eq:algebragrading}. Clearly, multiplication by any fixed element in $\F$ is a graded endomorphism of $W_G$
of trivial degree $e\in G$.
In the spirit of \S\ref{absec} we define
\begin{definition}\label{agsm}
A graded simple module \eqref{modec} of a $G$-graded $\F$-algebra \eqref{eq:algebragrading} is {\it absolutely $\F$-simply graded}
(or  {\it totally $\F$-simply graded}) if $\text{End}^{e}_{A_G}(W_G)=\F\cdot$Id$_{W_G}$.
In this case we write $[W_G]\in\text{Irr}^{\text{gr}}_{\F}(A_G).$
\end{definition}
\begin{lemma}\label{FWg}
Let $[W_G]\in\text{Irr}^{\text{gr}}_{\F}(A_G).$ Then $[W_g]\in$Irr$_{\F}(A_e)$ for every $g\in$Supp$_G(W_G)$ (and zero otherwise).
\end{lemma}
\begin{proof}
By Lemma \ref{grsimpsimp}, $W_g$ is simple as an $A_e$-module, whereas by Lemma \ref{aAWlemma} $\text{End}_{A_e}(W_g)\cong\text{End}^e_{A_G}(W_G)=\F\cdot$Id.
\end{proof}
The following result can be regarded as an alternative definition for absolutely graded simplicity.
\begin{lemma}\label{preobs}
Let \eqref{modec} be a graded simple left module of a $G$-graded algebra \eqref{eq:algebragrading}. Then $[W_G]\in\text{Irr}^{\text{gr}}_{\F}({A_G})$ iff
$\text{End}^{\text{r}(G)}_{A_G}(W_G)$ is a twisted group algebra over $\F$.
\end{lemma}
\begin{proof}
By Graded Schur's Lemma \ref{schurgr}, $\text{End}^{\text{r}(G)}_{A_G}(W_G)$ is a graded division algebra. By Lemma \ref{grdiv} it is a crossed product over its base algebra $\text{End}^{e}_{A_G}(W_G)$.
Suppose that $W_G$ is $\F$-absolutely graded simple. Then the base algebra $\text{End}^{e}_{A_G}(W_G)$ of this crossed product $\F$-algebra is $\F$ itself. By Definition \ref{deftga} we conclude that
$\text{End}^{\text{r}(G)}_{A_G}(W_G)$ is graded isomorphic to a twisted group algebra over $\F$.
The ``if" direction is straightforward from Definition \ref{agsm} of absolute $\F$-simplicity.
\end{proof}
Here is a seminal instance for the absolutely graded simple notion.
\begin{example}\label{absgsmex}
For any 2-cocycle $\alpha\in Z^2(G,\F^*)$, let $W_G:=(\F^{\alpha}G)^x$ be the free $\F^{\alpha}G$-module, generated by one element $x$. Then as a particular case of Example \ref{frendex} we have
\begin{equation}\label{endfree}\text{End}^{\text{r}(G)}_{A_G}((\F^{\alpha}G)^x)\stackrel{G}{\cong} \F^{\alpha}G.\end{equation}
Hence, $W_G$ is absolutely graded simple over $\F$.
\end{example}
The following result claims that the associated $G$-graded ${A_G}$-module (see \eqref{assogrd}) of an absolutely simple $A_e$-module is absolutely graded simple.
\begin{lemma}\label{grirrtga}
Let \eqref{eq:algebragrading} be a group-graded algebra and let $[M]\in$Irr$_{\F}(A_e)$. Then $[{A_G}\overline{\otimes}_{A_e}M]\in\text{Irr}^{\text{gr}}_{\F}({A_G}).$
\end{lemma}
\begin{proof}
Firstly, by Theorem \ref{indfunct}(2), ${A_G}\overline{\otimes}_{A_e}M$ is graded simple over $A_G$. Next, by Theorem \ref{indfunct}(1), $({A_G}\overline{\otimes}_{A_e}M)_e\cong M$ (in particular
$e\in$Supp$({A_G}\overline{\otimes}_{A_e}M)$) as $A_e$-modules.
Hence, by Lemma \ref{aAWlemma} $$\text{End}^e_{A_G}({A_G}\overline{\otimes}_{A_e}M)\cong\text{End}_{A_e}(({A_G}\overline{\otimes}_{A_e}M)_e)\cong\text{End}_{A_e}(M).$$
Since the $A_e$-module $M$ is absolutely simple over $\F$, then the above endomorphisms are just scalar multiplications by $\F$,
implying that the $A_G$-module ${A_G}\overline{\otimes}_{A_e}M$ is absolutely graded simple over $\F$.
\end{proof}

\section{Graded products}\label{actongral}
Two $G$-graded $\F$-spaces
\begin{equation}\label{twoalgebragrading}
\mathcal{G}:A_G=\bigoplus _{g\in G} A_g,\ \  \mathcal{G}':A'_G=\bigoplus _{g\in G} A'_g
\end{equation}
determine a new $G$-graded $\F$-space
\begin{equation}\label{Gprod}
\mathcal{G}\otimes_{\F}^G\mathcal{G}':A_G\otimes_{\F}^GA_G'=\bigoplus _{g\in G} A_g\otimes_{\F}A'_g\subseteq A_G\otimes_{\F} A'_G,
\end{equation}
namely the \textit{graded product} of $A_G$ and $A'_G$. If \eqref{twoalgebragrading} are $G$-graded algebras, then their graded product is a subalgebra of the tensor product
$A_G\otimes_{\F} A'_G$, and \eqref{Gprod} describes its $G$-graded algebra structure.
In particular, the base algebra of the $G$-graded algebra $A_G\otimes_{\F}^GA'_G$ is $A_e\otimes_{\F}A'_e$. The following claim is immediate.
\begin{lemma}\label{wdgp}
Let $\varphi:A_G\stackrel{G}\to B_G$ and $\varphi':A'_G\stackrel{G}\to B'_G$ be graded homomorphisms of $G$-graded $\F$-algebras. Then
$$\begin{array}{ccc}
A_G\otimes_{\F}^{\small{G}}A'_G&\to& B_G\otimes_{\F}^{\small{G}}B'_G\\
a_g\otimes_{\F}a'_g&\mapsto&\varphi(a_g)\otimes_{\F}\varphi'(a'_g)
\end{array}$$
is a graded homomorphism of the corresponding graded products. Consequently, graded products respect graded isomorphisms.
\end{lemma}
Let Gr$(G,\F)$=Gr$_{\kappa}(G,\F)$ be the set of graded-isomorphism classes of $G$-graded $\F$-algebras, whose cardinality is bounded by some infinite cardinality $\kappa$. By Lemma \ref{wdgp} we have
\begin{theorem}\label{thmgrprod}
The rule
$$\begin{array}{ccc}
\text{Gr}(G,\F)\times\text{Gr}(G,\F)&\to&\text{Gr}(G,\F)\\
([\mathcal{G}],[\mathcal{G}'])&\mapsto &[\mathcal{G}\otimes_{\F}^G\mathcal{G}']
\end{array}$$
determines a well-defined operation which
turns Gr$(G,\F)$ into an abelian monoid, whose identity is the graded class of the ordinary group algebra (or the $L^1$-space of the {\it group bundle}) $[\F G]$.
\end{theorem}
\subsection{Sub-monoids of Gr$(G,\F)$}\label{submonsec}
Let \eqref{twoalgebragrading} be two $G$-graded $\F$-algebras. Returning to the terminology of \S \ref{sgcpsect}, for every $g\in G$ there are two-sided ideals
$$\mathcal{I}_g=A_{g^{-1}}A_{g}\subseteq A_e,\ \ \ \ \mathcal{I}'_g=A'_{g^{-1}}A'_{g}\subseteq A'_e$$ of the corresponding base algebras. In turn, there is a two-sided ideal
$$(\mathcal{I}\otimes_{\F}^G\mathcal{I}')_g:=(A_{g^{-1}}\otimes_{\F}A'_{g^{-1}})(A_g\otimes_{\F}A'_g)\subseteq A_e\otimes_{\F}A'_e=(A_G\otimes_{\F}^GA'_G)_e$$
of the base algebra of the corresponding graded product.
It is not hard to show that
\begin{lemma}\label{strcomprod}
With the above notation, for every $g\in G$
$$(\mathcal{I}\otimes_{\F}^G\mathcal{I}')_g=\mathcal{I}_g\otimes_{\F}\mathcal{I}'_g.$$
In particular, with the notation \eqref{strcomp}
\begin{equation}\label{intersectrong}
(\mathcal{G}\otimes_{\F}^G\mathcal{G}')_{\text{str}}=\mathcal{G}_{\text{str}}\cap\mathcal{G}'_{\text{str}}.
\end{equation}
\end{lemma}
Equation \eqref{intersectrong} yields
\begin{corollary}\label{strprodcor}
The algebra $A_G\otimes_{\F}^GA_G'$ is strongly graded iff so are both algebras $A_G$ and $A_G'$.
\end{corollary}
From Corollary \ref{preserved} and from the ``if" direction of Corollary \ref{strprodcor} we obtain that the graded isomorphism classes of strongly $G$-graded $\F$-algebras form
a sub-monoid $\text{St}(G,\F)<\text{Gr}(G,\F)$.
From the ``only if" direction of Corollary \ref{strprodcor} we obtain that its complement $\text{Gr}(G,\F)\setminus\text{St}(G,\F)$ is a (prime) monoidal ideal of $\text{Gr}(G,\F)$.
We deduce that the sub-monoid $\text{St}(G,\F)$ admits a group of fractions $\overline{\text{St}(G,\F)}$, with
$$\overline{\text{St}(G,\F)}\cap\text{Gr}(G,\F)=\text{St}(G,\F).$$
Other important sub-monoids are introduced in the following claim.
\begin{lemma}\label{grprosgcp}
The grading classes of strongly $G$-graded $\F$-algebras, crossed products $\F$-algebras over $G$, skew group $\F$-algebras over $G$
and twisted group $\F$-algebras over $G$ are sub-monoids of $\text{Gr}(G,\F)$.
\end{lemma}
\begin{proof}
By direct computation.
\end{proof}
We focus in this section on two of the sub-monoids in Lemma \ref{grprosgcp}. Let us begin with twisted grading classes (Definition \ref{deftga}).
The grading classes $[\F^{\alpha}G]$ of twisted group algebras are exactly the invertible elements in $\text{Gr}(G,\F)$ as follows.
\begin{lemma}\label{invGr}
Let $\alpha,\alpha'\in Z^2(G,\F^*)$. Then
\begin{enumerate}
\item $[\F^{\alpha}G\otimes_{\F}^G\F^{\alpha'}G]=[\F^{\alpha\cdot\alpha'}G]$
\item $[\F^{\alpha}G]=[\F^{\alpha'}G]$ iff $[\alpha]=[\alpha']\in H^2(G,\F^*)$.
\item If $[\mathcal{G}\otimes_{\F}^G\mathcal{G}']=[\F G]$ then $\mathcal{G}$ and $\mathcal{G}'$ are twisted gradings.
\end{enumerate}
Consequently, the group of units $\text{Gr}(G,\F)^*$ is isomorphic to $H^2(G,\F^*)$.
\end{lemma}
\begin{proof}
The first claim is proven under the more general setting of Theorem \ref{pullback} in the sequel.
For the proof of (2) see, e.g., \cite[Prop. 2.4]{GS16}, whereas in order to prove (3) note that the homogeneous components of $G$-graded $\F$-algebras,
whose grading class is invertible, should be one-dimensional over $\F$ and admit homogeneous units. By Definition \ref{deftga} these are just the twisted group algebras.
\end{proof}
\begin{notation}\label{notsbmon}
We denote the sub-monoids of $\text{Gr}(G,\F)$, consisting of grading classes of skew group $\F$-algebras over $G$, crossed products $\F$-algebras over $G$, and strongly $G$-graded $\F$-algebras, by
\begin{equation}\label{sbmons}
\text{Sk}(G,\F)<\text{CP}(G,\F)<\text{St}(G,\F)
\end{equation}
respectively.
By Lemma \ref{invGr}, we identify the sub-monoid of $\text{Gr}(G,\F)$, consisting of grading classes of twisted group $\F$-algebras over $G$ with the cohomology group $H^2(G,\F^*)(<\text{CP}(G,\F))$.
Note that $$H^2(G,\F^*)\cap\text{Sk}(G,\F)=\{[\F G]\}=\{1\}.$$
\end{notation}
\subsection{Twisted graded algebras}\label{tgasec}
Quite a few notions of algebra twistings were posed in the literature (see, e.g. the symmetric twists for graded algebras over abelian groups \cite[Definition 3.1]{E18}). Here is our suggestion.
\begin{definition}\label{twistedalg}
The {\it twisted graded algebra} of \eqref{eq:algebragrading} by a 2-cocycle $\alpha\in Z^2(G,\F^*)$ is the graded product of \eqref{tga} and $\mathcal{G}$, that is
\begin{equation}\label{twist}
\alpha(\mathcal{G}):=\F^{\alpha} G\otimes_{\F}^G\mathcal{G},\ \ \alpha\in Z^2(G,\F^*),[\mathcal{G}]\in\text{Gr}(G,\F).
\end{equation}
\end{definition}

Owing to Lemma \ref{invGr}(2), we can think of $H^2(G,\F^*)$ as acting on $\text{Gr}(G,\F)$ via
\begin{equation}\label{Z2actgr}
[\alpha]([\mathcal{G}]):=[\alpha(\mathcal{G})]=[\F^{\alpha} G\otimes_{\F}^G\mathcal{G}]
\end{equation}
for every $[\alpha]\in H^2(G,\F^*)$ and $[\mathcal{G}]\in\text{Gr}(G,\F)$. The following claim is straightforward.
\begin{lemma}\label{basetwist}
With the notation \eqref{tga}, the rule
\begin{eqnarray}\label{samebase}
\begin{array}{ccl}
A_e&\to&\text{Span}_{\F}\{v_e\}\otimes_{\F}A_e\\
x&\mapsto& {v_e}\otimes_{\F}x
\end{array}
\end{eqnarray}
determines an isomorphism between the base algebras of $\mathcal{G}$ and its twisting ${\alpha}(\mathcal{G})$.
\end{lemma}
We thus regard classes in the same $H^2(G,\F^*)$-orbit in $\text{Gr}(G,\F)$ as having the same base algebra.
It is convenient (e.g., in order to prove Lemma \ref{grprosgcp}) to take $\{u_g\otimes_{\F}u'_g\}_{g\in G}$
as a set of invertible homogeneous elements for the crossed product obtained from the $G$-graded product of the crossed products
$A_e*G=\oplus_{g\in G}A_eu_g$ and $A'_e*G=\oplus_{g\in G}A_eu'_g$.
We record the special case of twisted crossed product in
\begin{lemma}\label{ugvg}
Let $\alpha(A_e*G):=\F^{\alpha} G\otimes_{\F}^GA_e*G$ be the graded product of a twisted group algebra
$\F^{\alpha} G:=\oplus_{g\in G}\text{Span}_{\F}\{v_g\}$ and a crossed product $A_e*G=\oplus_{g\in G}A_eu_g$.
Then \begin{equation}\label{twicp}\alpha(A_e*G)\stackrel{G}{\cong}\oplus_{g\in G}A_e(v_g\otimes_{\F} u_g),\end{equation}
in particular $\alpha(A_e*G)$ is a crossed product of $G$ over $A_e$.
\end{lemma}
Furthermore, the outer action \eqref{outer} of $G$ on the base algebra $A_e$ in the twisted crossed product \eqref{twicp} is given by the conjugation
\begin{equation}\label{outerA}
(v_g\otimes_{\F}u_g)(v_e\otimes_{\F}r)(v_g\otimes_{\F}u_g)^{-1}=v_e\otimes_{\F}u_gru_g^{-1}
\end{equation} for every $g\in G, r\in A_e$.
From Lemma \ref{ugvg} and equation \eqref{outerA} we infer
\begin{corollary}\label{sameactions}
With the identification \eqref{samebase}, crossed products in a $Z^2(G,\F^*)$-orbit admit the same $G$-outer action on $A_e$.
\end{corollary}

\subsection{Graded products of quotient graded twisted group algebras}\label{secpb}
Here is an appealing family of examples for graded products.
Let
\begin{equation}\label{beta12}
\beta:1\to N\xrightarrow{\epsilon} \Gamma \xrightarrow{\pi} G\to 1\text{   and } \ \beta':1\to N'\xrightarrow{\epsilon'} \Gamma' \xrightarrow{\pi'} G\to 1
\end{equation}
be two group extensions of $G$ by $N$ and $N'$ respectively.
Evidently, there is no sense in a graded product of a $\Gamma$-graded $\F$-algebra $A_{\Gamma}$ and a $\Gamma'$-graded $\F$-algebra $A_{\Gamma'}$.
However, their quotient $G$-gradings under the graded quotient homomorphisms $(1_{A_{\Gamma}},\pi)$ and $(1_{A_{\Gamma'}},\pi')$ respectively (see \eqref{quotgrad}) do admit a $G$-graded product,
whose base algebra is $A_N\otimes_{\F}A_{N'}$. When $A_{\Gamma}$ and  $A_{\Gamma'}$ are twisted group algebras, their $G$-graded product is described using the pull-back
\begin{equation}\label{pb}\Gamma\times_G\Gamma':=\{(\gamma,\gamma')\in\Gamma\times\Gamma'|\ \ \pi(\gamma)=\pi'(\gamma')\}.\end{equation}
\begin{theorem}\label{pullback}
Let \eqref{beta12} be two group extensions of $G$, and let $c\in Z^2(\Gamma,\F^*)$ and $c'\in Z^2(\Gamma',\F^*)$. Then
$$\F^{c}\Gamma\otimes^G_{\F}\F^{c'}\Gamma'\stackrel{G}{\cong}\F^{c\times_G c'}(\Gamma\times_G\Gamma'),$$
where $c\times_G c'\in Z^2(\Gamma\times_G\Gamma',\F^*)$ is given by
$$(c\times_G c')((\gamma_1,\gamma_1'),(\gamma_2,\gamma_2')):=
c(\gamma_1,\gamma_2)\cdot c'(\gamma_1',\gamma_2'),\ \forall (\gamma_1,\gamma_1'),(\gamma_2,\gamma_2')\in\Gamma\times_G\Gamma'.$$
\end{theorem}
\begin{proof}
Let $\{u_{\gamma}\}_{\gamma\in \Gamma}$ and $\{u'_{\gamma'}\}_{\gamma'\in \Gamma'}$ be $\F$-bases of invertible homogeneous elements for $\F^{c} \Gamma$ and $\F^{c'} \Gamma'$ respectively.
It is not hard to verify that
$$\left\{u_{\gamma}\otimes_{\F}u'_{\gamma'}|\ \pi(\gamma)=\pi'(\gamma')\right\}_{}=\left\{u_{\gamma}\otimes_{\F}u'_{\gamma'}\right\}_{(\gamma,\gamma')\in\Gamma\times\Gamma'}$$
is an $\F$-basis for $\F^{c}\Gamma\otimes^G_{\F}\F^{c'}\Gamma'$. The product of two such basis elements gives
$$\begin{array}{rl}
(u_{\gamma_1}\otimes_{\F}u'_{\gamma_1'})\cdot (u_{\gamma_2}\otimes_{\F}u'_{\gamma'_2})&=u_{\gamma_1}\cdot u_{\gamma_2}\otimes_{\F}u'_{\gamma_1'}\cdot u'_{\gamma'_2}\\ &=
c(\gamma_1,\gamma_2)\cdot u_{\gamma_1\cdot\gamma_2}\otimes_{\F}c'(\gamma'_1,\gamma'_2)\cdot u'_{\gamma_1'\cdot\gamma'_2}\\
 &=c(\gamma_1,\gamma_2)\cdot c'(\gamma'_1,\gamma'_2)\cdot ( u_{\gamma_1\cdot\gamma_2}\otimes_{\F}u'_{\gamma_1'\cdot\gamma'_2})\\
  &=(c\times_G c')((\gamma_1,\gamma_1'),(\gamma_2,\gamma_2'))\cdot ( u_{\gamma_1\cdot\gamma_2}\otimes_{\F}u'_{\gamma_1'\cdot\gamma'_2})
\end{array}$$
for every $(\gamma_1,\gamma_1'),(\gamma_2,\gamma_2')\in\Gamma\times_G\Gamma'.$ This proves the claim.
\end{proof}

Let $\F^c\Gamma$ be a twisted group algebra, graded by a quotient $G$ of $\Gamma$ (see \eqref{quotgrdef}). How does the twisting of this $G$-graded algebra by
$\alpha\in Z^2(G,\F^*)$ look like?
Recall that with the notation \eqref{beta12} the \textit{inflation} of $\alpha$ is a 2-cocycle of $\Gamma$ defined by
$$\text{inf}^G_{\Gamma}(\alpha)(\gamma_1,\gamma_2):=\alpha(\pi(\gamma_1),\pi(\gamma_2)).$$
It turns out that the twisting $\alpha(\F^c\Gamma)$ is $G$-graded isomorphic to the twisted group algebra $\F^{c''}\Gamma$,
where $c,c''\in Z^2(\Gamma,\F^*)$ differ by the inflation of $\alpha$ from $G$ to $\Gamma$.
More explicitly,
\begin{corollary}\label{formackey}
Let $c\in Z^2(\Gamma,\F^*)$ and $\alpha\in Z^2(G,\F^*)$, where $G=\Gamma/N$. Then
$$\alpha(\F^c\Gamma)\stackrel{G}{\cong}\F^{c\cdot\inf^G_{\Gamma}(\alpha)}\Gamma.$$
\end{corollary}
\begin{proof}
This is a consequence of Theorem \ref{pullback}, putting $\Gamma':=G$ with $\pi':=$Id$(G)$.
In this case $$\Gamma\times_GG:=\{(\gamma,\pi(\gamma))\}_{\gamma\in\Gamma}<\Gamma\times G$$ is identified with $\Gamma$ by sending the pairs $(\gamma,\pi(\gamma))$ to $\gamma$
for every $\gamma\in\Gamma$. Furthermore, with the above notation we have
$$(c\times_G \alpha)((\gamma_1,\pi(\gamma_1)),(\gamma_2,\pi(\gamma_2)))=
c(\gamma_1,\gamma_2)\cdot \alpha(\pi(\gamma_1),\pi(\gamma_2))=c\cdot\text{inf}^G_{\Gamma}(\alpha)(\gamma_1,\gamma_2)$$
for every $\gamma_1,\gamma_2\in\Gamma$, verifying the result.
\end{proof}

Note that Lemma \ref{invGr}(1) can be deduced from Theorem \ref{pullback}, or more particularly from its Corollary \ref{formackey} by putting $\Gamma=\Gamma'=G$.

\subsection{Action of Aut$(G)$ on Gr$(G,\F)$}\label{autsec}
Let \eqref{eq:algebragrading} be a $G$-graded $\F$-algebra.
Then every $\theta\in\Aut(G)$ gives rise to a $G$-graded $\F$-algebra $\theta(A_G)$, which is the same as $A_G$ as an ungraded $\F$-algebra, whereas its
$G$-grading is a relabelling of the homogeneous components of $A_G$ with respect to $\theta$, that is
$$\theta(A_G):=\bigoplus _{g\in G} \theta(A)_g,\ \text{ where }  \theta(A)_g:=A_{\theta^{-1}(g)} \text{ for every $g\in G$}.$$
Clearly,
\begin{equation}\label{suppaut}
\Supp_G(\theta(A_G))=\theta(\Supp_G(A_G)).
\end{equation}
\begin{lemma}\label{autact}
The rule
$$[A_G]\mapsto [\theta(A_{G})],\ \ \theta\in\Aut(G)$$
determines a well-defined left action of Aut$(G)$ on Gr$(G,\F)$ as abelian monoid automorphisms.
This action stabilizes the subgroup of units $H^2(G,\F^*)<\Gr(G,\F^*)$ as well as the submonoids \eqref{sbmons}. In particular, the above rule determines
an action of $G$ on Gr$(G,\F)$ via the conjugation automorphism
$$g([A_G]):=\iota_g([A_G]),\ \ g\in G.$$
\end{lemma}
\begin{proof}
As can easily be verified, for every $G$-graded $\F$-algebras
\eqref{twoalgebragrading} and $\theta,\theta'\in\Aut(G)$
\begin{enumerate}
\item $(\theta\circ \theta')(A_G)=\theta(\theta'(A_G)).$
\item $\theta(A_G\otimes_{\F}^{\small{G}}A'_G)=\theta(A_G)\otimes_{\F}^{\small{G}}\theta(A'_G).$
\item if $[A_G]=[A'_G]$ then $[\theta(A_G)]=[\theta(A'_G)]$.
\item the graded algebra $\theta(A_G)$ inherits the properties of $A_G$ of being strongly graded, a crossed product,
as well as being a twisted or a skew group algebra.
\end{enumerate}
Lemma \ref{autact} is a direct consequence of the above facts.
\end{proof}
Note that the action of Aut$(G)$ on the subgroup of units $H^2(G,\F^*)<\Gr(G,\F^*)$ is given by
\begin{equation}\label{eq:action}
\begin{aligned}
\theta([f])&=[{\theta}(f)],\\ 
{\theta}(f)(g_1,g_2):&=f({\theta}^{-1}(g_1),{\theta}^{-1}(g_2))
\end{aligned}
, [f]\in H^2(G,\C^*),\ \ \theta\in\Aut(G).
\end{equation}
It therefore coincides with case $n=2$ of the well-known action of $\Aut(G)$ on the $n$-th cohomology group $H^n(G,\F^*)$ for any non-negative integer $n$.
More generally, fixing any subgroup $H<G$, then any $\theta\in\Aut(G)$ and $[f]\in H^2(H,\F^*)$ give rise to a class $[{\theta}(f)]\in H^2({\theta}(H),\F^*)$ where
\begin{equation}\label{H<G}
{\theta}(f)(k_1,k_2):=f({\theta}^{-1}(k_1),{\theta}^{-1}(k_2)),\ \ k_1,k_2\in \theta(H).
\end{equation}
Equation \eqref{H<G} is used in \S\ref{pfB}.
\subsection{Graded products of graded modules}\label{WWsec}
Let \begin{equation}\label{WW'}
W_G=\bigoplus_{g\in G}W_g, \ \ W'_G=\bigoplus_{g\in G}W'_g\end{equation}
be $G$-graded left modules over $G$-graded $\F$-algebras \eqref{twoalgebragrading} respectively.
\begin{definition}\label{gmodprod} The \textit{graded product} of the graded modules \eqref{WW'} is their $\F$-space graded product (see \eqref{Gprod})
$$W_G\otimes_{\F}^GW'_G=\bigoplus_{g\in G}(W_G\otimes_{\F}^GW_G')_g\subseteq W_G\otimes_{\F}W'_G,\ \ (W_G\otimes_{\F}^GW'_G)_g:=W_g\otimes_{\F}W'_g,$$
furnished with a left $A_G\otimes_{\F}^GA'_G$-module structure
\begin{equation}\label{modstr}
(a_g\otimes_{\F}a'_g)(w_h\otimes_{\F}w'_h):=a_gw_h\otimes_{\F}a'_gw'_h\in (W_G\otimes_{\F}^GW'_G)_{gh},
\end{equation}
for every $a_g\otimes_{\F}a'_g\in (A_G\otimes_{\F}^GA'_G)_g$ and $w_h\otimes_{\F}w'_h\in (W_G\otimes_{\F}^GW'_G)_h.$
\end{definition}
The following is a straightforward analogue of Lemma \ref{wdgp} for graded modules.
\begin{lemma}\label{wdmgp}
Let $\varphi:W_G\stackrel{G}\to \widetilde{W_G}$ and $\varphi':W'_G\stackrel{G}\to \widetilde{W'_G}$ be graded homomorphisms of $G$-graded left modules over $G$-graded $\F$-algebras
\eqref{eq:algebragrading} respectively. Then
$$\begin{array}{ccc}
W_G\otimes_{\F}^{\small{G}}W'_G&\to& \widetilde{W_G}\otimes_{\F}^{\small{G}}\widetilde{W'_G}\\
w_g\otimes_{\F}w'_g&\mapsto&\varphi(w_g)\otimes_{\F}\varphi'(w'_g)
\end{array}$$
is a graded homomorphism of the corresponding graded module products over the graded product $A_G\otimes_{\F}^{\small{G}}A'_G$.
Consequently, graded products of graded modules respect graded isomorphisms.
\end{lemma}

Graded products of modules respect direct sums. In particular, if $(A_G)^X$ is a free $A_G$-module on a set $X$ of generators and $(A'_G)^{X'}$ is a free $A'_G$-module on a set $X'$ of generators, then
\begin{equation}\label{grfree}
(A_G)^X\otimes_{\F}^G(A'_G)^{X'}\stackrel{G}{\cong}(A_G\otimes_{\F}^GA'_G)^{X\times X'}
\end{equation}
as free $A_G\otimes_{\F}^GA'_G$-modules.

\begin{lemma}\label{fgfp}
Let \eqref{WW'} be finitely-generated $G$-graded left modules over strongly $G$-graded $\F$-algebras \eqref{twoalgebragrading} respectively.
Then the graded product $W_G\otimes_{\F}^GW'_G$ is finitely-generated over $A_G\otimes_{\F}^GA'_G$. If \eqref{WW'} are further assumed to be finitely-presented, then so is $W_G\otimes_{\F}^GW'_G$.
\end{lemma}
\begin{proof}
%

Since the graded modules \eqref{WW'} are finitely-generated over strongly graded algebras, then by Lemma \ref{freest} both are graded images
\begin{equation}\label{grimages}
\varphi_G:(A_G)^X\stackrel{G}{\twoheadrightarrow}W_G,\ \ \varphi'_G:(A'_G)^{X'}\stackrel{G}{\twoheadrightarrow}W'_G
\end{equation}
of the finitely-generated free $A_G$-module $(A_G)^X$ and of the finitely-generated free $A_G$-module $(A'_G)^{X'}$ respectively.
Then there is a surjective $A_G\otimes_{\F}^GA'_G$-module homomorphism
\begin{eqnarray}\label{grimprod}
\varphi_G'':\begin{array}{ccc}
(A_G)^X\otimes_{\F}^G(A'_G)^{X'}&\stackrel{G}{\twoheadrightarrow}&W_G\otimes_{\F}^GW'_G\\
y_g\otimes_{\F}y'_g&\mapsto & \varphi_G (y_g)\otimes_{\F}\varphi'_G(y'_g)
\end{array},
\end{eqnarray}
where $y_g\in (A_G)^X_g$ and $y'_g\in (A'_G)^{X'}_g$.
We infer that the graded product $W_G\otimes_{\F}^GW'_G$ is finitely-generated being a graded image under \eqref{grimprod}
of the finitely-generated free $A_G\otimes_{\F}^GA'_G$-module (see \eqref{grfree})
$$(A_G)^X\otimes_{\F}^G(A'_G)^{X'}\stackrel{G}{\cong}(A_G\otimes_{\F}^GA'_G)^{X\times X'},$$
proving the first part of the lemma.

What is the kernel of the graded homomorphism \eqref{grimprod}?
Since $A_G$ and $A'_G$ are both strongly graded, then by Corollary \ref{strprodcor} so is their graded product $A_G\otimes_{\F}^GA'_G$.
Lemma \ref{genst} tells us then that the $A_G\otimes_{\F}^GA'_G$-module ker$(\varphi_G'')$ is generated by its base algebra.
This $e$-component is just ker$(\varphi_e'')$, where
$$\varphi_e'':((A_G)^X\otimes_{\F}^G(A'_G)^{X'})_e\stackrel{}{\twoheadrightarrow}W_e\otimes_{\F}W'_e$$
is the restriction of $\varphi_G''$ to the $e$-component
$$((A_G)^X\otimes_{\F}^G(A'_G)^{X'})_e=(A_e)^X\otimes_{\F}(A'_e)^{X'}.$$
Since $\varphi_e''$ is an (ungraded) homomorphism of $A_e\otimes_{\F}A'_e$-modules, then by \eqref{kertens} we deduce that
\begin{equation}\label{kervarphi''}
\ker(\varphi_e'')\stackrel{}{=}\ker(\varphi_e)\otimes_{\F}(A'_e)^{X'}+(A_e)^X\otimes_{\F}\ker(\varphi'_e),
\end{equation}
where
$$\varphi_e:(A_e)^X\stackrel{G}{\twoheadrightarrow}W_e,\text{ and } \ \varphi'_e:(A'_e)^{X'}\stackrel{G}{\twoheadrightarrow}W'_e$$
are the corresponding (ungraded) restrictions of the graded homomorphisms \eqref{grimages}.
Suppose now that the modules \eqref{WW'} are also finitely-presented. Then all their finitely-generated covers, in particular those given in \eqref{grimages}, admit finitely generated kernels
(see \cite[Lemma 9, p. 21]{Bourbaki72}). Since both homomorphisms \eqref{grimages} are graded, then $\ker(\varphi_G)$ and $\ker(\varphi'_G)$, which were just proven to be finitely-generated,
are graded modules.
Once again by Lemma \ref{genst} and Lemma \ref{fgenst}, these kernels are generated over $A_G$ and $A'_G$ by finite subsets of $\ker(\varphi_e)$ and $\ker(\varphi'_e)$ respectively.
In particular, $\ker(\varphi_e)$ and $\ker(\varphi'_e)$ are themselves finitely-generated over $A_e$ and $A'_e$ respectively, and by
\eqref{kervarphi''}, we deduce that ker$(\varphi_e'')$ is finitely-generated over $A_e\otimes_{\F}A'_e$.
Finally, since ker$(\varphi_e'')$ generates ker$(\varphi_G'')$ over $A_G\otimes_{\F}^GA'_G$, we conclude that \eqref{grimprod} is a finite presentation of $W_G\otimes_{\F}^GW'_G$.
This settles the second claim.\end{proof}

A special case of graded product of graded modules occurs when one of the graded modules is $\F^{\alpha}G$ as a free module over itself.
\begin{definition}\label{tgm}
Let \eqref{modec} be a graded module over a $G$-graded $\F$-algebra \eqref{eq:algebragrading}.
A \textit{twisting} $\alpha(W_G)$ of $W_G$ by a 2-cocycle $\alpha\in Z^2(G,\F^*)$ is
the graded module product $\F^{\alpha}G\otimes^G_{\F}W_G$
over the twisted graded algebra $\F^{\alpha}G\otimes^G_{\F}A_G=\alpha(A_G)$.
\end{definition}
Let
$$\F^{\alpha}G=\oplus_{g\in G}\text{Span}_{\F}\{v_g\},\ \ \F^{\alpha^{-1}}G=\oplus_{g\in G}\text{Span}_{\F}\{\hat{v}_g\}.$$
We record for a later use that the rule
\begin{equation}\label{therule}
\chi_g:v_g\otimes_{\F}\hat{v}_g \otimes_{\F}w_g\mapsto w_g,\ \ w_g\in W_g,\ \  g\in G
\end{equation}
determines a graded isomorphism
$$\chi:\F^{\alpha}G\otimes^G_{\F}(\F^{\alpha^{-1}}G\otimes^G_{\F}A_G)=\alpha(\alpha^{-1}(W_G))\xrightarrow{G}W_G.$$
As can easily be checked, if $W'_G$ is a graded $\alpha^{-1}(A_G)$-module, then for every $w'_g\in W'_G$
\begin{equation}\label{rule1}
\hat{v}_g\otimes_{\F}\chi_g(v_g \otimes_{\F}w'_g)=w'_g.
\end{equation}

The next three claims are straightforward.
\begin{lemma}\label{inductwist}
Let \eqref{eq:algebragrading} be a $G$-graded $\F$-algebra, let $M$ be an $A_e$-module, and let $\alpha\in Z^2(G,\F^*)$ be a 2-cocycle.
Then there is a graded isomorphism
\begin{equation}\label{inductwisteq}
\alpha(A_G){\otimes}_{A_e}M^{}\stackrel{G}\cong\alpha({A_G}{\otimes}_{A_e}M)\end{equation}
of $\alpha(A_G)$-modules.
\end{lemma}
\begin{lemma}\label{quotwist}
Let $W'_G$ be a graded submodule of \eqref{modec}. Then for any 2-cocycle $\alpha\in Z^2(G,\F^*)$, the twisted graded module $\alpha(W'_G)$ is a graded $\alpha(A_G)$-submodule of $\alpha(W_G)$.
Furthermore, with the notation \eqref{tga} the rule
$$(v_g\otimes_{\F}w_g)+\alpha(W'_G)\mapsto v_g\otimes_{\F}(w_g+W'_G),\ \ w_g\in W_g$$ determines a well-defined graded isomorphism
\begin{equation}\label{equotwist}
\bigslant{\alpha(W_G)}{\alpha(W'_G)}\stackrel{G}{\cong}\alpha(\bigslant{W_G}{W'_G}).\end{equation}
\end{lemma}
\begin{lemma}\label{cortwist}
Let \eqref{modec} be a graded module over a $G$-graded $\F$-algebra \eqref{eq:algebragrading} and let $\alpha\in Z^2(G,\F^*)$ be a 2-cocycle.
Then $W_G\mapsto\alpha(W_G)$
determines a bijective correspondence between graded $A_G$-modules and graded $\alpha(A_G)$-modules, which preserves inclusion and hence graded simplicity.
Consequently, the localizing radicals (see \S\ref{sgcpsect}) satisfy
\begin{equation}\label{eqlocrad}
t_{\mathcal{C}_e}(\alpha(W_G))=\alpha(t_{\mathcal{C}_e}(W_G)).
\end{equation}
\end{lemma}
Notice that under the identification \eqref{samebase} of the base algebras of $A_G$ and its twisting $\alpha(A_G)$, the homogeneous components of $W_G$ and $\alpha(W_G)$ are isomorphic as $A_e$-modules.
With the notation \eqref{assogrd} we have
\begin{corollary}\label{assoctwwist}
Let \eqref{eq:algebragrading} be a $G$-graded $\F$-algebra, let $M$ be an $A_e$-module, and let $\alpha\in Z^2(G,\F^*)$ be a 2-cocycle.
Then there is an isomorphism
$$\alpha(A_G)\overline{\otimes}_{A_e}M^{}\stackrel{G}\cong\alpha({A_G}\overline{\otimes}_{A_e}M)$$
of graded $\alpha(A_G)$-modules.
\end{corollary}
\begin{proof}
By the definition of the associated graded module \eqref{assogrd} and by Lemma \ref{inductwist} we have
$$\begin{array}{cl}
\alpha(A_G)\overline{\otimes}_{A_e}M^{}&=\bigslant{\alpha({A_G}){\otimes}_{A_e}M}{t_{\mathcal{C}_e}(\alpha({A_G}){\otimes}_{A_e}M)}\\
&\stackrel{G}\cong\bigslant{\alpha({A_G}{\otimes}_{A_e}M)}{t_{\mathcal{C}_e}(\alpha({A_G}{\otimes}_{A_e}M))}=\cdots
\end{array}$$
By \eqref{equotwist} and \eqref{eqlocrad} we establish
$$\cdots=\bigslant{\alpha({A_G}{\otimes}_{A_e}M)}{\alpha(t_{\mathcal{C}_e}({A_G}{\otimes}_{A_e}M))}\stackrel{G}\cong
\alpha(\bigslant{{A_G}{\otimes}_{A_e}M}{t_{\mathcal{C}_e}({A_G}{\otimes}_{A_e}M)})
=\alpha({A_G}\overline{\otimes}_{A_e}M),$$
proving the claim.
\end{proof}
Under the graded product $W_G\otimes^G_{\F}W'_G$, a pair of graded submodules $S_G\subseteq W_G$ and $S'_G\subseteq W'_G$ is mapped to a graded  $A_G\otimes^G_{\F}A'_G$-submodule
$S_G\otimes^G_{\F}S'_G\subseteq W_G\otimes^G_{\F}W_G'$.
When $W_G$ is absolutely graded simple (see Definition \ref{agsm}) over $\F$ then this map is onto as shown in the following graded generalization of Theorem \ref{tensor}.
\begin{theorem}\label{grtensor}
Let $[W_G]\in\text{Irr}^{gr}_{\F}(A_G)$ and let $W'_G$ be any graded left $A'_G$-module. Then for every graded $A_G\otimes^G_{\F}A'_G$-submodule $S''_G\subseteq W_G\otimes^G_{\F}W_G'$
there exists a graded $A'_G$-submodule $S'_G\subseteq W'_G$ such that $S''_G=W_G\otimes^G_{\F}S'_G$.
In particular, if $W'_G$ is graded simple then $W_G\otimes^G_{\F}W'_G$ is either graded simple or $\{0\}$.
\end{theorem}
\begin{proof}
Given $S''_G\subseteq W_G\otimes^G_{\F}W'_G$, let $S'_G\subseteq W'_G$ be the graded left $A'_G$-module generated by the \textit{$W'_G$-content} of $S''_G$. More precisely,
$S'_G:=A'_G(\oplus_{g\in G}S'_g)$, where
$$S'_g:=\left\{{w'}_{g}^i\in W'_g\ \ | \ \ \exists \sum_{i=1}^nw_g^i\otimes_{\F}{w'}_g^i\in S''_g, \text{written in a shortest way}\right\}.$$
Clearly, every homogeneous element of $S''_G$ lies in $W_G\otimes^G_{\F}S'_G$, hence $S''_G\subseteq W_G\otimes^G_{\F}S'_G$.
For the converse inclusion, it is enough to prove that
\begin{equation}\label{qm}
w_{hg}\otimes_{\F}x'_h(w'_g)\stackrel{?}{\in} S''_G.\ \ \ \forall w_{hg}\in W_{hg},x'_h\in A'_h,w'_g\in S'_g.\end{equation}
Indeed, let $w_{gh}\in W_{gh},x'_h\in A'_h$ and $w_g'\in S'_g$.
By the definition of $S'_g$, there exist an integer $n$ and two sets of homogeneous elements $$\{w_g^1,\cdots,w_g^n\}\subseteq W_g,\{{w'}_g^1=w_g',\cdots,{w'}_g^n\}\subseteq W'_g,$$
both of which linearly independent over $\F$, such that
$\sum_{i=1}^nw_g^i\otimes_{\F}{w'}_g^i\in S''_G.$
Recall that since $W_G$ is absolutely graded simple, then End$^{\text{r}(G)}_{A_G}(W_G)$
is a twisted group algebra over $\F$ (see Lemma \ref{preobs}),
hence by Lemma \ref{linind}, the $\F$-independence of
$w_g^1,\cdots,w_g^n$ implies that these homogeneous elements are linearly
independent also over End$^{\text{r}(G)}_{A_G}(W_G)$. From the Graded Density Theorem
\cite[Theorem 2.5]{EK13} we deduce that there exists $x\in A_e$
such that
$$xw^1_g=w^1_g,\text{   and} \ \ xw_g^i=0,\ \  i\geq 2,$$
and therefore
\begin{equation}\label{shortest}
w^1_g\otimes_{\F}{w'}_g=w^1_g\otimes_{\F}{w'}_g^1=\sum_{i=1}^nxw_g^i\otimes_{\F}{w'}_g^i=(x\otimes_{\F}1)(\sum_{i=1}^nw_g^i\otimes_{\F}{w'}_g^i)\in S''_g.
\end{equation}
Now, being graded-simple, $W_G$ is generated in particular by the homogeneous element $w^1_g\in W_g$. Thus, there exists $x_h\in A_h$ such that
\begin{equation}\label{hghg}
w_{hg}=x_h(w^1_g).\end{equation}
By \eqref{shortest} and \eqref{hghg} we conclude that
$$w_{hg}\otimes_{\F}x'_h(w'_g)=x_h(w^1_g)\otimes_{\F}x'_h(w'_g)=(x_h\otimes_{\F}{x'}_h)(w^1_g\otimes_{\F}{w'}_g){\in} S''_G.$$
verifying \eqref{qm}. 
The second claim of the theorem follows immediately.
\end{proof}
\begin{theorem}\label{irrsgp}
Let $[W_G]\in\text{Irr}^{\text{gr}}_{\F}(A_G)$ and $[W'_G]\in\text{Irr}^{\text{gr}}_{\F}(A'_G),$ such that $W_G\otimes_{\F}^GW'_G$ is non-zero. Then
~$[W_G\otimes_{\F}^GW'_G]\in\text{Irr}^{\text{gr}}_{\F}(A_G\otimes_{\F}^GA'_G)$.
\end{theorem}
\begin{proof}
By the second claim of Theorem \ref{grtensor}, $W_G\otimes^G_{\F}W'_G$ is graded simple over $A_G\otimes^G_{\F}A'_G$.
Next, let
$$g\in\text{Supp}_G(W_G\otimes^G_{\F}W'_G)=\text{Supp}_G(W_G)\cap\text{Supp}_G(W'_G).$$ By Theorem \ref{FWg},
$[W_g]\in\text{Irr}_{\F}(A_e)$ and $[{W'}_g]\in\text{Irr}_{\F}(A'_e).$
Theorem \ref{tensor}(1) says that in this case
\begin{equation}\label{eq2.4}
\text{End}_{A_e\otimes_{\F}A'_e}(W_g\otimes_{\F}{W'}_g)\cong\End_{A_e}(W_g)\otimes_{\F}\End_{A'_e}(W'_g)\cong\F.
\end{equation}
Now, by Lemma \ref{aAWlemma}
$$\text{End}_{A_G\otimes^G_{\F}A'_G}^e(W_G\otimes^G_{\F}{W'_G})\cong\text{End}_{(A_G\otimes^G_{\F}A'_G)_e}((W_G\otimes^G_{\F}{W'_G})_g)=\text{End}_{A_e\otimes_{\F}A'_e}(W_g\otimes_{\F}{W'}_g),$$
which, together with \eqref{eq2.4} and the Definition \ref{agsm},
prove that $W_G\otimes^G_{\F}{W'_G}$ is absolutely $\F$-simply graded over $A_G\otimes^G_{\F}A'_G$.
\end{proof}
Let \eqref{modec} be a graded left module over a $G$-graded $\F$-algebra \eqref{eq:algebragrading} and let $\alpha(W_G)$ be its
twisting by $\alpha\in Z^2(G,\F^*)$ (see Definition \ref{tgm})
over the twisted graded algebra $\alpha(A_G)$.
The right graded endomorphisms End$^{\text{r}(G)}_{\alpha(A_G)}\alpha(W_G)$ can easily be obtained from End$^{\text{r}(G)}_{A_G}(W_G)$ as follows.
For every right graded endomorphism $\phi_g\in$ End$^{\text{r}(g)}_{A_G}(W_G)$ of degree $g\in G$, use the notations \eqref{tga} and \eqref{twist} to set
\begin{eqnarray}\label{twistendo}
\alpha(\phi_g):\begin{array}{ccc}
\alpha(W_G)&\to &\alpha(W_G)\\
\sum_{h\in G}v_h\otimes_{\F}w_h&\mapsto& \sum_{h\in G}\alpha(h,g)\cdot v_{hg}\otimes_{\F}(w_h)\phi_g
\end{array}.\end{eqnarray}
\begin{lemma}\label{twistendolemma}
Let \eqref{modec} be a $G$-graded left module of a $G$-graded algebra \eqref{eq:algebragrading} and let $\alpha\in Z^2(G,\F^*)$.
Then with the notation \eqref{twistendo}\begin{enumerate}
\item For every $g\in G$, $\alpha(\phi_g)\in$End$^{\textmd{r}(g)}_{\alpha(A_G)}\alpha(W_G)$.
\item For every $g,h\in G$, $\alpha(\phi_g)\circ\alpha(\phi_h)=\alpha(g,h)\cdot\alpha(\phi_g\circ\phi_h)$.
\end{enumerate}
\end{lemma}
\begin{proof}
We first prove that elements from $\alpha(A_G)$ can be pulled out of $\alpha(\phi_g)$.
Evidently, it is enough to check for homogeneous elements $v_k\otimes_{\F}a_k\in[\alpha(A_G)]_k,\ \ k\in G$.
Indeed, by the definition \eqref{twistendo} for every $v_h\otimes_{\F}w_h\in [\alpha(W_G)]_h$ we have
$$\begin{array}{rl}
[(v_k\otimes_{\F}a_k)(v_h\otimes_{\F}w_h)]\alpha(\phi_g)&=(\alpha(k,h)\cdot v_{kh}\otimes_{\F}a_kw_h)\alpha(\phi_g)\\
 &=\alpha(k,h)\cdot\alpha(kh,g)\cdot v_{khg}\otimes_{\F}(a_kw_h)\phi_g=\cdots
\end{array}$$
Applying the 2-cocycle condition and using the definitions once again, bearing in mind that $\phi_g$ is an $A_G$-module endomorphism of $W_G$ we obtain
$$\begin{array}{rl}
\cdots&=\alpha(k,hg)\cdot\alpha(h,g)\cdot v_{khg}\otimes_{\F}(a_kw_h)\phi_g=(v_k\otimes_{\F}a_k)(\alpha(h,g)\cdot v_{hg}\otimes_{\F}(w_h)\phi_g)\\
 &=(v_k\otimes_{\F}a_k)[(v_{h}\otimes_{\F}w_h)\alpha(\phi_g)],
\end{array}$$
proving that $\alpha(\phi_g)$ is an $\alpha(A_G)$-module endomorphism of $\alpha(W_G)$. The fact that $\alpha(\phi_g)$ is right homogeneous of degree $g$ is clear. This proves (1).
To prove the second part of the lemma, note that $\phi_g\circ\phi_h\in$ End$^{\text{r}(gh)}_{A_G}(W_G)$,
and so by the first claim of this lemma $$\alpha(\phi_g\circ\phi_h)\in\End^{\text{r}(gh)}_{\alpha(A_G)}(W_G).$$
Now, compare the action of both sides of (2) on a homogeneous element $v_{k}\otimes_{\F}w_k\in [\alpha(W_G)]_k$.
$$\begin{array}{rl}
[(v_{k}\otimes_{\F}w_k)\alpha(\phi_g)]\alpha(\phi_h)&=\alpha(k,g)\cdot(v_{kg}\otimes_{\F}(w_k)\phi_g)\alpha(\phi_h)\\
 &=\alpha(k,g)\cdot\alpha(kg,h)\cdot v_{kgh}\otimes_{\F}((w_k)\phi_g)\phi_h=\cdots
\end{array}$$
Apply the 2-cocycle condition to obtain
$$\cdots=\alpha(g,h)\cdot\alpha(k,gh)\cdot v_{kgh}\otimes_{\F}(w_k)(\phi_g\circ\phi_h)=\alpha(g,h)\cdot (v_{k}\otimes_{\F}w_k)\alpha(\phi_g\circ\phi_h).$$
Both endomorphisms agree on homogeneous elements and so are the same.
\end{proof}
\begin{corollary}\label{endtwist}
With the above notation, the map
$$\alpha_W:\begin{array}{ccc}
\alpha(End^{\textmd{r}(G)}_{A_G}(W_G))&\stackrel{G}\rightarrow& End^{\textmd{r}(G)}_{\alpha(A_G)}\alpha(W_G)\\
v_g\otimes_{\F}\phi_g&\mapsto&\alpha(\phi_g)
\end{array}$$
determines an isomorphism of $G$-graded algebras.
\end{corollary}
\begin{proof}
Let $\phi_g\in$End$^{\textmd{r}(g)}_{A_G}(W_G)$ and $\phi_h\in$End$^{\textmd{r}(h)}_{A_G}(W_G)$. Then
$$\alpha_W[(v_g\otimes_{\F}\phi_g)(v_h\otimes_{\F}\phi_h)]=\alpha_W(v_gv_h\otimes_{\F}\phi_g\circ\phi_h)=\alpha_W(\alpha(g,h)\cdot v_{gh}\otimes_{\F}\phi_g\circ\phi_h)=\cdots$$
By the definition of $\alpha_W$ and by Lemma \ref{twistendolemma}(2) we have
$$\cdots=\alpha(g,h)\cdot\alpha(\phi_g\circ\phi_h)=\alpha(\phi_g)\circ\alpha(\phi_h)=\alpha_W(v_g\otimes_{\F}\phi_g)\circ\alpha_W(v_h\otimes_{\F}\phi_h),$$
proving that $\alpha_W$ is multiplicative.
Clearly, $\alpha_W$ is a graded homomorphism. Using \eqref{therule} one can verify that $\alpha_W$ is invertible, completing the proof of the corollary.
\end{proof}
Here is a nice property of the graded product of graded modules $W_G$ and $W'_G$ over \textit{strongly graded algebras} $A_G$ and $A'_G$.
In this case we show that the $A_G\otimes_{\F}A'_G$-module $W_G\otimes_{\F}W'_G$ is induced (see \S\ref{sb}) from
the $A_G\otimes^G_{\F}A'_G$-module $W_G\otimes^G_{\F}W'_G$. We first need the following
\begin{lemma}\label{strgen}
Let \eqref{WW'} be $G$-graded left modules over $G$-graded $\F$-algebras \eqref{twoalgebragrading} respectively. Suppose that $A_G$ is strongly graded.
Then $$W_G\otimes_{\F}W'_G=(A_G\otimes_{\F}1)(W_G\otimes^G_{\F}W'_G)$$
(Similarly, $W_G\otimes_{\F}W'_G=(1\otimes_{\F}A'_G)(W_G\otimes^G_{\F}W'_G)$ if $A'_G$ is strongly graded). In particular, in both cases
\begin{equation}\label{gengrprod}
W_G\otimes_{\F}W'_G=(A_G\otimes_{\F}A'_G)(W_G\otimes^G_{\F}W'_G).
\end{equation}
\end{lemma}
\begin{proof}
The inclusion ``$\supseteq$" is obvious.
We prove the inclusion $``\subseteq"$ under the assumption that $A_G$ is strongly graded. The proof for $A'_G$ strongly graded is the same.
Clearly, it is enough to show this inclusion for any element of the form $w_g\otimes_{\F}w'_h\in W_g\otimes_{\F}W'_h$, where $g,h\in G$.
Indeed, given such an element, apply a homogeneous unit decomposition (see \eqref{unitdecomp})
\begin{equation}\label{unitgh-1}
1=\sum_j a_j^{(hg^{-1})}b_j^{(hg^{-1})},
\end{equation}
for appropriate homogeneous elements
$\{a_j^{(hg^{-1})}\}_j\subset A_{gh^{-1}},$ and $\{b_j^{(hg^{-1})}\}_j\subset A_{hg^{-1}},$ provided by the strongly graded condition on $A_G$.
We develop
\begin{eqnarray}\label{eq8.1}
\begin{array}{cl}
w_g\otimes_{\F}w'_h &=(\sum_j a_j^{(hg^{-1})}b_j^{(hg^{-1})}\otimes_{\F}1)(w_g\otimes_{\F}w'_h)\\
 &=\sum_j (a_j^{(hg^{-1})}\otimes_{\F}1)(b_j^{(hg^{-1})}\otimes_{\F}1)(w_g\otimes_{\F}w'_h)\\
 &\in(A_G\otimes_{\F}1)\sum_j (b_j^{(hg^{-1})}\otimes_{\F}1)(w_g\otimes_{\F}w'_h).
\end{array}\end{eqnarray}
Also, for every $j$ we have
\begin{equation}\label{eq8.2}
(b_j^{(hg^{-1})}\otimes_{\F}1)(w_g\otimes_{\F}w'_h)\in W_h\otimes^G_{\F}W'_h=(W_G\otimes^G_{\F}W'_G)_h.
\end{equation}
By equations \eqref{eq8.1} and \eqref{eq8.2} we obtain $w_g\otimes_{\F}w'_h\in(A_G\otimes_{\F}1)(W_G\otimes^G_{\F}W'_G).$
This completes the proof.
\end{proof}
\begin{lemma}\label{extendhom}
Let \eqref{WW'} be $G$-graded left modules over $G$-graded $\F$-algebras \eqref{twoalgebragrading} respectively,
and let $M$ is any left $A_G\otimes_{\F}A'_G$-module.
Suppose that either $A_G$ or $A'_G$ are strongly graded. Then the restriction map
\begin{equation}\label{pres}
\text{res}:\Hom^{}_{A_G\otimes_{\F}A'_G}(W_G\otimes_{\F}W'_G,M)\to\Hom^{}_{A_G\otimes_{\F}^GA'_G}(W_G\otimes^G_{\F}W'_G,M)
\end{equation}
is bijective.
\end{lemma}
\begin{proof}
Without loss of generality, let $A_G$ be strongly graded.
Let $$\varphi\in \Hom_{A_G\otimes_{\F}^GA'_G}(W_G\otimes^G_{\F}W'_G, M).$$
To establish the bijectivity property, we show that there exists a unique $\tilde{\varphi}\in \Hom_{A_G\otimes_{\F}A'_G}(W_G\otimes_{\F}W'_G, M)$ extending $\varphi$,
that is $\tilde{\varphi}|_{W_G\otimes^G_{\F}W'_G}=\varphi.$
If an extension $\tilde{\varphi}\in \Hom_{A_G\otimes_{\F}A'_G}(W_G\otimes_{\F}W'_G, M)$ does exist then by \eqref{gengrprod} it is determined by the rule
$\tilde{\varphi}(aw)=a\varphi(w)$ for any $a\in A_G\otimes_{\F}A'_G$ and $w\in W_G\otimes^G_{\F}W'_G$. This takes care of the uniqueness part.


For the existence, we need to show that the above rule is well-defined,
i.e., to verify that if $\sum_i c_iw_i=0$
for $c_i\in A_G\otimes_{\F}A'_G$ and $w_i\in W_G\otimes^G_{\F}W'_G$ then $\sum_i c_i\varphi(w_i)\stackrel{?}{=}0$.
Since $$W_G\otimes_{\F}W'_G=\bigoplus_{g,g'\in G}(W_g\otimes_{\F}W_{g'}'),$$
we may assume that all the summands $c_iw_i$ lie in $ W_g\otimes_{\F}W'_{g'}$ for fixed $g,g'\in G$.
Let us then write the vanishing sum assumption using homogeneous elements as follows.
$$\sum_i (x^i_{k_i}\otimes_{\F}x'^i_{h_i})(w^i_{g_i}\otimes_{\F}w'^i_{g_i})=0,$$
where for every $i$ in this summation
$$x^i_{k_i}\otimes_{\F}x'^i_{h_i}\in A_{k_i}\otimes_{\F}A'_{h_i}\text{ and }w^i_{g_i}\otimes_{\F}w'^i_{g_i}\in(W_G\otimes^G_{\F}W'_G)_{g_i}$$ such that
$k_ig_i=g$ and $h_ig_i=g'$.
Now, let $\sum_j a_j^{(g'g^{-1})}b_j^{(g'g^{-1})}=1$
be a unit decomposition corresponding to the strong grading condition $A_{gg'^{-1}}A_{g'g^{-1}}=A_e$ (see \eqref{unitgh-1}).
Next, use this decomposition to obtain
$$\begin{array}{cl}
\sum_i (x^i_{k_i}\otimes_{\F}x'^i_{h_i})\varphi(w^i_{g_i}\otimes_{\F}w'^i_{g_i})&
=\sum_j(a_j^{(g'g^{-1})}b_j^{(g'g^{-1})}\otimes_{\F}1)\sum_i (x^i_{k_i}\otimes_{\F}x'^i_{h_i})\varphi(w^i_{g_i}\otimes_{\F}w'^i_{g_i})\\
&=\sum_j(a_j^{(g'g^{-1})}\otimes_{\F}1)(b_j^{(g'g^{-1})}\otimes_{\F}1)\sum_i (x^i_{k_i}\otimes_{\F}x'^i_{h_i})\varphi(w^i_{g_i}\otimes_{\F}w'^i_{g_i})=\cdots
\end{array}$$
Note that for every $i$ and $j$ in the summation
$$(b_j^{(g'g^{-1})}\otimes_{\F}1) (x^i_{k_i}\otimes_{\F}x'^i_{h_i})=b_j^{(g'g^{-1})}x^i_{k_i}\otimes_{\F}x'^i_{h_i}\in A_{h_i}\otimes_{\F}A'_{h_i}=(A_G\otimes^G_{\F}A'_G)_{h_i}.$$
Together with the fact that $\varphi$ is an $A_G\otimes^G_{\F}A'_G$-homomorphism we get
$$\cdots=\sum_{j} (a_j^{(g'g^{-1})}\otimes_{\F}1)\varphi((b_j^{(g'g^{-1})}\otimes_{\F}1)\sum_i (x^i_{k_i}\otimes_{\F}x'^i_{h_i})(w^i_{g_i}\otimes_{\F}w'^i_{g_i}))=
\sum_{j} (a_j^{(g'g^{-1})}\otimes_{\F}1)\varphi(0)=0,$$
completing the proof of the lemma.
\end{proof}
Using the induction notation $$(W_G\otimes^G_{\F}W'_G)|^{A_G\otimes_{\F}A'_G}:=(A_G\otimes_{\F}A'_G)\otimes_{A_G\otimes^G_{\F}A'_G}(W_G\otimes^G_{\F}W'_G)$$
(see \S\ref{sb}) we have
\begin{corollary}\label{indstgrprod}
Let \eqref{WW'} be $G$-graded left modules over strongly $G$-graded $\F$-algebras \eqref{twoalgebragrading} respectively.
Then $$(W_G\otimes^G_{\F}W'_G)|^{A_G\otimes_{\F}A'_G}\cong W_G\otimes_{\F}W'_G.$$
\end{corollary}
\begin{proof}
To prove this consequence of Lemma \ref{extendhom}, and in view of the equality \eqref{gengrprod}, it suffices to introduce an inverse to the algebra homomorphism
\begin{eqnarray}\label{rhorho}\begin{array}{ccc}
(W_G\otimes^G_{\F}W'_G)|^{A_G\otimes_{\F}A'_G}&\to &(A_G\otimes_{\F}A'_G)(W_G\otimes^G_{\F}W'_G)\\
c\otimes_{A_G\otimes^G_{\F}A'_G}w&\mapsto& cw
\end{array},\end{eqnarray}
where $c\in A_G\otimes_{\F}A'_G$ and $w\in W_G\otimes^G_{\F}W'_G$. Indeed, let
$$\varphi\in \Hom_{A_G\otimes_{\F}^GA'_G}(W_G\otimes^G_{\F}W'_G, (W_G\otimes^G_{\F}W'_G)|^{A_G\otimes_{\F}A'_G})$$ be the natural left ${A_G\otimes_{\F}^GA'_G}$-module homomorphism
determined by
$$\varphi:\begin{array}{ccc}
W_G\otimes^G_{\F}W'_G&\to &(W_G\otimes^G_{\F}W'_G)|^{A_G\otimes_{\F}A'_G}\\
w&\mapsto& 1\otimes_{A_G\otimes^G_{\F}A'_G}w
\end{array},\ \  w\in W_G\otimes^G_{\F}W'_G.$$
Owing to Lemma \ref{extendhom}, there is a unique homomorphism
$$\tilde{\varphi}\in \Hom_{A_G\otimes_{\F}A'_G}(W_G\otimes_{\F}W'_G, (W_G\otimes^G_{\F}W'_G)|^{A_G\otimes_{\F}A'_G})$$ which extends $\varphi$ and,
as can easily be verified, is an inverse of \eqref{rhorho}.
\end{proof}
\subsection{Graded Hom-Tensor relations}\label{grHTrsec}
A graded version of Theorem \ref{ungrHTr} is given in this section.
Let \eqref{twoalgebragrading} be two $G$-graded $\F$-algebras, let $W_G$ and $\tilde{W}_G$ be graded left $A_G$-modules,
and let $W'_G$ and $\tilde{W}'_G$ be graded left $A'_G$-modules.
Then the (ungraded) Hom-Tensor mapping in this case
$$\Psi:\text{Hom}^{}_{A_G}(W_G,\tilde{W}_G)\otimes_{\F}\text{Hom}^{}_{A'_G}(W'_G,\tilde{W}_G')\xrightarrow{}
\Hom^{}_{A_G\otimes_{\F}A'_G}(W_G\otimes_{\F}W'_G,\tilde{W}_G\otimes_{\F}\tilde{W}'_G)$$
is determined by
(see \eqref{end-ten})
$$\Psi(f\otimes_{\F}f'):w\otimes_{\F}w'\mapsto f(w)\otimes_{\F}f'(w'),$$
where $f\in \text{Hom}^{}_{A_G}(W_G,\tilde{W}_G), w\in W_G, f'\in \text{Hom}^{}_{A'_G}(W'_G,\tilde{W}'_G)$ and $w'\in\tilde{W}'_G$.
Next, plug the $A_G\otimes_{\F}A'_G$-module $M:=\tilde{W}_G\otimes_{\F}\tilde{W}'_G$
inside the restriction map \eqref{pres}
$$\text{res}:\Hom^{}_{A_G\otimes_{\F}A'_G}(W_G\otimes_{\F}W'_G,\tilde{W}_G\otimes_{\F}\tilde{W}'_G)\to \Hom^{}_{A_G\otimes_{\F}^GA'_G}(W_G\otimes^G_{\F}W'_G,\tilde{W}_G\otimes_{\F}\tilde{W}'_G).$$
Let
$$\begin{array}{rl}B_G & :=\text{Hom}^{\text{r}(G)}_{A_G}(W_G,\tilde{W}_G)\otimes^G_{\F}\text{Hom}^{\text{r}(G)}_{A'_G}(W'_G,\tilde{W}_G')\\
 &\subseteq\text{Hom}^{}_{A_G}(W_G,\tilde{W}_G)\otimes_{\F}\text{Hom}^{}_{A'_G}(W'_G,\tilde{W}_G').\end{array}$$
Notice that for every homogeneous element $f_g\otimes_{\F}f'_g\in B_g,$
the map $$\res\circ\Psi(f_g\otimes_{\F}f'_g)\in \Hom^{}_{A_G\otimes_{\F}^GA'_G}(W_G\otimes^G_{\F}W'_G,\tilde{W}_G\otimes_{\F}\tilde{W}'_G)$$ takes any
\begin{equation}\label{eq8.3}
w_h\otimes_{\F}w'_h\in W_h\otimes_{\F}W'_h=(W_G\otimes^G_{\F}W'_G)_h
\end{equation}
to \begin{equation}\label{eq8.4}
f_g(w_h)\otimes_{\F}f'_g(w'_h)\in \tilde{W}_{hg}\otimes_{\F}\tilde{W}'_{hg} =(\tilde{W}_G\otimes^G_{\F}\tilde{W}'_G)_{hg} .
\end{equation}
Equations \eqref{eq8.3} and \eqref{eq8.4} say that for every $g\in G$
$$\res\circ\Psi:B_g\xrightarrow{}
\Hom^{\text{r}(g)}_{A_G\otimes^G_{\F}A'_G}(W_G\otimes^G_{\F}W'_G,\tilde{W}_G\otimes^G_{\F}\tilde{W}'_G).$$
We may thus regard
\begin{equation}\label{respsi}
\res\circ\Psi|_{B_G}:B_G\xrightarrow{G}
\Hom^{\text{r}(G)}_{A_G\otimes^G_{\F}A'_G}(W_G\otimes^G_{\F}W'_G,\tilde{W}_G\otimes^G_{\F}\tilde{W}'_G)
\end{equation}
as the {\it graded Hom-Tensor mapping}. As in the ungraded case, this linear map of $G$-graded $\F$-spaces preserves multiplication when $W_G=\tilde{W}_G$ and $W'_G=\tilde{W}'_G$.

\begin{theorem}\label{grHTr}
With the above notation, suppose that $A_G$ and $A'_G$ are strongly graded, and that the left modules $W_G$ and $W'_G$ are finitely-presented.
Then the graded Hom-Tensor mapping \eqref{respsi} is a $G$-graded linear isomorphism.
\end{theorem}
\begin{proof}
Firstly, recall that since $W_G$ and $W'_G$ are assumed to be finitely presented, then the ungraded Hom-Tensor $\F$-linear mapping $\Psi$ is invertible (Theorem \ref{ungrHTr}).
Next,
by Lemma \ref{extendhom}, and by the strongly grading assumption,
the restriction map \eqref{pres} is invertible. Tying up these two invertible maps we deduce that so is
\begin{equation}\label{recircpsi}
\res\circ\Psi:\text{Hom}^{}_{A_G}(W_G,\tilde{W}_G)\otimes_{\F}\text{Hom}^{}_{A'_G}(W'_G,\tilde{W}_G')\xrightarrow{\cong}
\Hom^{}_{A_G\otimes^G_{\F}A'_G}(W_G\otimes^G_{\F}W'_G,\tilde{W}_G\otimes_{\F}\tilde{W}'_G).\end{equation}
Since $B_G$ is a subspace of the domain of the invertible map \eqref{recircpsi},
injectivity of \eqref{respsi} is established (under the weaker assumption that at least one of $A_G, A'_G$ is strongly graded).

To show that \eqref{respsi} is onto, it is enough
to prove that the unique inverse image under \eqref{recircpsi} of any homogeneous element
$$\varphi_g\in\Hom^{\text{r}(g)}_{A_G\otimes^G_{\F}A'_G}(W_G\otimes^G_{\F}W'_G,\tilde{W}_G\otimes^G_{\F}\tilde{W}'_G)
(\subseteq \Hom^{}_{A_G\otimes^G_{\F}A'_G}(W_G\otimes^G_{\F}W'_G,\tilde{W}_G\otimes_{\F}\tilde{W}'_G))$$
of degree $g\in G$ actually lies inside $B_g$.
Let
$$\sum_{i=1}^mf_i\otimes_{\F}f'_i:=(\res\circ\Psi)^{-1}(\varphi_g)\in\text{Hom}^{}_{A_G}(W_G,\tilde{W}_G)\otimes_{\F}\text{Hom}^{}_{A'_G}(W'_G,\tilde{W}_G').$$
Now, for every $1\leq i\leq m,$ and $g\in G$ let
$${f}^g_i:\begin{array}{ccl}W_e&\to&\tilde{W}_g\\ w&\mapsto&(f_i(w))_g\end{array},\text{ and }{f'}^g_i:\begin{array}{ccl}W'_e&\to&\tilde{W}'_g\\ w'&\mapsto&(f'_i(w'))_g\end{array}.$$
By Lemma \ref{genst}, the strongly grading assumption yields that $W_G$ and $W'_G$ are generated by $W_e$ and $W'_e$ respectively, in other words $W_G=A_GW_e$ and $W'_G=A'_GW'_e$.
Thus, for every $1\leq i\leq m$, if ${f}^g_i$ and ${f'}^g_i$ are extendable to $W_G$ and $W'_G$ respectively, then the corresponding extensions should be
$$({f}_i)^{\text{r}(g)}:\sum_ja_jw_j\mapsto\sum_ja_j{f}^g_i(w_j),\ \ a_j\in A_G, w_j\in W_e,$$
and
$$({f}'_i)^{\text{r}(g)}:\sum_ja'_jw_j\mapsto\sum_ja'_j{f'}^g_i(w'_j),\ \ a'_j\in A'_G, w'_j\in W'_e.$$
We proceed as follows: (1) showing that $({f}_i)^{\text{r}(g)}$ is well-defined for every $1\leq i\leq m$, this is a
homogeneous homomorphism of degree $g$ in $\text{Hom}^{\text{r}(g)}_{A_G}(W_G,\tilde{W}_G)$.
The proof that $({f}'_i)^{\text{r}(g)}$is a well-defined homogeneous homomorphism in $\text{Hom}^{\text{r}(g)}_{A'_G}(W'_G,\tilde{W}'_G)$ is similar.
(2) After being convinced that $({f}_i)^{\text{r}(g)}$ and $({f}'_i)^{\text{r}(g)}$ are well-defined, we shall see that the homogeneous element
$$\sum_{i=1}^m({f}_i)^{\text{r}(g)}\otimes_{\F}({f}'_i)^{\text{r}(g)}\in B_g\subset \text{Hom}^{\text{r}(G)}_{A_G}(W_G,\tilde{W}_G)\otimes^G_{\F}\text{Hom}^{\text{r}(G)}_{A'_G}(W'_G,\tilde{W}_G')$$
is an inverse image under $\res\circ\Psi$ of $\varphi_g$ (and hence is the same as $\sum_{i=1}^mf_i\otimes_{\F}f'_i$).\\
(1) suppose that $\sum_ja_jw_j=0$ for some $a_j\in A_G,$ and $w_j\in W_e.$
We need to show that $\sum_ja_j{f}_i^{g}(w_j)\stackrel{?}{=}0$. For that one may assume that $a_j$ is homogeneous, say of degree $h$, for every $j$.
By the definition of ${f}^g_i$ and since $f_i\in\text{Hom}^{}_{A_G}(W_G,\tilde{W}_G)$ we have
$$\sum_ja_j{f}_i^{g}(w_j)=\sum_ja_j(f_i(w_j))_g=\sum_j(a_jf_i(w_j))_{hg}=(f_i(\sum_ja_jw_j))_{hg}=0,$$ and we are done with (1).

Before moving to (2), let us check more carefully how $({f}_i)^{\text{r}(g)}$ operates on homogeneous elements, say $w_h\in W_h$.
Write $w_h=\sum_ja_jw_j$, where $a_j\in A_h$ and $w_j\in W_e$ for every $j$. By the definitions we have
\begin{eqnarray}\label{hats}
\begin{array}{cl}
({f}_i)^{\text{r}(g)}(w_h)&=({f}_i)^{\text{r}(g)}(\sum_ja_jw_j)=\sum_ja_j{f}^g_i(w_j)=\sum_ja_j(f_i(w_j))_g\\
 &=(f_i(\sum_ja_jw_j))_{hg}=(f_i(w_h))_{hg},
\end{array}
\end{eqnarray}
and similarly, \begin{equation}\label{hats'}
({f}'_i)^{\text{r}(g)}(w'_h)=(f'_i(w'_h))_{hg}
\end{equation} for every $w_h\in W_h, w'_h\in W'_h$ and $1\leq i\leq m.$\\
(2) Evaluate $\res\circ\Psi$ on
$$\sum_{i=1}^m({f}_i)^{\text{r}(g)}\otimes_{\F}({f}'_i)^{\text{r}(g)}\in \text{Hom}^{\text{r}(g)}_{A_G}(W_G,\tilde{W}_G)\otimes^G_{\F}\text{Hom}^{\text{r}(g)}_{A'_G}(W'_G,\tilde{W}_G').$$
For every $w_h\otimes_{\F}w'_h\in W_h\otimes_{\F}W'_h=(W\otimes^G_{\F}W')_h$ we get
$$\res\circ\Psi(\sum_{i=1}^m({f}_i)^{\text{r}(g)}\otimes_{\F}({f}'_i)^{\text{r}(g)})(w_h\otimes_{\F}w'_h)
=\sum_{i=1}^m({f}_i)^{\text{r}(g)}(w_h)\otimes_{\F}({f}'_i)^{\text{r}(g)}(w'_h)=\cdots$$
Applying \eqref{hats} and \eqref{hats'} we obtain
$$\begin{array}{rl}
\cdots&=\sum_{i=1}^m({f}_i(w_h))_{hg}\otimes_{\F}({f}'_i(w'_h))_{hg}=(\res\circ\Psi(\sum_{i=1}^m({f}_i\otimes_{\F}{f}'_i))(w_h\otimes_{\F}w'_h))_{hg}\\
&=(\varphi_g(w_h\otimes_{\F}w'_h))_{hg}=\varphi_g(w_h\otimes_{\F}w'_h).
\end{array}$$
Consequently, $\res\circ\Psi(\sum_{i=1}^m({f}_i)^{\text{r}(g)}\otimes_{\F}({f}'_i)^{\text{r}(g)})$ agrees with $\varphi_g$
on $(W\otimes^G_{\F}W')_h$ for every $h\in G$, and hence
$\res\circ\Psi(\sum_{i=1}^m({f}_i)^{\text{r}(g)}\otimes_{\F}({f}'_i)^{\text{r}(g)})=\varphi_g$. This verifies (2) and completes the proof.
\end{proof}
We establish a graded version of the End-Tensor relations (see Theorem \ref{ungrHTrcor}).
\begin{corollary}\label{grHTrcor}
Let \eqref{WW'} be $G$-graded left modules over strongly $G$-graded $\F$-algebras \eqref{twoalgebragrading} respectively.
Suppose that $W_G$ and $W'_G$ are finitely-presented as ungraded modules over the algebras $A_G$ and $A'_G$ respectively. Then there is a
$G$-graded isomorphism of algebras
\begin{equation}\label{grendtens}
\End^{\text{r}(G)}_{A_G}(W_G)\otimes^G_{\F}\End^{\text{r}(G)}_{A'_G}(W'_G)\stackrel{G}{\cong}\End^{\text{r}(G)}_{A_G\otimes^G_{\F}A'_G}(W_G\otimes^G_{\F}W'_G).
\end{equation}
\end{corollary}
\subsection{The monoid $\mathcal{M}_{G,\F}$ and the endomorphism map}\label{MGFsec}
\begin{definition}\label{defeq}

Let $\psi_G:A_G\stackrel{G}{\to} A'_G$ be a graded isomorphism of $G$-graded $\F$-algebras.
Two $G$-graded left modules $W_G$ and $W'_G$ over $A_G$ and $A'_G$ respectively
are called $\psi_G$-{\it equivariant} if there exists
a $G$-graded $\F$-linear spaces isomorphism $\varphi_G:W_G\stackrel{G}{\to} W'_G$ such that
$$\varphi_G(x w)=\psi_G(x)\varphi_G(w), \ \ \forall x\in A_G,\ \  w\in W_G.$$
The modules $W_G,W'_G$ are said to be {\it $G$-equivariant} if they are $\psi_G$-{equivariant} for some graded isomorphism $\psi_G:A_G\stackrel{G}{\to} A'_G$ of $G$-graded $\F$-algebras.
\end{definition}
Evidently, graded equivariance is an equivalence relation
coarsening the graded isomorphism relation on $G$-graded modules. We
denote the graded equivariance class of $W_G$ by $[W_G]_{\e}$, and
let $\mathcal{M}_{G,\F}=\mathcal{M}_{\kappa,\kappa',G,\F}$ be the set of
graded-equivariance classes of left graded modules, whose cardinalities are bounded by some infinite cardinality $\kappa'$, over $G$-graded
$\F$-algebras $A_G$, where $[A_G]\in\text{Gr}_{\kappa}(G,\F)$. A zero object is also affiliated to $\mathcal{M}_{G,\F}$. This is the class of zero modules over all $G$-graded
$\F$-algebras.
The concept of graded equivariance gets along well with few of the
above mentioned concepts as recorded in the following claims.
\begin{lemma}\label{getsalong}
Let \eqref{WW'} be $G$-graded left modules over strongly $G$-graded $\F$-algebras \eqref{twoalgebragrading} respectively,
such that $[W_G]_{\e}=[W'_G]_{\e}\in\mathcal{M}_{G,\F}$. Then
\begin{enumerate}
\item $[\text{End}^{\text{r}(G)}_{A_G}(W_G)]=[\text{End}^{\text{r}(G)}_{A'_G}(W'_G)]\in\text{Gr}(G,\F)$.
  \item $W_G$ is graded simple iff so is $W'_G$.
  \item $W_G$ is finitely-generated (alternatively, finitely-presented) iff so is $W'_G$.
  \item For every $[W_G'']_{\e}\in\mathcal{M}_{G,\F}$
  $$[W_G\otimes^G_{\F}W_G'']_{\e}=[W_G''\otimes^G_{\F}W_G]_{\e}=[W_G''\otimes^G_{\F}W'_G]_{\e}\in\mathcal{M}_{G,\F}.$$
\end{enumerate}
In particular, by (1) and (2), $[W_G]\in\text{Irr}^{\text{gr}}_{\F}(A_G)$ iff $[W'_G]\in\text{Irr}^{\text{gr}}_{\F}(A'_G).$
\end{lemma}
\begin{proof}
  By direct manipulations.
\end{proof}
\begin{corollary}\label{corgrdprod}
The rule
$$\begin{array}{ccc}
\mathcal{M}_{G,\F}\times\mathcal{M}_{G,\F}&\to&\mathcal{M}_{G,\F}\\
([W_{G}]_{\e},[W_{G}']_{\e})&\mapsto &[W_{G}\otimes_{\F}^G W_{G}']_{\e}
\end{array}$$
determines a well-defined operation which
turns $\mathcal{M}_{G,\F}$ into an abelian monoid, whose identity is the graded equivariance class of the free module of rank one over the ordinary group algebra $\F G$.
\end{corollary}
\begin{proof}
  This is an immediate consequence of Lemma \ref{getsalong}(4).
\end{proof}
With the notation of Definition \ref{tgm} we have
\begin{lemma}\label{MHset}
The rule
\begin{eqnarray}\label{actHonM}\begin{array}{ccl}
H^2(G,\F^*)\times\mathcal{M}_{G,\F}&\to&\mathcal{M}_{G,\F}\\
([\alpha],[W_G]_{\e})&\mapsto & [\alpha]([W_G]_{\e}):=[\F^{\alpha} G\otimes_{\F}^GW_G]_{\e}
\end{array}\end{eqnarray}
determines a well-defined action of $H^2(G,\F^*)$ on $\mathcal{M}_{G,\F}$.
\end{lemma}
Similarly to \S\ref{submonsec}, we are interested in certain sub-monoids of $\mathcal{M}_{G,\F}$. Another important sub-monoid will be introduced in Corollary \ref{invsubmon}.
\begin{lemma}\label{submono}
The following are sub-monoids of $\mathcal{M}_{G,\F}$.
\begin{enumerate}
  \item The graded equivariance classes of graded left modules over $G$-graded $\F$-algebras whose graded classes belong to either one of the sub-monoids (see Notation \ref{notsbmon})
~$H^2(G,\F^*),\text{Sk}(G,\F),\text{CP}(G,\F),\text{St}(G,\F)\subset\text{Gr}(G,\F)$.
 \item The graded equivariance classes in $\mathcal{M}_{G,\F}$ of finitely-generated graded left modules, and of finitely-presented left graded modules,
both over strongly $G$-graded $\F$-algebras.
  \item The graded equivariance classes $\text{Irr}^{}_{G,\F}\subset\mathcal{M}_{G,\F}$ of absolutely $\F$-simple graded left modules, together with $0$.
\end{enumerate}
\end{lemma}
\begin{proof}\begin{enumerate}
\item The claim follows from closure with respect to graded products of the corresponding families of graded algebras in Lemma \ref{grprosgcp}.
\item By Lemma \ref{getsalong}(3) the finitely-generated and finitely-presented properties are preserved under equivariance.
By Lemma \ref{fgfp} these properties are closed under the graded product.
\item By Lemma \ref{getsalong}(1) and (2), absolutely graded-simplicity over $\F$ is an equivariance class property. Theorem \ref{irrsgp}
 says that this property is closed under the graded product.\end{enumerate}
\end{proof}

We are ready to introduce the endomorphism map. By Lemma \ref{getsalong}(1) there is a well-defined map
\begin{eqnarray}\label{endmap}
\End=\End_{\kappa,\kappa',\F,G}:\begin{array}{ccc}
\mathcal{M}_{G,\F}&\to&\text{Gr}_{}(G,\F)\\
~[W_{G}]_{\e}
 &\mapsto &[\End^{\text{r}(G)}_{A_G}(W_{G})]
\end{array},
\end{eqnarray}
where $W_G$ is a left graded module over a $G$-graded $\F$-algebra \eqref{eq:algebragrading}.
Taking $W_G:=(\F G)^x$ to be a one-generated free module over the group algebra $A_G:=\F G$, that is $$[(\F G)^x]_{\e}=1\in\mathcal{M}_{G,\F},$$ then by \eqref{endfree} we have
$$\End^{\text{r}(G)}_{\F G}((\F G)^x)\stackrel{G}\cong\F G.$$ Since $[\F G]=1\in\text{Gr}_{}(G,\F),$
it follows that \eqref{endmap} is a morphism of pointed sets. Moreover, recall that by \eqref{Z2actgr} and \eqref{actHonM} both $\mathcal{M}_{G,\F}$ and $\text{Gr}_{}(G,\F)$ are $H^2(G,\F^*)$-sets.
We have
\begin{theorem}\label{mapH2sets}
The endomorphism map is a morphism of pointed $H^2(G,\F^*)$-sets.
\end{theorem}
\begin{proof}
This is a consequence of Corollary \ref{endtwist}.
\end{proof}
The map \eqref{endmap} serves as an obstruction, see Theorem \ref{sgaobs}.
In several interesting cases, as presented herein, \eqref{endmap} is a homomorphism of monoids. A notable instance will be given later on in Theorem \ref{preobst}.
At the moment, let $\mathcal{SM}_{G,\F}$ be the set of equivariance classes in $\mathcal{M}_{G,\F}$ of finitely-presented left graded modules over strongly $G$-graded algebras.
Lemma \ref{submono}(2) says that $\mathcal{SM}_{G,\F}$
is a sub-monoid of $\mathcal{M}_{G,\F}$.
\begin{lemma}\label{endhom}
With the above notation, the endomorphism map restricted to $\mathcal{SM}_{G,\F}$
\begin{equation}\label{resendmap}
\End|_{\mathcal{SM}_{G,\F}}:
\mathcal{SM}_{G,\F}\to\text{Gr}_{}(G,\F)
\end{equation}
is a homomorphism of abelian monoids.
\end{lemma}
\begin{proof}
For any finitely-presented graded left modules \eqref{WW'} over strongly graded algebras \eqref{twoalgebragrading} we have
$$\End([W_{G}]_{\e}\cdot[W'_{G}]_{\e})=\End([W_{G}\otimes_{\F}^G W'_{G}]_{\e})=[\End^{\text{r}(G)}_{A_G\otimes_{\F}^GA'_G}(W_{G}\otimes_{\F}^G W'_{G})]=\cdots$$
Corollary \ref{grHTrcor} deals exactly with this setup.
$$\cdots=[\End^{\text{r}(G)}_{A_G}(W_G)\otimes^G_{\F}\End^{\text{r}(G)}_{A'_G}(W'_G)]=\End([W_{G}]_{\e})\cdot\End([W'_{G}]_{\e}),$$
proving that \eqref{resendmap} is multiplicative.
\end{proof}
The following map is of use in the next sections.
\begin{lemma}\label{preiotaG}
Let \eqref{eq:algebragrading} be a fixed $G$-graded $\F$-algebra. Then the rule
\begin{eqnarray}\label{iotaG}
\iota_{\mathcal{G}}:\begin{array}{ccc}
\text{Mod}(A_e)&\to&\mathcal{M}_{G,\F}\\
~[M]&\mapsto&[{A_G}\overline{\otimes}_{A_e}M]_{\e}
\end{array}\end{eqnarray}
determines a well-defined map, with (see the notation in Lemma \ref{submono}(3)) $$\text{Irr}_{\F}(A_e)\stackrel{\iota_{\mathcal{G}}}{\rightarrow}\text{Irr}_{G,\F}.$$
\end{lemma}
\begin{proof}
By Theorem \ref{indfunct}(3), the map \eqref{iotaG} is indeed well-defined.
From Lemma \ref{grirrtga} we deduce that a class in $\text{Irr}_{\F}(A_e)$ is mapped under $\iota_{\mathcal{G}}$ to an absolutely $\F$-simple graded equivariance class, that is to
$\text{Irr}_{G,\F}$.
\end{proof}
It is important to emphasize that the map
\eqref{iotaG} depends on the grading \eqref{eq:algebragrading}, and hence $\text{Irr}_{\F}(A_e)$ may be mapped to $\mathcal{M}_{G,\F}$ in many ways.
\section{$G$-invariant modules}\label{sb1}
\subsection{An action of $G$ on $G$-graded modules}\label{sb212}
Let \eqref{modec} be a $G$-graded $\F$-linear space, and let $h\in G$.
Then $W_G$ admits another $G$-grading, namely its $h$-\textit{suspension}
\cite[\S 2.2]{NVO04}\begin{equation}\label{suspension}
h(W_G):=\bigoplus_{g\in G} h(W_G)_g,
\end{equation}
which is obtained by relabelling the homogeneous components $h(W)_g:=W_{gh}$ (yes, this is the right orientation).
Note that for every $k,g,h\in G$ we have $g(h(W))_k=h(W)_{kg}=W_{kgh}$, and hence
\begin{equation}\label{actsusp}
g(h(W_G))=gh(W_G),\ \ \forall g,h\in G.
\end{equation}
We may regard \eqref{suspension} as a left action of $G$ on the $G$-graded $\F$-linear spaces, and in particular on the $G$-graded left $A_G$-modules.
The following claim is straightforward.
\begin{lemma}\label{gactMGF}
Let \eqref{WW'} be equivariant $G$-graded left modules over $G$-graded $\F$-algebras \eqref{twoalgebragrading}.
Then $[h(W_G)]_{\e}=[h(W'_G)]_{\e}$ for every $h\in G$. Consequently, the rule
$$(h,[W_G]_{\e})\mapsto [h(W_G)]_{\e}$$
determines a well-defined action of $G$ on $\mathcal{M}_{G,\F}$.
\end{lemma}
It is well-known that all minimal left ideals of a simple algebra are isomorphic as left modules. The following graded version of this result shows that the action \eqref{suspension} is
transitive on the (perhaps empty) set of minimal left graded ideals of a simply graded algebra.
\begin{lemma}\label{minimalgrideals}(see \cite[Lemma 2.7]{EK13})
Let $I_G,I'_G$ be minimal left graded ideals of a simply graded algebra \eqref{eq:algebragrading}. Then there exists $g\in G$ such that $I_G\stackrel{G}{\cong}g(I'_G)$ as graded $A_G$-modules.
In particular, if \eqref{eq:algebragrading} satisfies an additional descending chain condition on graded left ideals, then the above holds for any two minimal left graded modules of $A_G$.
\end{lemma}
\begin{proof}
By minimality of these ideals there exist non-zero homogeneous elements, say $x_h\in A_h$ and $y_{h'}\in A_{h'}$, such that $I_G=A_G\cdot x_h$ and $I'_G=A_G\cdot y_{h'}$.
Next, the non-zero element $x_h$ belongs to the two-sided graded ideal $I_G\cdot A_G=A_G\cdot x_h \cdot A_G$,
and hence by the graded-simplicity assumption, this two-sided ideal is the entire algebra $A_G$.
In particular, there is a decomposition of the unit as
$$1=\sum_ia_i\cdot x_h\cdot b_i,\ \  a_i,b_i\in A_G.$$
Multiplication by $y_{h'}$ yields
$$0\neq y_{h'}=\sum_ia_i\cdot x_h\cdot b_i\cdot y_{h'},$$
and so there exists at least one $i$ such that $x_h\cdot b_i\cdot y_{h'}$ is non-zero. In turn, for at least one of the homogeneous parts of $b_i$, say $b_{h''}\in A_{h''}$, there is a non-zero element
$$y_{k}:=x_h\cdot b_{h''}\cdot y_{h'}\in A_{k}\setminus \{ 0\},$$ which is homogeneous of degree $k:=h\cdot h''\cdot h'$, and of course belongs to $I'_G$.
Again by minimality of $I'_G$ it is generated by the non-zero element $y_k$. Finally, as can easily be verified by putting $g:=h''\cdot h'$, the map
$$\begin{array}{ccc}
I_G&\to& g(I'_G)\\
a\cdot x_h&\mapsto& a\cdot y_k
\end{array},\ \  a\in A_G$$
is a graded isomorphism of $A_G$-modules, proving the first part of the lemma.

If, additionally, $A_G$ is graded Artinian, then the second claim follows from the first one,
since in this case minimal left graded $A_G$-modules are isomorphic as graded $A_G$-modules to minimal left graded $A_G$-ideals (see \cite[page 47]{NVO82}).
\end{proof}
The next claim is also straightforward. It is needed in the sequel.
\begin{lemma}\label{quotsusp}
Let $W'_G\subseteq W_G$ be a graded submodule. Then $g(\bigslant{W_G}{W'_G})\stackrel{G}{\cong}\bigslant{g(W_G)}{g(W'_G)}$ for every $g\in G$.
\end{lemma}
\begin{lemma}\label{suspmorph}
Let \eqref{WW'} be $G$-graded left modules over $G$-graded $\F$-algebras \eqref{twoalgebragrading} and let $h\in G$. Then the following are equivalent.
\begin{enumerate}
\item $\varphi\in \text{Hom}_{A_G}^{\text{gr}}(W_G,W_G')$.
\item $\varphi\in \text{Hom}_{A_G}^{\text{gr}}(h(W_G),h(W_G'))$.
\item $\varphi\in \text{Hom}_{A_G}^{\text{r}(h)}(h(W_G),W_G')$.
\end{enumerate}
The same equivalence holds also for the corresponding entire $\F$-linear homomorphisms.
\end{lemma}
\begin{proof}
By direct computation.
\end{proof}
For every $h\in G$ let $\iota_h\in\Aut(G)$ be the inner automorphism determined by the $h$-conjugation.
Then with the notation in \S\ref{autsec} we have
\begin{lemma}\label{susp-sut}
Let \eqref{modec} be a $G$-graded left $A_G$-module. Then for every $h\in G$
$$\End^{\text{r}(G)}_{A_G}(h(W_{G}))\stackrel{G}{\cong}\iota_{h}(\End^{\text{r}(G)}_{A_G}(W_{G})).$$
\end{lemma}
\begin{proof}
By direct computation.
\end{proof}
Lemma \ref{susp-sut}, together with the definition of the $G$-action on $\mathcal{M}_{G,\F}$ in Lemma \ref{gactMGF} yields
\begin{corollary}\label{mapGsets}
The endomorphism map is a morphism of pointed $G$-sets.
\end{corollary}
\begin{definition}\label{inertdef}\cite[\S 2.2]{NVO04}
Let \eqref{modec} be a $G$-graded left $A_G$-module. The {\it inertia} (alternatively, the {\it stabilizer}) of $W_G$ consists of those elements $g\in G$ such that
$W_G$ is graded-isomorphic to the suspension $g(W_G)$ (see \eqref{suspension}) as graded $A_G$-modules.
Formally,
\begin{equation}\label{preinertia}
\mathcal{I}(W_G):=\{g\in G |\ \ W_G \stackrel{G}{\cong}g(W_G)\}.
\end{equation}
The $A_G$-module $W_G$ is called {\it $G$-invariant} if $\mathcal{I}(W_G)=G$.
\end{definition}
\begin{lemma}\label{sbgrpinert}
With the above notation,\begin{enumerate}
\item The inertia $\mathcal{I}(W_G)$ is a subgroup of the grading group $G$.
\item \cite[Proposition 2.2.3]{NVO04} For every $h\in G$, $\mathcal{I}(h(W_G))=h\mathcal{I}(W_G)h^{-1}.$
\item Any crossed product $A_e*G$ is $G$-invariant as a free graded left module over itself.
\item If $[W_G]_{\e}=[W'_G]_{\e}$ then $\mathcal{I}(W_G)=\mathcal{I}(W'_G)$. 
\end{enumerate}
\end{lemma}
\begin{proof}
By Lemma \ref{suspmorph} a graded-isomorphism $W_G \stackrel{G}{\cong}g(W_G)$ determines a graded-isomorphism
$g^{-1}(W_G)\stackrel{G}{\cong}g^{-1}(g(W_G))$ between their corresponding $g^{-1}$-suspensions. By \eqref{actsusp}, the right hand side is again $W_G$,
proving that the inertia is closed to inverses.
Next, suppose that $\varphi_1:g(W_G)\xrightarrow{G} W_G$ and $\varphi_2:h(W_G)\xrightarrow{G} W_G$ are graded-isomorphisms. Then by Lemma \ref{suspmorph}, $\varphi_1$ is a graded-isomorphism
between the corresponding $h$-suspensions $hg(W_G)$ and $h(W_G)$, and so $\varphi_2\circ\varphi_1$ is a graded-isomorphism between $hg(W_G)$ and $W_G$.
This proves that $\mathcal{I}(W_G)$ is closed also under multiplication, completing the proof of (1).
The second claim of the lemma follows from \eqref{actsusp}.
To prove (3), define for any $g\in G$
$$\varphi_g:\begin{array}{ccc}
A_e*G&\to&A_e*G\\
x&\mapsto &x\cdot u_{g}
\end{array}, $$
that is right multiplication by a fixed homogeneous unit of degree $g$.
Clearly, $\varphi_g:A_e*G\to g(A_e*G)$ is an invertible homomorphism of graded left modules and hence $g\in \mathcal{I}(A_e*G)$ for every $g\in G$.
The last claim of the Lemma is left to the reader.
\end{proof}
Let \eqref{WW'} be $G$-graded left modules over $G$-graded $\F$-algebras \eqref{twoalgebragrading} respectively. Recall that their graded product $W_G\otimes_{\F}^GW'_G$ is a $G$-graded
$A_G\otimes_{\F}^GA'_G$-module via the rule \eqref{modstr}.
The suspension \eqref{suspension} respects the graded product as the following straightforward lemma asserts.
\begin{lemma}\label{susresp}
With the above notation
$$g(W_G\otimes_{\F}^GW'_G)\stackrel{G}{\cong}g(W_G)\otimes_{\F}^Gg(W'_G), \ \forall g\in G.$$
In particular, if $g\in\mathcal{I}(W_G)\cap\mathcal{I}(W'_G)$ then $g\in\mathcal{I}(W_G\otimes_{\F}^GW'_G).$ Hence, if both $W_G$ and $W'_G$ are $G$-invariant then so is $W_G\otimes_{\F}^GW'_G$.
\end{lemma}
\begin{corollary}\label{galphcomm}
Let \eqref{modec} be a $G$-graded left module of a $G$-graded algebra \eqref{eq:algebragrading} and $\alpha\in Z^2(G,\F^*)$ any 2-cocycle. Then
$g(\alpha(W_G))\stackrel{G}{\cong}\alpha(g(W_G))$ for every $g\in G$.
In particular, $\mathcal{I}(W_G)=\mathcal{I}(\alpha(W_G))$.
\end{corollary}
\begin{proof}
By Lemmas \ref{sbgrpinert}(3) $\F^{\alpha}G$ is $G$-invariant as a free graded left module over itself, that is $\mathcal{I}(\F^{\alpha}G)=G$.
Together with Lemma \ref{wdmgp} and Lemma \ref{susresp}, we deduce that for every $g\in G$
$$g(\alpha(W_G))=g(\F^{\alpha}G\otimes_{\F}^GW_G)\stackrel{G}{\cong}g(\F^{\alpha}G)\otimes_{\F}^Gg(W_G)\stackrel{G}{\cong}\F^{\alpha}G\otimes_{\F}^Gg(W_G)=\alpha(g(W_G)).$$
The second part of the claim follows from the invertibility of $\alpha\in Z^2(G,\F^*)$.
\end{proof}
By Lemma \ref{sbgrpinert}(4), $G$-invariance is an equivariance class property, enabling us to denote the following.
\begin{notation}\label{MIrr}
Denote by $\mathcal{M}^{\text{inv}}_{G,\F}\subseteq\mathcal{M}_{G,\F}$ the subset of graded equivariance classes of
invariant left graded modules,
and by $\text{Irr}^{\text{inv}}_{G,\F}:=\mathcal{M}^{\text{inv}}_{G,\F}\bigcap\text{Irr}^{\text{}}_{G,\F}$
the subset of $\mathcal{M}_{G,\F}$ of graded equivariance classes of
invariant left graded modules which are absolutely graded-simple over $\F$, together with $0$.
\end{notation}
\begin{corollary}\label{invsubmon}
The sets $\mathcal{M}^{\text{inv}}_{G,\F}$ and $\text{Irr}^{\text{inv}}_{G,\F}$ are sub-monoids of $\mathcal{M}_{G,\F}$.
\end{corollary}
\begin{proof}
The first set is a sub-monoid by Lemma \ref{susresp}.
The latter is an intersection of two sub-monoids (see Lemma \ref{submono}(3)).
\end{proof}
The following lemma connects the inertia of a graded left ${A_G}$-module \eqref{modec}
with the support of the graded endomorphism algebra $\text{End}^{\text{r}(G)}_{A_G}(W_G)$ (see \eqref{endgr}).
\begin{lemma}(see \label{supendinert}\cite[Proposition 2.7.1(3)]{NVO04})
Let \eqref{modec} be a $G$-graded left module of a graded algebra \eqref{eq:algebragrading}. Then
\begin{enumerate}
\item $\mathcal{I}_{}(W_G)\subseteq\text{Supp}_G(\text{End}^{\text{r}(G)}_{A_G}(W_G)).$
\item If, additionally, $W_G$ is graded simple then
$\mathcal{I}(W_G)=\text{Supp}_G(\text{End}^{\text{r}(G)}_{A_G}(W_G)).$
\end{enumerate}
\end{lemma}
\begin{proof}
Let $h\in\mathcal{I}_{}(W_G)$. Then there exists a graded isomorphism $\varphi:W_G\stackrel{G}{\to} h(W_G)$ of graded left ${A_G}$-modules.
By Lemma \ref{suspmorph}, the non-zero homomorphism $\varphi\in\text{End}^{\text{r}(h)}_{A_G}(W_G)$, proving (1).
Next, if $W_G$ is graded simple then by the graded version of Schur's Lemma (Lemma \ref{schurgr}) a non-zero homogeneous $\varphi\in \text{End}^{\text{r}(h)}_{A_G}(W_G)$
is a homogeneous isomorphism of degree $h$.
Again by Lemma \ref{suspmorph}, it determines a graded isomorphism $W_G \stackrel{G}{\cong}h(W_G)$ proving $h\in\mathcal{I}(W_G)$.
\end{proof}

\begin{lemma}\label{endgriscp}(see \cite[Theorem 2.10.2]{NVO04})
Let \eqref{modec} be a left graded module of a $G$-graded algebra \eqref{eq:algebragrading}.
Then $W_G$ is $G$-invariant if and only if the $G$-graded algebra $\text{End}^{\text{r}(G)}_{A_G}(W_G)$
is graded isomorphic to a crossed product $\text{End}^e_{A_G}(W_G)*G.$
\end{lemma}
\begin{proof}
By Definition \ref{inertdef}, the invariance condition says that there exists a graded isomorphism of $A_G$-modules $\phi_g:W_G\xrightarrow{\sim} g(W_G)$ for every $g\in G$.
By Lemma \ref{suspmorph}, this is the same as to say that there exists an invertible homogeneous element
$\phi_g\in\text{Aut}^{\text{r}(g)}_{A_G}(W_G)$ for every $g\in G$.
Equivalently, $\text{End}^{\text{r}(G)}_{A_G}(W_G)$ is a crossed product (see \S\ref{sgcpsect}) of $G$ over the base algebra $\text{End}^e_{A_G}(W_G)$.
\end{proof}
With Notations \ref{notsbmon} and \ref{MIrr} we have
\begin{lemma}\label{end-1}\begin{enumerate}
Under the endomorphism map \eqref{endmap}
\item The inverse image of the sub-monoid $\text{CP}(G,\F)<\text{Gr}_{}(G,\F)$ is $\mathcal{M}^{\text{inv}}_{G,\F}.$
\item The inverse image of the sub-monoid  $H^2(G,\F^*)<\text{Gr}_{}(G,\F)$
is $\text{Irr}^{\text{inv}}_{G,\F}$. \end{enumerate}
That is
$\mathcal{M}^{\text{inv}}_{G,\F}=\End^{-1}(\text{CP}(G,\F)),\text{  and }\text{Irr}^{\text{inv}}_{G,\F}=\End^{-1}(H^2(G,\F^*))  .$
\end{lemma}
\begin{proof}
Lemma \ref{endgriscp} is responsible for the first part. Next, by Lemma \ref{preobs} graded equivariance classes of $\F$-absolutely graded-simple modules are exactly those classes which
are sent under the
endomorphism map to graded classes of twisted group algebras (supported by the various subgroups of $G$). By the first part of the lemma,
these $\F$-absolutely graded-simple classes are $G$-invariant iff the corresponding twisted group algebras are
supported by $G$ itself (and so their classes belong to $H^2(G,\F^*)$).
\end{proof}
By Lemma \ref{end-1} we obtain that the image of the restriction of the endomorphism map to $\text{Irr}^{\text{inv}}_{G,\F}$ lies in $H^2(G,\F^*)$. Denote this restriction by
\begin{equation}\label{omega}
\omega_{G,\F}:=\End|_{\text{Irr}^{\text{inv}}_{G,\F}}:\text{Irr}^{\text{inv}}_{G,\F}\to H^2(G,\F^*).\end{equation}
We should clarify this notation.
Let $W_G$ be a left $G$-graded module over a $G$-graded $\F$-algebra $A_G$.
Suppose that this module is $\F$-absolutely graded-simple and that its isomorphism class (and so its equivariance class) is $G$-invariant.
Then the endomorphism map takes $[W_G]_{\e}\in\text{Irr}^{\text{inv}}_{G,\F}$ to the unique cohomology class $[\alpha]\in H^2(G,\F^*)$
such that $\End_{A_G}^{\text{r}(G)}(W_G)$ is graded isomorphic to the twisted group algebra $\F^{\alpha}G$ for any $\alpha\in [\alpha]$.
Formally, for every  $[W_G]_{\e}\in\text{Irr}^{\text{inv}}_{G,\F}$
\begin{equation}\label{omend}
\End_{A_G}^{\text{r}(G)}(W_G)\stackrel{G}{\cong}\F^{\alpha}G,\ \ \ [\alpha]=\End([W_G]_{\e})=\omega_{G,\F}([W_G]_{\e}).
\end{equation}
\begin{theorem}\label{preobst}
The map \eqref{omega} is a homomorphism of abelian monoids.
\end{theorem}
\begin{proof}
Let $[W_G]_{\e},[W'_G]_{\e}\in\text{Irr}^{\text{inv}}_{G,\F}$, where $W_G$ and $W'_G$ are left graded modules over $A_G$ and $A'_G$ respectively.
Then with the notation \eqref{omend}, there exist cohomology classes $[\alpha],[\alpha'] \in H^2(G,\F^*)$
such that $\End_{A_G}^{\text{r}(G)}(W_G)$ and $\End_{A'_G}^{\text{r}(G)}(W'_G)$ are graded isomorphic to $\F^{\alpha}G$ and $\F^{\alpha'}G$ respectively,
and so by Lemma \ref{invGr}(1)
\begin{equation}\label{eq1}
\End^{\text{r}(G)}_{A_G}(W_G)\otimes^G_{\F}\End^{\text{r}(G)}_{A'_G}(W'_G)\stackrel{G}{\cong}\F^{\alpha\cdot\alpha'}G.
\end{equation}
On the other hand, Corollary \ref{invsubmon} tells us that the graded product of classes in $\text{Irr}^{\text{inv}}_{G,\F}$ is again in $\text{Irr}^{\text{inv}}_{G,\F}$.
Thus, $[W_G\otimes_{\F}^GW'_G]_{\e}\in\text{Irr}^{\text{inv}}_{G,\F}$.
Therefore, there exists $[\alpha''] \in H^2(G,\F^*)$ such that
\begin{equation}\label{eq2}
\End^{\text{r}(G)}_{A_G\otimes_{\F}^GA'_G}(W_G\otimes_{\F}^GW'_G)\stackrel{G}{\cong}\F^{\alpha''}G.
\end{equation}
Now, consider the graded Hom-Tensor homomorphism \eqref{respsi} of graded algebras
\begin{equation}\label{eq3}
\End^{\text{r}(G)}_{A_G}(W_G)\otimes^G_{\F}\End^{\text{r}(G)}_{A'_G}(W'_G)\stackrel{G}{\to}\End^{\text{r}(G)}_{A_G\otimes^G_{\F}A'_G}(W_G\otimes^G_{\F}W'_G).
\end{equation}
Plugging equations \eqref{eq1} and \eqref{eq2} into \eqref{eq3}, we obtain a graded homomorphism of graded algebras
$$\F^{\alpha\cdot\alpha'}G\stackrel{G}{\to}\F^{\alpha''}G,$$
which must be a graded isomorphism by Lemma \ref{cphomim}. We get
$$\begin{array}{cl}
\End([W_G]_{\e})\cdot\End([W'_G]_{\e})&=
[\End^{\text{r}(G)}_{A_G}(W_G)\otimes^G_{\F}\End^{\text{r}(G)}_{A'_G}(W'_G)]\\
&=[\End^{\text{r}(G)}_{A_G\otimes_{\F}^GA'_G}(W_G\otimes_{\F}^GW'_G)]\\
&=\End([W_G]_{\e}\cdot[W'_G]_{\e}),
\end{array}$$
proving that $\End$ respects products when restricted to $\text{Irr}^{\text{inv}}_{G,\F}$.
\end{proof}
\begin{theorem}\label{surjomega}(compare with \cite[XII.4.30]{FellDoran2})
Let $[W_G]_{\e}\in\text{Irr}^{\text{inv}}_{G,\F}$, and let $[\omega]\in H^2(G,\F^*)$. Then there exists $[\alpha]\in H^2(G,\F^*)$ such that $\omega_{G,\F}([\alpha(W_G)]_{\e})=[\omega]$.
In particular, the homomorphism \eqref{omega} is surjective.
\end{theorem}
\begin{proof}
This is a consequence of the facts that both homomorphisms are maps of $H^2(G,\F^*)$-sets (Theorem \ref{mapH2sets}) and that $H^2(G,\F^*)$ acts transitively on itself by multiplication.
That is, $[\alpha]:=\omega_{G,\F}([W_G]_{\e})^{-1}\cdot[\omega]$ does the job.
\end{proof}
\subsection{The inertia of modules over base algebras, proof of Theorem B}\label{sb11}
Let $M$ be a left $A_e$-module. While it is fairly natural to define its inertia (and, particularly, to determine if it is invariant) with respect to a strong grading,
it is not clear how to generalize this notion for any $G$-grading \eqref{eq:algebragrading}.
The following definition, which involves the notions of associated graded modules \eqref{assogrd}
and of inertia groups of graded modules \eqref{preinertia}, suits our requirements very well.
\begin{definition}\label{basealginertdef}(compare with \cite[\S 3.1]{NVO04})
Let \eqref{eq:algebragrading} be any $G$-graded algebra and let $M$ be a left $A_e$-module.
The {\it inertia group} of $M$ with respect to the grading $\mathcal{G}$
is the inertia group of its associated graded left $A_G$-module $A_G\overline{\otimes}_{A_e}M$, more precisely
\begin{equation}\label{inertiasuspension}
\mathcal{I}_{\mathcal{G}}(M):=\{g\in G |\  {A_G}\overline{\otimes}_{A_e}M \stackrel{G}{\cong}g({A_G}\overline{\otimes}_{A_e}M)\}.
\end{equation}
In particular, $M$ is {\it invariant} with respect to \eqref{eq:algebragrading} if $\mathcal{I}_{\mathcal{G}}(M)=G$.
\end{definition}
Denote the elements of Mod$(A_e)$ which are invariant with respect to \eqref{eq:algebragrading} by Mod$(A_e)^{\mathcal{G}}$.
Similarly, denote the elements of Irr$_{\F}(A_e)$ which are invariant with respect to \eqref{eq:algebragrading} by
\begin{equation}\label{inirr}
(\text{Irr}^{}_{\F}(A_e))^{\mathcal{G}}:=\left\{[M]\in\text{Irr}^{}_{\F}(A_e)|\ \mathcal{I}_{\mathcal{G}}(M)=G\right\}.
\end{equation}
Theorem \ref{indfunct}(1) yields
\begin{lemma}\label{indsuffice}
Let \eqref{eq:algebragrading} be strongly graded. Then $$\mathcal{I}_{\mathcal{G}}(M)=\{g\in G|{A_g}{\otimes}_{A_e}M\cong M\}.$$
\end{lemma}

Recall that a $G$-graded algebra $\F$-algebra \eqref{eq:algebragrading} and its twisting $\alpha(A_G)$ admit isomorphic base algebras
(see Lemma \ref{basetwist}) for any $\alpha\in Z^2(G,\F^*)$. Thinking of modules over $\alpha(A_G)_e$ as $A_e$-modules we have
\begin{proposition}\label{invwrtgract}
Let \eqref{eq:algebragrading} be a $G$-graded algebra $\F$-algebra and let $M$ be a left $A_e$-module. Then
$\mathcal{I}_{\mathcal{G}}(M)=\mathcal{I}_{\alpha(\mathcal{G})}(M)$ for every $\alpha\in Z^2(G,\F^*)$.
In particular, $M$ is $G$-invariant with respect to $\mathcal{G}$ if and only if it is
 $G$-invariant with respect to $\alpha(\mathcal{G})$. Consequently, $\text{Irr}_{\F}(A_e)^{\mathcal{G}}=\text{Irr}_{\F}(A_e)^{\alpha(\mathcal{G})}$
for every $\alpha\in Z^2(G,\F^*)$.
\end{proposition}
\begin{proof}
By Corollary \ref{assoctwwist}
\begin{equation}\label{eq11}\alpha(A_G)\overline{\otimes}_{A_e}M^{}\stackrel{G}{\cong}\alpha({A_G}\overline{\otimes}_{A_e}M).\end{equation}
Suspend \eqref{eq11} by $g\in G$ and apply Corollary \ref{galphcomm} to obtain
\begin{equation}\label{eq11.1}
g(\alpha(A_G)\overline{\otimes}_{A_e}M)\stackrel{G}{\cong}g(\alpha({A_G}\overline{\otimes}_{A_e}M))\stackrel{G}{\cong}\alpha(g({A_G}\overline{\otimes}_{A_e}M)).\end{equation}
Hence, by equations \eqref{eq11} and \eqref{eq11.1}, Definition \ref{basealginertdef} says that $g\in\mathcal{I}_{\alpha(\mathcal{G})}(M)$ if and only if
$\alpha({A_G}\overline{\otimes}_{A_e}M)\stackrel{G}{\cong}\alpha(g({A_G}\overline{\otimes}_{A_e}M)).$
Bearing the invertibility of $\alpha\in Z^2(G,\F^*)$ in mind, this condition is the same as $g\in\mathcal{I}_{\mathcal{G}}(M)$.
\end{proof}
\begin{lemma}\label{iotaGlemma}
Let \eqref{eq:algebragrading} be a fixed $G$-graded $\F$-algebra. Then the map \eqref{iotaG}
sends Mod$(A_e)^{\mathcal{G}}$ to $\mathcal{M}^{\text{inv}}_{G,\F}$. In particular, $\text{Irr}_{\F}(A_e)^{\mathcal{G}}$ is mapped into $\text{Irr}^{\text{inv}}_{G,\F}$ under this map.
\end{lemma}
\begin{proof}
By Definition \ref{basealginertdef} the map $\iota_{\mathcal{G}}$ takes invariant classes to invariant graded equivariance classes.
In particular, since the map \eqref{iotaG} sends $\text{Irr}_{\F}(A_e)^{}$ into $\text{Irr}^{}_{G,\F}$ (see Lemma \ref{preiotaG}), then this map
sends $\text{Irr}_{\F}(A_e)^{\mathcal{G}}$ into $\text{Irr}^{\text{inv}}_{G,\F}$.
\end{proof}
Next, composition of the maps \eqref{iotaG} and \eqref{omega} yields a map
\begin{equation}\label{omegaiotaG}
\omega_{\mathcal{G}}:=\omega_{G,\F}\circ\iota_{\mathcal{G}}:\text{Irr}_{\F}(A_e)^{\mathcal{G}}\to H^2(G,\F^*).
\end{equation}
Theorem B is a consequence of Theorem \ref{preobst} and Lemma \ref{iotaGlemma}.\ \ \ \ \ \ \ \ \ \ \ \ \ \ \ \ \ \ \  $\qed$

Any $\F$-absolutely simple left $A_e$-module $M$ (not necessarily $\mathcal{G}$-invariant) determines a class
$\omega_{\mathcal{I}}([M])\in H^2(\mathcal{I}_{\mathcal{G}}(M),\F^*)$ in the second cohomology of its inertia $\mathcal{I}_{\mathcal{G}}(M)$.
With the notations \eqref{omend} and \eqref{omegaiotaG} we have
\begin{theorem}\label{let}
Let \eqref{eq:algebragrading} be a $G$-graded algebra $\F$-algebra, and let $[M]\in\text{Irr}_{\F}(A_e)$ of inertia $\mathcal{I}_{\mathcal{G}}(M)<G$. Then
$$\text{End}^{\text{r}(G)}_{A_G}({A_G}\overline{\otimes}_{A_e}M)\stackrel{G}{\cong} \F^{\omega}\mathcal{I}_{\mathcal{G}}(M)$$ for any 2-cocycle
$\omega$ in the cohomology class $\omega_{\mathcal{I}}([M])\in H^2(\mathcal{I}_{\mathcal{G}}(M),\F^*)$.
\end{theorem}

\subsection{Proof of Theorem C}\label{pfB}
Let $M$ be a simple left $A_e$-module. Then
$$\dim_{\F}(M)\leq\dim_{\F}(A_e)<\infty,$$
and so by Schur's Lemma $[M]\in\text{Irr}_{\F}(A_e)$ (see Theorem \ref{dixmier}).
Next, 
by Theorem \ref{let} the graded endomorphism algebra $D_G:=\text{End}^{\text{r}(G)}_{A_G}({A_G}\overline{\otimes}_{A_e}M)$ is
a graded division algebra which is graded isomorphic to the twisted group algebra $\F^{\omega}\mathcal{I}_{\mathcal{G}}(M)$.
The rest of the proof goes along the lines of \cite[Corollary 4.6.7]{NVO04}, and \cite[Theorem 2.6]{EK13}.
Applying Lemma \ref{varphixglemma} for the graded left $A_G$-module $W_G:={A_G}\overline{\otimes}_{A_e}M$ there is a graded homomorphism
\begin{equation}\label{phivg}
A_G\stackrel{G}\to\text{End}^{\text{l}(G)}_{D_G}({A_G}\overline{\otimes}_{A_e}M).\end{equation}
The fact that $A_G$ is graded Artinian, together with the graded Graded Density Theorem \cite[Theorem 2.5]{EK13}, is responsible for ${A_G}\overline{\otimes}_{A_e}M$ being finite-dimensional over $D_G$.
In turn, the surjectivity of \eqref{phivg} follows again from the graded Graded Density Theorem and from the finite $D_G$-dimension of ${A_G}\overline{\otimes}_{A_e}M$.
\ \ \ \ \ \ \ \ \ \ \ $\qed$

Under the setup of Corollary D, the graded Artinian demand on the graded simple algebra \eqref{eq:algebragrading} enables us to take $M$ as a minimal left $A_e$-ideal.
What would happen if we chose another minimal left $A_e$-ideal, say $M'$? Although $M$ and $M'$ might behave totally different, their associated graded modules
${A_G}\overline{\otimes}_{A_e}M$ and ${A_G}\overline{\otimes}_{A_e}M'$ are always isomorphic. All the more so, by the second part of Lemma \ref{minimalgrideals} there exists $g\in G$ such that
$${A_G}\overline{\otimes}_{A_e}M'\stackrel{G}{\cong}g({A_G}\overline{\otimes}_{A_e}M).$$ In turn, by Lemma \ref{susp-sut} the corresponding endomorphism algebras of these two associated graded
modules are graded isomorphic up to relabelling their components with respect to conjugation by this element $g$.
This fact is more comprehensible when noting that by Lemma \ref{sbgrpinert}(2) the above two associated graded modules admit conjugate inertia subgroups, that is
$$\mathcal{I}_{\mathcal{G}}(M')=\mathcal{I}({A_G}\overline{\otimes}_{A_e}M')=\mathcal{I}(g({A_G}\overline{\otimes}_{A_e}M))=
g\mathcal{I}({A_G}\overline{\otimes}_{A_e}M)g^{-1}=\iota_g(\mathcal{I}_{\mathcal{G}}(M)).$$
The corresponding twisting in $Z^2(\mathcal{I}_{\mathcal{G}}(M'),\F^*)=Z^2(\iota_g(\mathcal{I}_{\mathcal{G}}(M)),\F^*)$ that determines the fine part of \eqref{eq:algebragrading}
is described using the notation \eqref{H<G} putting
$\theta:=\iota_g\in\Aut(G)$.
We summarize the discussion in
\begin{theorem}\label{uptoconj}
Let \eqref{eq:algebragrading} be a $G$-graded algebra $\F$-algebra, and let $[M]\in\text{Irr}_{\F}(A_e)$ of inertia $\mathcal{I}_{\mathcal{G}}(M)<G$. Then
for any $[M']\in\text{Irr}_{\F}(A_e)$ there exists $g=g([M'])\in G$ such that
$$\text{End}^{\text{r}(G)}_{A_G}({A_G}\overline{\otimes}_{A_e}M')\stackrel{G}{\cong} \F^{\omega'}\mathcal{I}_{\mathcal{G}}(M'),$$
where $\mathcal{I}_{\mathcal{G}}(M')=\iota_g(\mathcal{I}_{\mathcal{G}}(M))$, and $\omega'=\iota_g(\omega)\in Z^2(\mathcal{I}_{\mathcal{G}}(M'),\F^*)$ for any 2-cocycle
$\omega$ in the cohomology class $\omega_{\mathcal{I}}([M])\in H^2(\mathcal{I}_{\mathcal{G}}(M),\F^*)$.
\end{theorem}

\section{Extension criteria}
We begin our discussion about modules extensions from a base algebra $A_e$ to the entire $G$-graded algebra $\F$-algebra \eqref{eq:algebragrading} by recalling once again (see Lemma \ref{basetwist})
that this algebra and its twistings $\alpha(A_G)$ admit isomorphic base algebras
for all $\alpha\in Z^2(G,\F^*)$. We think of modules over $\alpha(A_G)_e$ as $A_e$-modules.
The following easy result says that the extension property is not affected by coboundary twistings.
\begin{proposition}\label{resalpha}
Let \eqref{eq:algebragrading} be a $G$-graded algebra $\F$-algebra, let $M$ be a left $A_e$-module and let $\alpha\in B^2(G,\F^*)$ be a 2-coboundary (see \eqref{coboundary}).
Then $M$ is restricted from $\alpha(A_G)$ if and only if it is restricted from $A_G$.
\end{proposition}
\begin{proof}
Suppose that an $A_G$-module structure $\hat{\eta}:A_G\to$End$_{\F}(M)$ restricts to the $A_e$-module structure
\begin{equation}\label{eta}
\eta:A_e\to\text{End}_{\F}(M).
\end{equation}
Apply the coboundary condition \eqref{coboundary} of $\alpha$ with respect to a graded basis $\{v_g\}_{g\in G}$ to define $\alpha(\hat{\eta}):\alpha(A_G)\to\text{End}_{\F}(M),$ where
$$
\alpha(\hat{\eta})(v_g\otimes_{A_e}x_g)(m):=\lambda_e^{-1}\lambda_g\hat{\eta}(x_g)(m),\ \
x_g\in A_g, g\in G, m\in M.$$
As can easily be checked, $\alpha(\hat{\eta})$ is an $\alpha(A_G)$-action on $M$ which restricts to $\eta$. 
\end{proof}
The concept of skew group algebras (see Definition \ref{skewdef}) 
yields a necessary and sufficient condition for an $A_e$-module $M$ to be extendable to $A_G$. For strongly graded algebras, Dade established the following theorem,
formulated here in our terminology.
\begin{theorem}\label{dadethm}\cite[Theorem 2.8]{Dade81}
Let \eqref{eq:algebragrading} be a strongly $G$-graded algebra and let $M$ be any left $A_e$-module.
Then $M$ is restricted from $A_G$ if and only if $\text{End}^{\text{r}(G)}_{A_G}({A_G}\overline{\otimes}_{A_e}M)$ is graded isomorphic to a skew group algebra supported by $G$.
\end{theorem}
By Theorem \ref{dadethm} we deduce that a necessary condition for an $A_e$-module $M$ to be restricted from a strongly graded algebra \eqref{eq:algebragrading} is that
$\text{End}^{\text{r}(G)}_{A_G}({A_G}\overline{\otimes}_{A_e}M)$
(which is the same as $\text{End}^{\text{r}(G)}_{A_G}({A_G}{\otimes}_{A_e}M)$ in this case, see Theorem \ref{indfunct}(1)) is in particular graded isomorphic to a $G$-supported crossed product
(over the base algebra $\text{End}^{e}_{A_G}({A_G}\overline{\otimes}_{A_e}M)$).
Lemma \ref{endgriscp} says then that in this case ${A_G}\overline{\otimes}_{A_e}M$ is $G$-invariant, in other words,
$M$ is invariant with respect to $A_G$ (see Definition \ref{basealginertdef}). We record this discussion.
\begin{corollary}\label{sgextinv}(Dade)
Let \eqref{eq:algebragrading} be strongly graded, then a necessary condition for $M$ to be restricted from $A_G$ is being $G$-invariant with respect to $\mathcal{G}$.
\end{corollary}

Theorem \ref{dadethm} admits the following version in case where the grading is not necessarily strong.
\begin{theorem}\label{gendade}
Let \eqref{eq:algebragrading} be a $G$-graded algebra and let $M$ be any left $A_e$-module. Then $M$ is restricted from $A_G$ if and only if there exists a $G$-graded left $A_G$-module \eqref{modec} with
$W_e\cong M$ such that $\text{End}^{\text{r}(G)}_{A_G}(W_G)$ is graded isomorphic to a skew group algebra supported by $G$. 
\end{theorem}
\begin{proof}
Suppose that a graded module \eqref{modec} satisfying the hypothesis does exist. Let
\begin{equation}\label{homogenfamily}
\{\phi_g\}_{g\in G}\subset\text{End}^{\text{r}(G)}_{A_G}(W_G)
\end{equation}
be a set of homogeneous graded automorphisms satisfying the skew group algebra condition, that is
\begin{equation}\label{skewphi}\phi_g\circ\phi_h=\phi_{gh},\ \ \forall g,h\in G.
\end{equation} Then $W=W_G$ can be endowed with a left ungraded $A_G$-module structure $^{\star}W$ defined by homogeneous
elements in $A_G$ as follows
\begin{equation}\label{staract}
a_g\star w:=(a_gw)\phi_g^{-1}=a_g[(w)\phi_g^{-1}], \ \forall g\in G, a_g\in A_g, w\in W.
\end{equation}
The right equality in \eqref{staract} follows from the intertwining property of the automorphisms in \eqref{homogenfamily}. 
Then the skew group algebra condition \eqref{skewphi} ascertains that \eqref{staract} indeed defines an $A_G$-action.
Note that $a\star w_e\in W_e$ for every $a\in A_G$ and $w_e\in W_e$, hence $W_e$ is an $A_G$-submodule of $^{\star}W$. Moreover, by \eqref{skewphi} the automorphism $\phi_e$ is the identity, and so
$$a_ew_e=a_e\star w_e, \ \forall  a_e\in A_e, w_e\in W_e,$$ that is the two actions of $A_e$ on $W_e$ agree.
In other words, the $A_G$-module structure of $W_e$ restricts to $W_e\cong M$ as an $A_e$-module.

Conversely, suppose that $M$ is restricted from $A_G$. Endow the $G$-graded space
$$W_G:=\F G\otimes_{\F}M=\bigoplus_{g\in G}g\otimes_{\F}M$$
with the left $A_G$-action $*$ defined by
$$a_h \ast (g\otimes_{\F}m):={hg}\otimes_{\F}a_hm,\ \ a_h\in A_h,g\in G,m\in M.$$
Then it is easily seen that $W_G$ is a $G$-graded $A_G$-module with $W_e\cong M$ as $A_e$-modules.
Moreover, $W$ admits homogeneous $A_G$-module automorphisms \eqref{homogenfamily} defined by the shifting $$(h\otimes_{\F}m)\phi_g:={hg}\otimes_{\F}m,$$
which satisfy the skew group algebra
condition \eqref{skewphi}, completing the proof.
\end{proof}
Theorems \ref{dadethm} and \ref{gendade} can be formulated using the endomorphism map.
\begin{theorem}\label{sgaobs}
Let \eqref{eq:algebragrading} be a $G$-graded algebra and let $M$ be any left $A_e$-module.
Then with the notations \eqref{sbmons} and \eqref{endmap}, $M$ is restricted from $A_G$ if
\begin{equation}\label{formomega}
\End([{A_G}\overline{\otimes}_{A_e}M]_{\e})\in\text{Sk}(G,\F).
\end{equation}
When \eqref{eq:algebragrading} is strongly graded, then $M$ is restricted from $A_G$ if and only if \eqref{formomega} holds.
\end{theorem}

\subsection{Extension of absolutely simple modules, proof of Theorem A}\label{pfA1}
We now focus on the extension criterion in the invariant absolutely simple case so as to prove Theorem A.
Let us sharpen the surjectivity property of the homomorphism \eqref{omega} in Theorem \ref{surjomega}.
Since $G$-graded $\F$-algebras in the same $Z^2(G,\F^*)$-orbit admit isomorphic base algebras (Lemma \ref{basetwist}), we can ask if a given left module over such algebra $A_e$
is restricted from \textit{any} of the members of this orbit.
Leaning on Theorem \ref{sgaobs}, we also obtain that absolutely simple $A_e$-modules which are invariant with respect to \eqref{eq:algebragrading} are always restricted,
either from $A_G$ itself or from at least one of its $Z^2(G,\F^*)$-twistings. Note that if $M$ is such a module then
\begin{eqnarray}\label{aseq}\begin{array}{cl}
\alpha^{-1}(\End_{A_G}^{\text{r}(G)}({A_G}\overline{\otimes}_{A_e}M))&\stackrel{G}\cong\alpha^{-1}(\End_{A_G}^{\text{r}(G)}(\alpha(\alpha^{-1}(A_G)\overline{\otimes}_{A_e}M^{})))\\
 & \stackrel{G}\cong\End_{\alpha^{-1}(A_G)}^{\text{r}(G)}(\alpha^{-1}(A_G)\overline{\otimes}_{A_e}M),\end{array}\end{eqnarray}
where the upper graded isomorphism follows from Corollary \ref{assoctwwist} and Lemma \ref{getsalong}(1), and the lower one follows from Corollary \ref{endtwist}.
Equation \eqref{aseq}, written in terms of the obstruction map \eqref{omega}, is
\begin{equation}\label{aseq1}[\alpha^{-1}]\cdot\omega_{G,\F}([{A_G}\overline{\otimes}_{A_e}M]_{\e})=\omega_{G,\F}([\alpha^{-1}(A_G)\overline{\otimes}_{A_e}M]_{\e})\in H^2(G,\F^*).\end{equation}
Equations \eqref{omegaiotaG} and \eqref{aseq1} yield Equation \eqref{tract} in the introduction. We also deduce
\begin{theorem}\label{obsinvas}
Let \eqref{eq:algebragrading} be a $G$-graded $\F$-algebra and let $[M]\in\text{Irr}^{}_{\F}(A_e)^{\mathcal{G}}$.
Then $M$ is restricted from $\alpha^{-1}(A_G)$ for every
\begin{equation}\label{formomeg}
\alpha\in\omega_{G,\F}([{A_G}\overline{\otimes}_{A_e}M]_{\e}).\end{equation}
When \eqref{eq:algebragrading} is strongly graded, then $M$ is restricted from $\alpha^{-1}(A_G)$ if and only if \eqref{formomeg} holds.
\end{theorem}
\begin{proof}
The right hand side of \eqref{aseq1} vanishes if and only if $[\alpha^{-1}(A_G)\overline{\otimes}_{A_e}M]_{\e}\in\ker(\omega_{G,\F})$,
whereas the left hand side of \eqref{aseq1} vanished if and only if \eqref{formomeg} holds.
Then we are done in light of Theorem \ref{sgaobs}.
\end{proof}
At this point, apply the map $\omega_{\mathcal{G}}$ \eqref{omegaiotaG} together with Theorem \ref{obsinvas} to establish Theorem A.

The following corollary of Theorem \ref{obsinvas} is known under the conditions that \eqref{eq:algebragrading} is strongly graded and the simple $A_e$-module $M$ is finite-dimensional over
the algebraically closed field $\F$ (see \cite[XII.1.22]{FellDoran2},\cite[Corollary 11.2.9]{Karpilovsky3},\cite[page 140]{NVO04}).
\begin{corollary}\label{cyclic}
Let \eqref{eq:algebragrading} be a $G$-graded $\F$-algebra, where $\F$ is algebraically closed, such that $H^2(G,\F^*)=1$. Then every simple $A_e$-module $M$ which is $\mathcal{G}$-invariant
with dim$_{\F}(M)< |\F|$ is restricted from an $A_G$-module.
\end{corollary}
\begin{proof}
By Theorem \ref{dixmier}, the above upper bound on the dimension of $M$ over $\F$ implies that $[M]\in$Irr$_{\F}(A_e)^{\mathcal{G}}$.
The result follows now from Theorem \ref{obsinvas} and from the triviality demand on the second cohomology group.
\end{proof}
The extension property in Corollary \ref{cyclic} holds, in particular, for any finite $G$ having cyclic $p$-Sylow subgroups for $p>2$,
and either a cyclic, semidihedral or a generalized quaternion group for $p=2$.
If char$(\F)=q>0$ then the property holds for finite groups $G$ having $p$-Sylow subgroups as above for every $p\neq q$, and any $q$-Sylow subgroup.
To verify that $H^2(G,\F^*)=1$ in these particular scenarios see \cite[page 58]{Br}, \cite[Proposition 1.5.5, Corollary 1.5.8 and Corollary 10.1.27]{K3}, and \cite[Corollary 9.90]{RotmanHA}.

Finally, let us formulate the discrete version of a central result in \cite{FellDoran2}.
In this theorem Fell and Doran generalize Mackey's celebrated bijection to saturated bundles.
Let \eqref{eq:algebragrading} be a {strongly} $G$-graded algebra, and let $[\alpha]=\omega_{\mathcal{G}}([M])$,
where $[M]\in\text{Irr}_{\F}(A_e)^{\mathcal{G}}$. By Theorem A the module $M$ is restricted from the twisted graded algebra $\alpha^{-1}(A_G)$.
Independently of the choice of the $\alpha^{-1}(A_G)$-module $\tilde{M}$ extending $M$ we have
\begin{Theorem B1}(Mackey \cite[Theorem 8.3]{M58}, Fell and Doran \cite[Theorem XII,4.28]{FellDoran2})
With the above notation, the map
\begin{equation}\label{correspond}
[U]\mapsto [U\otimes_{\F}\tilde{M}]\end{equation} is a bijection between the isomorphism types $[U]$ of
simple $\F^{\alpha}G$-modules, and the isomorphism types of simple $A_G$-modules lying above $M$.
\end{Theorem B1}

\noindent{\bf Acknowledgement.} The author thanks U. First for suggesting Theorem \ref{ungrHTr}, and
E. aljadeff, A. Amsalem, A. Braun, U. Onn, O. Schnabel, M. Roitman and S. Westreich for valuable discussions.


\begin{thebibliography}{10}

\bibitem{AGdR}
E.~Aljadeff, Y.~Ginosar, and {\'A}.~del R{\'{\i}}o.
\newblock Semisimple strongly graded rings.
\newblock {\em J. Algebra}, 256(1):111--125, 2002.

\bibitem{MR2488221}
Y.~A. Bahturin, M.~V. Zaicev, and S.~K. Sehgal.
\newblock Finite-dimensional simple graded algebras.
\newblock {\em Mat. Sb.}, 199(7):21--40, 2008.

\bibitem{Bourbaki72}
N.~Bourbaki.
\newblock {\em Elements of mathematics. {C}ommutative algebra}.
\newblock Hermann, Paris; Addison-Wesley Publishing Co., Reading, Mass., 1972.
\newblock Translated from the French.

\bibitem{Br}
K.~S. Brown.
\newblock {\em Cohomology of groups}, volume~87 of {\em Graduate Texts in
  Mathematics}.
\newblock Springer-Verlag, New York, 1994.
\newblock Corrected reprint of the 1982 original.

\bibitem{CG97}
N.~Chriss and V.~Ginzburg.
\newblock {\em Representation theory and complex geometry}.
\newblock Birkh\"{a}user Boston, Inc., Boston, MA, 1997.

\bibitem{Dade70}
E.~C. Dade.
\newblock Compounding {C}lifford's theory.
\newblock {\em Ann. of Math. (2)}, 91:236--290, 1970.

\bibitem{D80}
E.~C. Dade.
\newblock Group-graded rings and modules.
\newblock {\em Math. Z.}, 174(3):241--262, 1980.

\bibitem{Dade81}
E.~C. Dade.
\newblock Extending irreducible modules.
\newblock {\em J. Algebra}, 72(2):374--403, 1981.

\bibitem{das2008}
S.~D{\u{a}}sc{\u{a}}lescu.
\newblock Group gradings on diagonal algebras.
\newblock {\em Arch. Math. (Basel)}, 91(3):212--217, 2008.

\bibitem{DMI71}
F.~DeMeyer and E.~Ingraham.
\newblock {\em Separable algebras over commutative rings}.
\newblock Lecture Notes in Mathematics, Vol. 181. Springer-Verlag, Berlin-New
  York, 1971.

\bibitem{E18}
A.~Elduque.
\newblock Graded-simple algebras and cocycle twisted loop algebras.
\newblock {\em Proc. Amer. Math. Soc.}, 147(7):2821--2833, 2019.

\bibitem{EK13}
A.~Elduque and M.~Kochetov.
\newblock {\em Gradings on simple {L}ie algebras}, volume 189 of {\em
  Mathematical Surveys and Monographs}.
\newblock American Mathematical Society, Providence, RI; Atlantic Association
  for Research in the Mathematical Sciences (AARMS), Halifax, NS, 2013.

\bibitem{FellDoran1}
J.~M.~G. Fell and R.~S. Doran.
\newblock {\em Representations of {$^*$}-algebras, locally compact groups, and
  {B}anach {$^*$}-algebraic bundles. {V}ol. 1}, volume 125 of {\em Pure and
  Applied Mathematics}.
\newblock Academic Press, Inc., Boston, MA, 1988.
\newblock Basic representation theory of groups and algebras.

\bibitem{FellDoran2}
J.~M.~G. Fell and R.~S. Doran.
\newblock {\em Representations of {$^*$}-algebras, locally compact groups, and
  {B}anach {$^*$}-algebraic bundles. {V}ol. 2}, volume 126 of {\em Pure and
  Applied Mathematics}.
\newblock Academic Press, Inc., Boston, MA, 1988.
\newblock Banach $^*$-algebraic bundles, induced representations, and the
  generalized Mackey analysis.

\bibitem{GS16}
Y.~Ginosar and O.~Schnabel.
\newblock Groups of central-type, maximal connected gradings and intrinsic
  fundamental groups of complex semisimple algebras.
\newblock {\em Trans. Amer. Math. Soc.}, 371(9):6125--6168, 2019.

\bibitem{Jac}
N.~Jacobson.
\newblock Structure theory of simple rings without finiteness assumptions.
\newblock {\em Trans. Amer. Math. Soc.}, 57:228--245, 1945.

\bibitem{Jacobson64}
N.~Jacobson.
\newblock {\em Structure of rings}.
\newblock American Mathematical Society Colloquium Publications, Vol. 37.
  American Mathematical Society, Providence, R.I., revised edition, 1964.

\bibitem{K3}
G.~Karpilovsky.
\newblock {\em Group representations. {V}ol. 2}, volume 177 of {\em
  North-Holland Mathematics Studies}.
\newblock North-Holland Publishing Co., Amsterdam, 1993.

\bibitem{Karpilovsky3}
G.~Karpilovsky.
\newblock {\em Group representations. {V}ol. 3}, volume 180 of {\em
  North-Holland Mathematics Studies}.
\newblock North-Holland Publishing Co., Amsterdam, 1994.

\bibitem{Knus69}
M.-A. Knus.
\newblock Algebras graded by a group.
\newblock pages 117--133, 1969.

\bibitem{M58}
G.~W. Mackey.
\newblock Unitary representations of group extensions. {I}.
\newblock {\em Acta Math.}, 99:265--311, 1958.

\bibitem{ML71}
S.~MacLane.
\newblock {\em Categories for the working mathematician}.
\newblock Springer-Verlag, New York-Berlin, 1971.
\newblock Graduate Texts in Mathematics, Vol. 5.

\bibitem{NVO82}
C.~N{\u{a}}st{\u{a}}sescu and F.~Van~Oystaeyen.
\newblock {\em Graded ring theory}, volume~28 of {\em North-Holland
  Mathematical Library}.
\newblock North-Holland Publishing Co., Amsterdam-New York, 1982.

\bibitem{NVO04}
C.~N{\u{a}}st{\u{a}}sescu and F.~Van~Oystaeyen.
\newblock {\em Methods of graded rings}, volume 1836 of {\em Lecture Notes in
  Mathematics}.
\newblock Springer-Verlag, Berlin, 2004.

\bibitem{NOP18}
P.~Nystedt, J.~\"{O}inert, and H.~Pinedo.
\newblock Epsilon-strongly graded rings, separability and semisimplicity.
\newblock {\em J. Algebra}, 514:1--24, 2018.

\bibitem{RotmanHA}
J.~J. Rotman.
\newblock {\em An introduction to homological algebra}, volume~85 of {\em Pure
  and Applied Mathematics}.
\newblock Academic Press, Inc. [Harcourt Brace Jovanovich, Publishers], New
  York-London, 1979.

\end{thebibliography}
\end{document}